\documentclass[a4paper,11pt]{amsart}
\usepackage{amsmath,amsfonts,amssymb,amsthm}
\usepackage{bbm}
\usepackage[dvipsnames]{xcolor}
\usepackage[utf8]{inputenc}
\usepackage[T1]{fontenc}
\usepackage[autostyle]{csquotes}
\usepackage[margin = 1in]{geometry}
\usepackage{mathtools}
\usepackage{bbm}
\usepackage{enumitem}
\usepackage{textcomp}
\usepackage{cite}

\usepackage[hyphens]{url}
\usepackage{breakurl}
\usepackage[hidelinks, colorlinks, allcolors = black, breaklinks = true]{hyperref}
\makeatletter
\newcommand{\mylabel}[2]{#2\def\@currentlabel{#2}\label{#1}}
\makeatother

\makeatletter
\newcommand{\proofstep}[1]{%
  \par
  \addvspace{\medskipamount}
  \textit{#1\@addpunct{.}}\enspace\ignorespaces
}
\makeatother

\renewcommand{\d}{\ensuremath{\mathrm{d}}}

\newcommand{\EE}{\ensuremath{\mathbb{E}}}
\newcommand{\N}{\ensuremath{\mathbb{N}}}
\newcommand{\PP}{\ensuremath{\mathbb{P}}}
\newcommand{\R}{\ensuremath{\mathbb{R}}}
\newcommand{\TTT}{\ensuremath{\mathbb T}}
\newcommand{\Z}{\ensuremath{\mathbb{Z}}}
\def\R{\mathbb{R}}

\def\N{\mathbb{N}}

\newcommand{\TT}{\ensuremath{\mathcal T}}

\newcommand{\XX}{\ensuremath{\mathcal X}}

\newtheorem{theorem}{Theorem}[section]
\newtheorem{definition}[theorem]{Definition}
\newtheorem{proposition}[theorem]{Proposition}
\newtheorem{remark}[theorem]{Remark}
\newtheorem{lemma}[theorem]{Lemma}
\newtheorem{corollary}[theorem]{Corollary}

\numberwithin{equation}{section}
\newcommand{\cA}{{\mathcal A}}  
\newcommand{\cB}{{\mathcal B}}  
\newcommand{\cF}{{\mathcal F}}  
\newcommand{\cK}{{\mathcal K}}  
\newcommand{\cT}{{\mathcal T}}  
\newcommand{\cU}{{\mathcal U}}  
\newcommand{\cX}{{\mathcal X}}  

\DeclareMathOperator*{\esssup}{\mathrm{ess-sup}}

\newcommand{\dualpair}[2]{\ensuremath{\left\langle #1,#2 \right\rangle}}
\newcommand{\Dualpair}[2]{\ensuremath{\left\langle \kern-0.5ex \left\langle #1,#2 \right\rangle \kern-0.5ex \right\rangle}}

\newcommand{\floor}[1]{\ensuremath{\left\lfloor #1 \right\rfloor}}
\newcommand{\ceil}[1]{\ensuremath{\left\lceil #1 \right\rceil}}
\DeclareMathOperator{\ind}{\ensuremath{\mathbbm{1}}}
\DeclareMathOperator{\supp}{\ensuremath{\mathrm{supp}}}

\newcommand{\qvar}[2]{\ensuremath{\left\langle \kern-0.5ex \left\langle #1 \right\rangle \kern-0.5ex \right\rangle_{#2}}}

\newcommand{\eps}{\ensuremath{\varepsilon}}
\newcommand{\verti}[1]{\ensuremath{\left\lvert #1 \right\rvert}}
\newcommand{\vertii}[1]{\ensuremath{\left\lVert #1 \right\rVert}}
\newcommand{\vertiii}[1]{{\vertii{\kern-0.25ex\vertii{\kern-0.25ex\vertii{ #1
    }\kern-0.25ex}\kern-0.25ex}}}
\renewcommand{\d}{\ensuremath{{\rm d}}}

\begin{document}
%
\title[Martingale solutions to the stochastic thin-film equation]{The stochastic thin-film equation: existence of nonnegative martingale solutions}
\subjclass[2010]{}
\thanks{BG appreciates discussions with G\"unther Gr\"un. MVG acknowledges discussions with Christian Kuehn, is obliged to Jonas Sauer for answering a question regarding interpolation of operators, and thanks Heidelberg University, the Max Planck Institute for Mathematics in the Sciences, and the Technical University of Munich for their kind hospitality. The authors thank Konstantinos Dareiotis, Alexandra Neamtu, and the anonymous reviewer for critical readings of the manuscript. BG acknowledges support by the Max Planck Society through the Max Planck Research Group \emph{Stochastic partial differential equations} and by the Deutsche Forschungsgemeinschaft (DFG, German Research Foundation) through the CRC 1283 \emph{Taming uncertainty and profiting from randomness and low regularity in analysis, stochastics and their applications}. MVG's work was partially funded by the DFG under project \# 334362478.}
\date{\today}
\author{Benjamin Gess}
\address[Benjamin Gess]{Max Planck Institute for Mathematics in the Sciences, Inselstr.~22, 04103 Leipzig, Germany and Faculty of Mathematics, Bielefeld University, Universit\"atsstr.~25, 33615 Bielefeld, Germany}
\email{b.gess@mis.mpg.de}
\author{Manuel V. Gnann}
\address[Manuel~V.~Gnann]{Delft Institute of Applied Mathematics, Faculty of Electrical Engineering, Mathematics and Computer Science, Delft University of Technology, Van Mourik Broekmanweg 6, 2628 XE Delft, Netherlands}
\email{M.V.Gnann@tudelft.nl}
\begin{abstract}
We consider the stochastic thin-film equation with colored Gaussian Stratonovich noise in one space dimension and establish the existence of nonnegative weak (martingale) solutions. The construction is based on a Trotter-Kato-type decomposition into a deterministic and a stochastic evolution, which yields an easy to implement numerical algorithm. Compared to previous work, no interface potential has to be included, the initial data and the solution can have de-wetted regions of positive measure, and the Trotter-Kato scheme allows for a simpler proof of existence than in case of It\^o noise.
\end{abstract}
\maketitle
\tableofcontents
%

\section{Introduction}

\subsection{Setting}
Consider the stochastic thin-film equation with quadratic mobility
\begin{equation}\label{stfe}
\d u = - \partial_x \left(u^2 \partial_x^3 u\right) \d t + \partial_x \left(u \circ \d W\right) \quad \mbox{for} \quad (t,x) \in [0,T) \times \TTT_L,
\end{equation}
where $T, L \in (0,\infty)$ and $\TTT_L$ denotes the torus of the interval $[0,L]$. We will always assume periodic boundary conditions
\[
\partial_x^j u(\cdot,0) = \partial_x^j u(\cdot,L) \quad \mbox{for} \quad j \in \{0,1,2,3\}
\]
without further mentioning it. Moreover, suppose that periodic nonnegative initial data $u_0 \colon \TTT_L \to [0,\infty)$ are given, satisfying certain regularity properties that we will specify below. Equation~\eqref{stfe} describes the evolution of the height $u$ of a two-dimensional viscous thin film as a function of time $t$ and lateral position $x$ influenced by thermal noise $W$ and assuming Navier slip at the substrate. The noise $W$ is assumed to be colored Gaussian and the symbol $u \circ \d W$ denotes Stratonovich noise. Equation \eqref{stfe} can be regarded as an approximate model to the full stochastic thin-film equation
\begin{equation}\label{stfe_full}
\d u = - \partial_x \left(\left(\ell_\mathrm{s} u^2 + u^3\right) \partial_x^3 u\right) \d t + \partial_x \left(\sqrt{\ell_\mathrm{s} u^2 + u^3} \circ \d W\right) \quad \mbox{for} \quad (t,x) \in [0,T) \times \TTT_L,
\end{equation}
where the constant $\ell_\mathrm{s} > 0$ denotes the slip length. Hence, \eqref{stfe} is an approximation of \eqref{stfe_full} for film heights $u$ that are much smaller than the slip length $\ell_\mathrm{s}$.

\medskip

In this paper we prove the existence of nonnegative martingale solutions to  \eqref{stfe} (cf.~Theorem~\ref{th:main} and Remark~\ref{rm:main} below) for initial data $u_0 \in H^1(\TTT_L)$ such that $u_0 \ge 0$. The construction is based on the following Trotter-Kato scheme
\begin{subequations}\label{scheme_intro}
\begin{equation}\label{deterministic_intro}
\partial_t v_N = - \partial_x \left(v_N^2 \partial_x^3 v_N\right) \quad \mbox{for} \quad (t,x) \in [(j-1) \delta,j \delta) \times \TTT_L
\end{equation}
and
\begin{equation}\label{stochastic_intro}
\d w_N = \partial_x \left(w_N \circ \d W\right) \quad \mbox{for} \quad (t,x) \in [(j-1) \delta,j \delta) \times \TTT_L,
\end{equation}
\end{subequations}
where $\delta := \frac{T}{N+1}$, $j \in \{1,\ldots,N+1\}$, and $N \in \N_0$, glueing together according to $v_N(0,\cdot) := u_0$, $w_N((j-1)\delta,\cdot) := \lim_{t \nearrow j \delta} v_N(t,\cdot)$ for $j \in \{1,\ldots,N\}$, and $v_N(j\delta,\cdot) :=
\lim_{t \nearrow j \delta} w_N(t,\cdot)$ for $j \in \{1,\ldots,N\}$, and taking the limit as $N \to \infty$. Before giving mathematical details, we will next discuss the choice of Stratonovich instead of It\^o noise in \eqref{stfe}, \eqref{stfe_full}, and \eqref{stochastic_intro}.

\subsection{It\^o versus Stratonovich formulation\label{sec:ito_strat}}
Two versions of the stochastic thin-film equation have been proposed independently. The first due to Davidovitch, Moro, and Stone \cite{DavidovitchMoroStone2005} is in line with the formulation \eqref{stfe} and has been applied to describe the enhanced spreading of droplets. The other ground-laying work by Gr\"un, Mecke, and Rauscher \cite{GruenMeckeRauscher2006} additionally takes an interface potential between fluid and substrate into account that prevents $u$ from becoming negative. The study in \cite{GruenMeckeRauscher2006} focuses on coarsening and de-wetting phenomena.

\medskip

The first rigorous construction of nonnegative martingale solutions to the stochastic thin-film equation with It\^o noise and additional interface potential, as derived in \cite{GruenMeckeRauscher2006}, has been recently given by Fischer and Gr\"un in \cite{FischerGruen2018}. A generalization to more general mobilities at the expense of introducing suitable nonlocal source terms has subsequently been introduced by Cornalba in \cite{Cornalba2018}. The inclusion of an additional interface potential is crucially used in these works in order to obtain suitable a-priori estimates.

\medskip

The starting point of the (informal) derivation of the stochastic thin-film equation in \cite{GruenMeckeRauscher2006} is the transport equation (see \cite[p.~1265, Eq.~(6)]{GruenMeckeRauscher2006})
\begin{equation}\label{transport_filmheight}
\partial_t u = v_y - v_x \, \partial_x u,
\end{equation}
where $v_x$ and $v_y$ denote the horizontal and vertical components of the fluid velocity, respectively. Since the fluid velocity is modelled as a solution to the \textit{stochastic} incompressible Navier-Stokes equation, it should be understood as a stochastic process. Therefore, the product in \eqref{transport_filmheight} needs to be understood in the sense of a stochastic integral. We next recall the (informal) derivation of \eqref{transport_filmheight} in order to justify the choice of stochastic integration (It\^o versus Stratonovich). Equation~\eqref{transport_filmheight} can be derived by considering the movement of fluid particles at the liquid-air interface with trajectories parametrized by $\left(x(t),u\left(t,x(t)\right)\right)$, where $x(t)$ denotes the lateral position as a function of time $t$. The change of the height of the fluid is given by the vertical component $v_y$ of the fluid velocity, that is, 
\begin{equation}\label{transport_filmheight_2}
\frac{\d}{\d t} u(t,x(t)) = v_y(x(t),u(t,x(t))).
\end{equation}
The lateral position of a fluid particle changes according to the horizontal component $v_x$ of the fluid velocity
\[
\dot x(t) = v_x(x(t),u(t,x(t))),
\]
which again should be understood as a stochastic equation. Informally, It\^o's formula dictates 
\begin{equation}
\label{transport_filmheight_3} \begin{split}
  \frac{\d}{\d t} u(t,x(t)) 
  &= (\partial_tu)(t,x(t)) + (\partial_x u)(t,x(t))\circ \dot x(t) \\
  &= (\partial_tu)(t,x(t)) + (\partial_x u)(t,x(t))\circ v_x(x(t),u(t,x(t))),
\end{split}
\end{equation}
which together with \eqref{transport_filmheight_2} yields \eqref{transport_filmheight}. If we were to use the It\^o interpretation in \eqref{transport_filmheight_3}, an appropriate It\^o correction term would appear. This indicates that the derivation of the stochastic thin-film equation in \cite{GruenMeckeRauscher2006} relies on Stratonovich calculus and that the resulting model, as well as the one of \cite{DavidovitchMoroStone2005}, is naturally formulated with Stratonovich noise. In \cite[Appendix~C]{GruenMeckeRauscher2006} it was then claimed that the specific choice of the stochastic calculus (It\^o versus Stratonovich) is immaterial, at least in the case of space-time white noise.

\medskip

In the present work we choose to consider the Stratonovich formulation of the thin-film equation due to two points: First, we prove that in Stratonovich formulation the construction of nonnegative martingale solutions is possible \emph{without} an additional interface potential and allowing for \emph{touch down} of solutions, thus relaxing the assumptions of \cite{FischerGruen2018,Cornalba2018}. Second, we show that the Stratonovich formulation allows for a simpler construction of solutions via a Trotter-Kato scheme. Notably, this scheme requires Stratonovich noise as only then the transport equation \eqref{stochastic_intro} is well-posed.

\subsection{Weak formulation and main result\label{sec:weak_form}}
Let
\begin{equation}\label{def_noise}
W(t,x) := \sum_{k \in \Z} \lambda_k \psi_k(x) \beta^k(t),
\end{equation}
where $(\lambda_k)_{k \in \Z}$ are real and nonnegative, define the family $(\psi_k)_{k \in \Z}$ through
\begin{equation}\label{basis}
\psi_k(x)
:= \sqrt{\frac{2}{L \left(1 + \left(\frac{2 \pi k}{L}\right)^2 + \left(\frac{2 \pi k}{L}\right)^4\right)}}
\begin{cases} \cos\left(\frac{2 \pi k}{L} x\right) & \mbox{for} \quad k > 0 \quad \mbox{and} \quad x \in [0,L], \\
\frac{1}{\sqrt 2} & \mbox{for} \quad k = 0 \quad \mbox{and} \quad x \in [0,L], \\
\sin\left(\frac{2 \pi k}{L} x\right) & \mbox{for} \quad k < 0 \quad \mbox{and} \quad x \in [0,L],
\end{cases}
\end{equation}
being an orthonormal basis of $H^2(\TTT_L)$ of eigenfunctions of the periodic Laplacian, and let $(\beta^k)_{k \in \Z}$ be a family of mutually independent standard real-valued $(\cF_t)$-Wiener processes on a complete filtered probability space $\left(\Omega,\cF,\left(\cF_t\right)_{t \in [0,T)},\PP\right)$, with a complete and right-continuous filtration $\left(\cF_t\right)_{t \in [0,T)}$. From \eqref{basis} it follows in particular
\begin{subequations}\label{der_psi_k}
\begin{equation}\label{der_psi_k_0}
\partial_x \psi_k = \frac{2 \pi k}{L} \psi_{-k} \quad \mbox{for} \quad k \in \Z,
\end{equation}
so that
\begin{equation}\label{der_psi_k_1}
\partial_x^2 \psi_k = - \frac{4 \pi^2 k^2}{L^2} \psi_k, \quad \partial_x^3 \psi_k = - \frac{8 \pi^3 k^3}{L^3} \psi_{-k}, \quad \partial_x^4 \psi_k = \frac{16 \pi^4 k^4}{L^4} \psi_k \quad \mbox{for} \quad k \in \Z.
\end{equation}
\end{subequations}
We will further assume the decay condition
\begin{equation}\label{cond_lambda}
\sum_{k \in \Z} \lambda_k^2 < \infty.
\end{equation}
This ensures that $W$ takes values in $H^2(\TTT_L)$. Condition~\eqref{cond_lambda} is the same as in \cite[p.~417, condition~(H4)]{FischerGruen2018}, taking into account that Fischer and Gr\"un choose an orthonormal basis of $L^2(\TTT_L)$. 

\medskip

Equation~\eqref{stfe} with noise $W$ as in \eqref{def_noise} may be rewritten using It\^o calculus as (see~\cite[\S3]{DareiotisGess2017} for an analogous case)
\begin{equation}\label{stfe_ito}
\d u = \partial_x \left(- \left(u^2 \partial_x^3 u\right) + \frac 1 2 \sum_{k \in \Z} \lambda_k^2 \psi_k \, \partial_x (\psi_k u)\right) \d t + \partial_x \left(\sum_{k \in \Z} \lambda_k \psi_k u \, \d\beta^k\right)
\end{equation}
and its weak formulation is given by
\begin{equation}\label{weak_ito}
\begin{aligned}
\d \left(u,\varphi\right)_2 &= \left(\int_{\{u(t,\cdot) > 0\}} u^2 (\partial_x^3 u) \, (\partial_x \varphi) \, \d x - \frac 1 2 \sum_{k \in \Z} \lambda_k^2 \left(\left(\psi_k \partial_x \left(\psi_k u\right)\right),\partial_x\varphi\right)_2\right) \d t \\
&\phantom{=} - \sum_{k \in \Z} \lambda_k \left(\psi_k u, \partial_x \varphi\right)_2 \d \beta^k,
\end{aligned}
\end{equation}
$\PP$-almost surely, for any $\varphi \in C^\infty(\TTT_L)$, where $\left(v,w\right)_2 := \int_0^L v(x)w(x) \, \d x$ for $v, w \in L^2(\TTT_L)$ denotes the inner product in $L^2(\TTT_L)$. Note that in the weak formulation, we only require the third derivative $\partial_x^3 u$ to exist on the positivity set $\{u > 0\}$ (cf.~Definition~\ref{def:martingale_solution}).

\medskip

We use the following notion of solutions:
\begin{definition}\label{def:martingale_solution}
Let $u_0 \in H^1(\TTT_L)$ be nonnegative. A triple, consisting of a filtered probability space
$
\left(\tilde\Omega,\tilde\cF,\begin{pmatrix} \tilde\cF_t \end{pmatrix}_{t \in [0,T)},\tilde\PP\right),
$
where $\left(\tilde\cF_t\right)_{t \in [0,T)}$ is a complete and right-continuous filtration, an $(\tilde \cF_t)$-adapted bounded continuous $H^1_\mathrm{w}(\TTT_L)$-valued process $\tilde u$ on $[0,T)$ such that the distributional derivative $\partial_x^3 \tilde u$ is $(\tilde \cF_t)$-adapted with $\partial_x^3 \tilde u \in L^2(\{\tilde u > r\})$ for any $r > 0$ and $\tilde u^2 (\partial_x^3 \tilde u) \in L^2(\{\tilde u > 0\})$, $\tilde\PP$-almost surely, as well as mutually independent standard real-valued $(\tilde \cF_t)$-Wiener processes $\tilde \beta^k$, is called a martingale solution of the stochastic thin-film equation \eqref{stfe} if its weak formulation
\begin{align*}
\left(\tilde u(t,\cdot),\varphi\right)_2 - \left(u_0,\varphi\right)_2 &= \int_0^t \int_{\left\{\tilde u(t',\cdot) > 0\right\}} \tilde u^2 \, (\partial_x^3 \tilde u) \, (\partial_x\varphi) \, \d x \, \d t' \\
&\phantom{=} - \frac 1 2 \sum_{k \in \Z} \lambda_k^2 \int_0^t \left(\psi_k \partial_x \left(\psi_k \tilde u(t',\cdot)\right),\partial_x \varphi\right)_2 \d t' \\
&\phantom{=} - \sum_{k \in \Z} \lambda_k \int_0^t \left(\psi_k \tilde u(t',\cdot),\partial_x \varphi\right)_2 \d\tilde\beta^k(t')
\end{align*}
is satisfied for every $\varphi \in C^\infty(\TTT_L)$ and $t \in [0,T)$, $\tilde\PP$-almost surely.
\end{definition}

The main result of this work is
\begin{theorem}[martingale solutions to the stochastic thin-film equation]\label{th:main}
Suppose that $u_0 \in H^1(\TTT_L)$ such that $u_0 \ge 0$. Then, in the sense of Definition~\ref{def:martingale_solution}, there exists a martingale solution
\[
\left([0,1],\tilde\cF,\begin{pmatrix} \tilde\cF_t \end{pmatrix}_{t \in [0,T)},\tilde\PP\right), \quad \tilde u, \quad \mbox{and} \quad \left(\tilde\beta_k\right)_{k \in \Z}
\]
to the stochastic thin-film equation \eqref{stfe} for which $\tilde u \ge 0$, $\PP$-almost surely, and for which the a-priori estimate
\[
\tilde\EE \esssup_{t \in [0,T)} \vertii{\tilde u(t,\cdot)}_{1,2}^p \le C \vertii{u_0}_{1,2}^p
\]
is satisfied for any $p \in [2,\infty)$, where $C < \infty$ is independent of $\tilde u$ and $u_0$.
\end{theorem}
The proof of the above result is given in Section \ref{sec:sol_stfe} below.

\medskip

We emphasize once more that compared to the previous works \cite{FischerGruen2018,Cornalba2018} we do not require an interface potential and that the occurrence of de-wetted regions with positive measure $\tilde\PP\verti{\left\{\tilde u(t,\cdot) = 0\right\}} > 0$ is included. This is due to the fact that the arguments of the present work only rely on controlling the energy $\frac 1 2 \int_0^L (\partial_x \tilde u(t,x))^2 \, \d x$ and not on controlling the entropy $\int_0^L \verti{\ln \tilde u(t,x)} \d x$ as in \cite{FischerGruen2018,Cornalba2018}.

\begin{remark}\label{rm:main}
Note that Theorem~\ref{th:main} easily translates to the case of random initial data $u_0$ satisfying
\[
u_0 \in L^q\left(\Omega,\cF_0,\PP;H^1(\TTT_L)\right) \quad \mbox{with} \quad u_0 \ge 0, \quad \mbox{$\PP$-almost surely},
\]
where $q \ge p$ is sufficiently large.
\end{remark}
%

\subsection{Decomposition of the dynamics}
The idea of the construction is to split the dynamics of \eqref{weak_ito} into a deterministic evolution and a stochastic evolution; a Trotter-Kato-type decomposition that has been utilized in many other solution approaches for SPDEs, too. See for instance the works of Bensoussan, Glowinsky, and R\u{a}\c{s}canu \cite{BensoussanGlowinskiRascanu1990} and Gyöngy and Krylov \cite{GK03} on the Zaka\"i equation or Govindan \cite{Govindan2015} for a mild-solution approach to semilinear stochastic evolution equations.

\medskip

To begin with, we split the time interval $[0,T)$ into pieces of length $\delta := \frac{T}{N+1}$, where $N \in \N_0$. Then we define
\begin{enumerate}
\item[\mylabel{item:deterministic}{(D)}] \emph{Deterministic dynamics}: On $[(j-1) \delta, j\delta)$ the function $v_N$ satisfies the evolution
\begin{subequations}\label{scheme}
\begin{equation}\label{deterministic}
\left(v_N(t,\cdot),\varphi\right)_2 - \left(v_N\left((j-1)\delta,\cdot\right),\varphi\right)_2 = \int_{(j-1)\delta}^t \int_{\{v_N(t',\cdot) > 0\}} v_N^2 \, (\partial_x^3 v_N) \, (\partial_x\varphi) \, \d x \, \d t'
\end{equation}
for $t \in [(j-1)\delta, j \delta)$, $\PP$-almost surely, where $j = 1,\ldots,N+1$ and $\varphi \in C^\infty\left(\TTT_L\right)$.
\item[\mylabel{item:stochastic}{(S)}] \emph{Stochastic dynamics}: On $[(j-1) \delta,j\delta)$ the function $w_N$ satisfies the evolution
\begin{equation}\label{stochastic}
\begin{aligned}
\left(w_N(t,\cdot),\varphi\right)_2 - \left(w_N\left((j-1)\delta,\cdot\right),\varphi\right)_2 &= - \frac 1 2 \sum_{k \in \Z} \lambda_k^2 \int_{(j-1) \delta}^t \left(\psi_k \partial_x (\psi_k w_N(t',\cdot)), \partial_x \varphi\right)_2 \d t' \\
&\phantom{=} - \sum_{k \in \Z} \lambda_k \int_{(j-1)\delta}^t \left(\psi_k w_N(t',\cdot), \partial_x\varphi\right)_2 \d\beta^k(t')
\end{aligned}
\end{equation}
for $t \in [(j-1)\delta, j \delta)$, $\PP$-almost surely, where $j = 1,\ldots,N+1$ and $\varphi \in C^\infty\left(\TTT_L\right)$.
\item[\mylabel{item:deterministic_stochastic}{(DS)}] \emph{Connecting deterministic and stochastic dynamics}: We use
\begin{equation}\label{connection}
v_N(0,\cdot) := u_0, \quad v_N\left(j\delta,\cdot\right) := \lim_{t \nearrow j \delta} w_N\left(t,\cdot\right), \quad \mbox{and} \quad w_N\left((j-1)\delta,\cdot\right) := \lim_{t \nearrow j \delta} v_N\left(t,\cdot\right),
\end{equation}
\end{subequations}
$\PP$-almost surely, where $j \in \{1,\ldots,N\}$.
\end{enumerate}
Notice that \eqref{deterministic} is the weak formulation of \eqref{deterministic_intro}, while \eqref{stochastic} is the weak formulation of \eqref{stochastic_intro}, i.e., with noise $W$ as in \eqref{def_noise},
\begin{equation}\label{stochastic_2}
\d w_N =  \frac 1 2 \sum_{k \in \Z} \lambda_k^2 \partial_x\left(\psi_k\partial_x(\psi_k w_N)\right) \d t + \sum_{k \in \Z} \lambda_k \partial_x (\psi_k w_N) \, \d\beta^k \quad \mbox{for} \quad t \in [(j-1)\delta,j\delta)
\end{equation}
and $j \in \{1,\ldots,N+1\}$. Since \eqref{deterministic_intro} and \eqref{stochastic_2} are in divergence form, they both automatically conserve mass $\int_0^L v_N(t,x) \, \d x$ or $\int_0^L w_N(t,x) \, \d x$, respectively.

\medskip

Note that the dynamics in \ref{item:deterministic} are purely deterministic, while the dynamics in \ref{item:stochastic} are purely stochastic, with \ref{item:deterministic_stochastic} connecting them. In this work we show that solutions to \ref{item:deterministic} and \ref{item:stochastic} exist and that as $N \to \infty$, the scheme \ref{item:deterministic}--\ref{item:stochastic}--\ref{item:deterministic_stochastic} converges to a martingale solution to \eqref{stfe}. 

\medskip

Note that the deterministic dynamics \ref{item:deterministic} are determined by the deterministic thin-film equation \eqref{deterministic_intro}, for which an existence theory of weak solutions due to Bernis and Friedman \cite{BernisFriedman1990} is available. This theory has been subsequently upgraded to entropy-weak solutions by Beretta, Bertsch, and Dal Passo in \cite{BerettaBertschDalPasso1995} and Bertozzi and Pugh in \cite{BertozziPugh1996} and to higher dimensions by Dal Passo, Garcke, and Gr\"un in \cite{DalPassoGarckeGruen1998} and by Gr\"un in \cite{Gruen2004}. The stochastic dynamics \ref{item:stochastic}, on the other hand, are determined by a transport equation, to which we will apply a viscous regularization and the variational approach in order to construct solutions. While the existence of variational solutions is well-known (e.g.~Krylov, Rozovski\u{\i} \cite{KR82} and Gerencs\'er, Gy\"{o}ngy, Krylov \cite{GGK15}), we recall some details on the proof in order to keep track on the dependency of the constants on the time step, as needed for the proof of convergence of the Trotter-Kato scheme.  By construction, the scheme will preserve nonnegativity of solutions as long as we start with nonnegative and sufficiently regular initial data $u_0$, since this is known to be true for weak solutions to the deterministic thin-film equation \eqref{deterministic_intro} (cf.~\cite[Theorem~4.1]{BernisFriedman1990}), while \eqref{stochastic_2} is a transport equation for which this assertion is not difficult to prove (cf.~Proposition~\ref{prop:stochastic_entropy} below). Note, however, that the additional drift term in \eqref{stochastic} is crucial in order to allow for the construction of solutions and that the dynamics \ref{item:stochastic} without this additional drift term are in fact not well-defined.

\subsection{Outline}
In \S\ref{sec:deterministic}--\ref{sec:time_reg_approx} we prove that nonnegative solutions to the splitting scheme \ref{item:deterministic}--\ref{item:stochastic}--\ref{item:deterministic_stochastic} exist such that certain bounds and regularity properties are satisfied. More precisely, in \S\ref{sec:deterministic} we derive that solutions to the deterministic thin-film dynamics \ref{item:deterministic} (cf.~Theorem~\ref{th:tfe_weak} below and \cite{BernisFriedman1990,BerettaBertschDalPasso1995,BertozziPugh1996}) satisfy suitable bounds on the surface energy $\int_0^L (\partial_x v_N)^2 \, \d x$ (cf.~Corollary~\ref{cor:energy_tfe} below). In \S\ref{sec:stochastic} and Appendix~\ref{sec:viscous} we move on to the stochastic dynamics \ref{item:stochastic} and prove that solutions exist by the vanishing viscosity method employing the variational approach (cf.~Proposition~\ref{prop:var} and Proposition~\ref{prop:weak} below). The solution satisfies a bound on the expected surface energy $\EE \int_0^L (\partial_x w_N)^2 \, \d x$ with suitable constants and we further prove that $w_N$ is, $\PP$-almost surely, nonnegative (cf.~Proposition~\ref{prop:stochastic_entropy} below). In \S\ref{sec:time_reg_approx} we summarize the results for the concatenated solution $u_N$ fulfilling \ref{item:deterministic}--\ref{item:stochastic}--\ref{item:deterministic_stochastic} (cf.~Proposition~\ref{prop:reg_vn_wn} below) and prove additional regularity in time by cross interpolation (cf.~Proposition~\ref{prop:reg_time_un} and Appendix~\ref{sec:interpolation} below).

\medskip

In \S\ref{sec:sol_stfe} we construct solutions to the original equation \eqref{stfe}. The compactness argument in \S\ref{sec:tight} is based on a generalization of Skorokhod's representation theorem due to Jakubowski (cf.~Theorem~\ref{th:jakubowski} below and \cite[Theorem~2]{Jakubowski1997}) by proving tightness in suitable spaces (cf.~Proposition~\ref{prop:point_convergence} below). The rest of \S\ref{sec:tight} is devoted to the identification of the limits of the convergent subsequences (cf.~Propositions~\ref{prop:point_convergence}, \ref{prop:brownian}, and \ref{prop:identify_limit} below). In \S\ref{sec:limit} we subsequently recover the stochastic thin-film equation \eqref{stfe} in the sense of Definition~\ref{def:martingale_solution}, leading to the main result, formulated in Theorem~\ref{th:main}, in which nonnegative martingale solutions are obtained.

\medskip

The paper is completed in \S\ref{sec:conclude} with concluding remarks on possible future directions.

\subsection{Notation and conventions}
%
\subsubsection*{Sets}
We write $\N := \{1,2,3,\ldots\}$ for positive integers and $\N_0 := \N \cup \{0\}$. The set $\TTT_L$ denotes the torus of the interval $[0,L]$, where $L > 0$. For sets $X$ and $K$ we write $K \Subset X$ if $K$ is a subset of $X$ ($K \subseteq X$) and $K$ is compact. We write $\ind_A$ for the indicator function of a set $A \subseteq X$.

\subsubsection*{Lebesgue spaces}
We denote by $L^p(\Omega,\cA,\mu;X)$ the Lebesgue spaces with $p \in [1,\infty]$ of functions $\Omega \to X$, where $\Omega$ is a set, $\cA$ is a $\sigma$-algebra on $\Omega$, $\mu \colon \cA \to [0,\infty]$ is a measure, and $X$ denotes a Banach space. In case that $\cA$ denotes the Borel-$\sigma$-algebra and $\mu$ is the Lebesgue measure, we simply write $L^p(\Omega;X)$, and if $X = \R$, we write $L^p(\Omega)$. We write $\left(u,v\right)_2 := \int_0^L u \, v \, \d x$ and $\vertii{u}_2 := \sqrt{(u,u)_2}$ for the inner product and norm, where $u, v \in L^2(\TTT_L)$.

\subsubsection*{H\"older spaces and spaces of bounded continuous functions}
For $\Omega \subseteq \R^d$ with $\partial\Omega \in C^\infty$, the space $C^{k+\alpha}(\Omega;X)$ is the space of $k$-times differentiable functions $\Omega \to X$, where $k \in \N_0$, whose $k$-th derivatives are H\"older continuous with exponent $\alpha \in (0,1)$ on compact subsets of $\Omega$. For $k \in \N$ we write $C^{k-}(\Omega;X)$ for the space of $(k-1)$-times differentiable functions $\Omega \to X$ whose $(k-1)$-th derivatives are Lipschitz continuous. We write $BC^0(\Omega;X)$ for the space of bounded continuous functions $\Omega \to X$.

\subsubsection*{Sobolev(-Slobodeckij) spaces}
Suppose that $\Omega \subseteq \R^d$ with $\partial\Omega \in C^\infty$, $s \in [0,\infty)$, $p \in [1,\infty]$, and let $X$ be a Banach space. For a locally integrable function $u \colon \Omega \to X$ we define
\[
\vertii{u}_{W^{s,p}(\Omega;X)} := \left(\sum_{\alpha \in \N_0^d, \, 0 \le \verti{\alpha} \le s} \int_\Omega \vertii{\partial^\alpha u}_X^p \, \d x\right)^{\frac 1 p} \quad \mbox{for} \quad p \in [1,\infty) \quad \mbox{and} \quad s \in \N_0,
\]
and $\vertii{u}_{W^{s,p}(\Omega;X)} := \vertii{u}_{W^{\floor{s},p}(\Omega;X)} + \left[u\right]_{W^{s,p}(\Omega;X)}$ for $s \in (0,\infty) \setminus \N$, where
\[
\left[u\right]_{W^{s,p}(\Omega;X)} := \left(\sum_{\alpha \in \N_0^d, \, \verti{\alpha} = \floor{s}} \int_\Omega \int_\Omega \frac{\vertii{\partial^\alpha u(x) - \partial^\alpha u(y)}_X^p}{\verti{x-y}^{(s-\floor{s})p+d}} \, \d x \, \d y\right)^{\frac 1 p} \quad \mbox{for} \quad p \in [1,\infty),
\]
with the usual modifications for $p = \infty$. Then, the Sobolev-Slobodeckij space $W^{s,p}(\Omega;X)$ is the space of all locally integrable $u \colon \Omega \to X$ such that $\vertii{u}_{W^{s,p}(\Omega;X)} < \infty$. If $X = \R$, we simply write $W^{s,p}(\Omega)$. The space $W^{s,p}_0(\Omega;X)$ is defined as the closure of $C^\infty_\mathrm{c}\left(\mathring{\Omega};X\right)$ in $W^{s,p}(\Omega;X)$. The space $W^{s,p}(\Omega;X)$ for $s < 0$ and $1 \le p < \infty$ is defined as the dual of $W^{-s,p'}_0(\Omega;X')$, where $\frac{1}{p} + \frac{1}{p'} = 1$ and $X$ is reflexive.

\subsubsection*{Besov spaces}
Assuming that $\Omega \subseteq \R^d$ with $\partial\Omega \in C^\infty$, $s \in \R\setminus \Z$, and $p,q \in [1,\infty)$, we introduce the Besov space $B^{s,p}_q(\Omega;X) := \left(W^{\floor{s},p}(\Omega;X),W^{\ceil{s},p}(\Omega;X)\right)_{\kappa,q}$, where $\kappa = s - \floor{s}$ and $(\cdot,\cdot)_{\kappa,q}$ denotes the real interpolation functor. For $s \in \Z$ we define the Besov space $B^{s,p}_q(\Omega;X) := \left(W^{s-1,p}(\Omega;X),W^{s+1,p}(\Omega;X)\right)_{\frac 1 2,q}$. For an introduction to complex and real interpolation of operators, we refer to \cite[\S3, \S4]{BergLofstrom1976} or \cite[\S1]{Triebel1978}, while Besov spaces with values in a Banach space are discussed for instance in \cite{Amann2000} and \cite[Chapter~VII, \S2]{Amann2019}.

\subsubsection*{Periodic Bessel-potential spaces}
For $s \in [0,\infty)$ and $p \in (1,\infty)$ we define $H^{s,p}(\TTT_L)$ as the space of locally integrable $u \colon \TTT_L \to \R$ such that $\vertii{u}_{s,p} < \infty$, where for $p \ne 2$ we use
\[
\vertii{u}_{s,p} := \left(\sum_{k \in \Z} (1+k^2)^{\frac{s p}{2}} \verti{\hat u(k)}^p\right)^{\frac 1 p}, \quad \mbox{where} \quad \hat u(k) := \frac{1}{\sqrt L} \int_0^L e^{\frac{2 \pi i k}{L} x} \, u(x) \, \d x
\]
and for $p = 2$ we write $H^s(\TTT_L) := H^{s,2}(\TTT_L)$ and define the inner products and norms by
\begin{align*}
(u,v)_{s,2} &:= \sum_{j = 0}^s \int_0^L (\partial_x^j u) \, (\partial_x^j v) \, \d x \quad \mbox{for} \quad s \in \N_0, \\
(u.v)_{s,2} &:= \sum_{k \in \Z} (1+k^2)^s \, (\hat u(k))^* \, \hat v(k) \quad \mbox{for} \quad s \in (0,\infty) \setminus \N, \quad \hat u(k) := \frac{1}{\sqrt L} \int_0^L e^{\frac{2 \pi i k}{L} x} \, u(x) \, \d x,
\end{align*}
and $\vertii{u}_{s,2} := \sqrt{(u,u)_{s,2}}$, where $u, v \in H^s(\TTT_L)$. We write $\dot H^1(\TTT_L)$ for the homogeneous Sobolev space of all locally integrable $u \colon \Omega \to \R$ with norm $\vertii{\partial_x u}_2 < \infty$, where we identify $u, v \in \dot H^1(\TTT_L)$, if $u-v$ is a constant. The space $H^{-s,p}(\TTT_L)$ is defined as the dual of $H^{s,p'}(\TTT_L)$, where $\frac{1}{p} + \frac{1}{p'} = 1$. We write $\dualpair \cdot \cdot$ or $\Dualpair \cdot \cdot$ for the dual pairing of $H^{-1}(\TTT_L)$ with $H^1(\TTT_L)$ in $L^2(\TTT_L)$ or $L^2(\TTT_L)$ with $H^2(\TTT_L)$ in $H^1(\TTT_L)$, respectively. We write $H^1_\mathrm{w}(\TTT_L)$ for the space $H^1(\TTT_L)$ endowed with the weak topology induced by $\vertii{\cdot}_{1,2}$.

\subsubsection*{Periodic Besov spaces}
For $s \in \R \setminus \Z$, $p \in (1,\infty)$, and $q \in [1,\infty)$, we define periodic Besov spaces by real interpolation as $B^{s,p}_q(\TTT_L) := \left(H^{\floor{s},p}(\TTT_L),H^{\ceil{s},p}(\TTT_L)\right)_{\kappa,q}$, where $\kappa := s - \floor{s}$. For $s \in \Z$ we set $B^{s,p}_q(\TTT_L) := \left(H^{s+1,p}(\TTT_L),H^{s-1,p}(\TTT_L)\right)_{\frac 1 2,q}$. Periodic Besov spaces are investigated in detail in \cite[\S3]{SchmeisserTriebel1987}.

\subsubsection*{Hilbert-Schmidt operators}
We denote by $L_2(U;H)$ the space of Hilbert-Schmidt operators $U \to H$, where $U$ and $H$ are separable Hilbert spaces, i.e., the space of bounded linear operators $B \colon U \to H$ with finite norm $\vertii{B}_{L_2(U;H)} := \sqrt{\sum_{k \in \N} \vertii{B e_k}_H^2}$, where $(e_k)_{k \in \N}$ denotes any orthonormal basis of $U$.

\subsubsection*{Probability spaces}
We write $\EE$ or $\tilde \EE$ for the expectation with respect to a probability space $\left(\Omega,\cF,\PP\right)$ or $\left(\tilde\Omega,\tilde\cF,\tilde\PP\right)$, respectively. The symbol $\qvar{\cdot}{t}$ denotes the quadratic variation process. For probability spaces $\left(\Omega,\cF,\PP\right)$ and $\left(\tilde\Omega,\tilde\cF,\tilde\PP\right)$, and a topological space $(\cX,\cT)$, suppose we are given random variables $X \colon \Omega \to \cX$ and $\tilde X \colon \tilde\Omega \to \cX$. Then we write $X \sim \tilde X$ and say that the laws of $X$ and $\tilde X$ coincide if $\PP\{X \in \cU\} = \tilde\PP\{\tilde X \in \cU\}$ for every $\cU \in \cT$.

\subsubsection*{Constants}
In what follows, $c$, $C$, $c_j$, and $C_j$ will denote generic positive and finite constants and if deemed necessary, their (in-)dependence on parameters or functions is specified.

\section{Deterministic dynamics\label{sec:deterministic}}
Consider the deterministic thin-film dynamics \eqref{deterministic_intro}, i.e.,
\begin{equation}\label{deterministic_3}
\partial_t v = - \partial_x \left(v^2 \partial_x^3 v\right) \quad \mbox{on} \quad [0,\delta).
\end{equation}
We use the existence and regularity results on solutions to \eqref{deterministic_3} developed in \cite{BerettaBertschDalPasso1995,BertozziPugh1996} as the proof of non-negativity therein does not require the use of the entropy as in \cite{BernisFriedman1990}. Note that Beretta, Bertsch, and Dal Passo in \cite{BernisFriedman1990} consider solutions to \eqref{deterministic_3} on the interval $[0,L]$ but with homogeneous (Neumann) data
\[
\partial_x v(\cdot,0) = \partial_x v(\cdot,L) = 0 \quad \mbox{and} \quad (v^n\partial_x^3 v)(\cdot,0) = (v^n\partial_x^3 v)(\cdot,L) = 0,
\]
though the construction of solutions on the torus $\TTT_L$ works in the same manner. The following statements form a summary of those in \cite[Proposition~1.1]{BerettaBertschDalPasso1995} and \cite[Theorem~2.1, Proposition~4.6, Proposition~4.8]{BertozziPugh1996}.

\begin{theorem}[Beretta, Bertsch, Dal Passo \cite{BerettaBertschDalPasso1995}, Bertozzi and Pugh \cite{BertozziPugh1996}]\label{th:tfe_weak}
Assume that $v_0 \in H^1(\TTT_L)$ with $v_0 \ge 0$. Then, there exists a function $v \colon [0,\delta) \times \TTT_L \to [0,\infty)$ with the following properties:
\begin{enumerate}
\item\label{item:bf1} $v \in C^{\frac 1 8, \frac 1 2}\left([0,\delta] \times \TTT_L\right)$ (mixed H\"older continuity with exponent $\frac 1 8$ in time and $\frac 1 2$ in space).
\item Initial value: $v(0,\cdot) = v_0$ in the sense that $\vertii{v(t,\cdot)-v_0}_{1,2} \to 0$ as $t \searrow 0$.
\item $v \in L^\infty\left([0,\delta);H^1(\TTT_L)\right)$.
\item $v^2 \partial_x^3 v \in L^2\left(\{v > 0\}\right)$.
\item\label{item:mass_conserved} Mass conservation: $\int_0^L v \, \d x = \int_0^L v_0 \, \d x$ on the time interval $[0,\delta)$.
\item\label{item:bf7} The function $v$ satisfies
\begin{equation}\label{test_tfe}
\int_0^\delta \int_0^L v \, (\partial_t \phi) \, \d x \, \d t + \int_0^\delta \int_{\{v(t,\cdot) > 0\}} v^2 \, (\partial_x^3 v) \, (\partial_x \phi) \, \d x \, \d t = 0
\end{equation}
for all $\phi \in C_\mathrm{c}^\infty\left((0,\delta);C^\infty(\TTT_L)\right)$.
\end{enumerate}
\end{theorem}

\medskip

In addition to mass conservation, we also need a quantitative energy estimate, which essentially follows from the construction of \cite{BernisFriedman1990,BerettaBertschDalPasso1995,BertozziPugh1996}:

\begin{corollary}[quantitative estimate]\label{cor:energy_tfe}
In the situation of Theorem~\ref{th:tfe_weak} there exists a solution $v \colon [0,\delta) \times \TTT_L \to [0,\infty)$ satisfying the properties \eqref{item:bf1}--\eqref{item:bf7} and further
\begin{equation}\label{apriori_bf}
\vertii{\partial_x v(t,\cdot)}_2^p + 2 \int_0^t \vertii{\partial_x v(t',\cdot)}_2^{p-2} \int_{\{v(t',\cdot) > 0\}} (v(t',x))^2 (\partial_x^3 v(t',x))^2 \, \d x \, \d t' \le \vertii{\partial_x v_0}_2^p
\end{equation}
for $t \in [0,\delta)$, where $p \in [2,\infty)$ is arbitrary.
\end{corollary}
\begin{proof}[Proof of Corollary~\ref{cor:energy_tfe}]
Denote by $v^\eps$ unique classical solutions to the approximating problems
\[
\partial_t v^\eps + \partial_x \left(f_\eps\left(v^\eps\right) \partial_x^3 v^\eps\right) = 0 \quad \mbox{in} \quad [0,\delta) \times \TTT_L,
\]
with $f_\eps(s) = \frac{s^4}{\eps + s^2}$ and initial data $v_0^\eps \in C^\infty(\TTT_L)$ such that $v_0^\eps > 0$ and $\vertii{v_0 - v_0^\eps}_{1,2} \to 0$ as $\eps \searrow 0$ (cf.~\cite[\S1]{BerettaBertschDalPasso1995} for details). From \cite[Eq.~(1.8)]{BerettaBertschDalPasso1995} we infer that
\begin{equation}\label{apriori_bf_eps}
\vertii{\partial_x v^\eps(t,\cdot)}_2^2 + 2 \int_0^t \int_{\{v^\eps(t',\cdot) > 0\}} f_\eps(v^\eps) (\partial_x^3 v^\eps)^2 \, \d x \, \d t' = \vertii{\partial_x v_0^\eps}_2^2 \quad \mbox{for all} \quad t \in [0,\delta)
\end{equation}
holds true. Since as $\eps \searrow 0$ a subsequence of $v^\eps$ uniformly converges to $v$ of Theorem~\ref{th:tfe_weak} (cf.~\cite[(1.13)]{BerettaBertschDalPasso1995}), for any $r > 0$ and $\eps > 0$ sufficiently small such that $\sup_{(t,x) \in [0,\delta) \times \TTT_L} \verti{v^\eps(t,x) - v(t,x)} \le \frac r 2$, we have
\[
\int_0^\delta \int_{\{v(t,\cdot) > r\}} (\partial_x^3 v^\eps)^2 \, \d x \, \d t \le \frac{1}{f_\eps(r/2)} \int_0^\delta \int_{\{v(t,\cdot) > r\}} f_\eps(v^\eps) (\partial_x^3 v^\eps)^2 \, \d x \, \d t \stackrel{\eqref{apriori_bf_eps}}{\le} \frac{8 \eps + 2 r^2}{r^4} \vertii{\partial_x v_0}_2^2.
\]
A diagonal sequence argument implies that, up to taking another subsequence, we have for some $\zeta \in L^2(\TTT_L)$
\[
\int_0^\delta \int_{\{v(t,\cdot) > r\}} \eta \, \partial_x^3 v^\eps \, \d x \, \d t \to \int_0^\delta \int_{\{v(t,\cdot) > r\}} \eta \, \zeta \, \d x \, \d t \quad \mbox{as} \quad \eps \searrow 0
\]
for any $r > 0$ and any $\eta \in C^\infty_\mathrm{c}\left(\{v > r\}\right)$. On the other hand, through integration by parts and bounded convergence
\[
\int_0^\delta \int_{\{v(t,\cdot) > r\}} \eta \, (\partial_x^3 v^\eps) \, \d x \, \d t = - \int_0^\delta \int_{\{v(t,\cdot) > r\}} (\partial_x^3 \eta) \, v^\eps \, \d x \, \d t \to - \int_0^\delta \int_{\{v(t,\cdot) > r\}} (\partial_x^3 \eta) \, v \, \d x \, \d t,
\]
as $\eps \searrow 0$, i.e., $\zeta = \partial_x^3 v$. From \eqref{apriori_bf_eps} we deduce that, up to taking another subsequence, also estimate \eqref{apriori_bf} is valid for $p = 2$ by weak lower-semicontinuity using $\partial_x^3 v^\eps \rightharpoonup \partial_x^3 v$ in $L^2\left(\{v > 0\}\right)$ and $\sup_{(t,x) \in [0,\delta) \times \TTT_L} \verti{v^\eps(t,x) - v(t,x)} \to 0$ as $\eps \searrow 0$. Estimate~\eqref{apriori_bf} for $p \in [2,\infty)$ follows from the one for $p = 2$ by noting that
\begin{eqnarray*}
\lefteqn{\vertii{\partial_x v(t,\cdot)}_2^p + \int_0^t \vertii{\partial_x v(t',\cdot)}_2^{p-2} \int_{\{v(t',\cdot) > 0\}} (v(t',x))^2 (\partial_x^3 v(t',x))^2 \, \d x \, \d t'} \\
&\le& \sup_{t' \in [0,t]} \vertii{\partial_x v(t',\cdot)}_2^{p-2} \left(\vertii{\partial_x v(t,\cdot)}_2^2 + \int_0^t \int_{\{v(t',\cdot) > 0\}} (v(t',x))^2 (\partial_x^3 v(t',x))^2 \, \d x \, \d t'\right) \\
&\le& \vertii{\partial_x v_0}_2^{p-2} \left(\vertii{\partial_x v(t,\cdot)}_2^2 + \int_0^t \int_{\{v(t',\cdot) > 0\}} (v(t',x))^2 (\partial_x^3 v(t',x))^2 \, \d x \, \d t'\right) \le \vertii{\partial_x v_0}_2^p.
\end{eqnarray*}
\end{proof}
%

\section{Stochastic dynamics\label{sec:stochastic}}
Denote by $\left(\Omega,\cF,\left(\cF_t\right)_{t \in [0,\delta]},\PP\right)$ a complete filtered probability space such that the filtration $\left(\cF_t\right)_{t \in [0,\delta]}$ is complete and right-continuous. Further denote by $\left(\beta^k\right)_{k \in \Z}$ mutually independent standard real-valued $(\cF_t)$-Wiener processes. Our aim is to construct weak solutions to equation~\eqref{stochastic_2}, i.e.,
\begin{equation}\label{stochastic_3}
\d w =  \frac 1 2 \sum_{k \in \Z} \lambda_k^2 \partial_x\left(\psi_k\partial_x(\psi_k w)\right) \d t + \sum_{k \in \Z} \lambda_k \partial_x (\psi_k w) \, \d\beta^k \quad \mbox{on} \quad [0,\delta),
\end{equation}
satisfying suitable bounds. The material leading to Proposition~\ref{prop:weak} is standard (see for instance \cite{KR82,GGK15}) and given in Appendix~\ref{sec:viscous}. There, we present some additional details in order to track the dependency of the occurring constants on the time step, which will be needed below.

\medskip

We introduce the operator
\begin{equation}\label{a_op_0}
A^0 \colon H^1(\TTT_L) \to H^{-1}(\TTT_L), \quad w \mapsto \frac 1 2 \sum_{k \in \Z} \lambda_k^2 \partial_x\left(\psi_k\partial_x(\psi_k w)\right)
\end{equation}
and the diagonal Hilbert-Schmidt-valued operator
\begin{equation}\label{b_op_0}
B^0 \colon H^1(\TTT_L) \to L_2\left(H^2(\TTT_L);L^2(\TTT_L)\right), \quad w \mapsto \left(v \mapsto \sum_{k \in \Z} \lambda_k \left(v,\psi_k\right)_{2,2} \left(\partial_x (\psi_k w)\right)\right).
\end{equation}
Equation~\eqref{stochastic_3} now attains the abstract form
\begin{equation}\label{stochastic_3_abstract}
\partial_t w = A^0 w + \left(B^0 w\right) \d W_{H^2(\TTT_L)}, \quad \mbox{where} \quad W_{H^2(\TTT_L)} := \sum_{k \in \Z} \beta^k \psi_k.
\end{equation} 
Note that $W_{H^2(\TTT_L)}$ is a cylindrical $(\cF_t)$-Wiener process in $H^2(\TTT_L)$ with $(B^0 w) \d W_{H^2(\TTT_L)} = \partial_x (w \circ \d W)$ for any $w \in H^1(\TTT_L)$, where $W$ is as in \eqref{def_noise}. We introduce the concept of weak solutions to \eqref{stochastic_3_abstract}:
\begin{definition}\label{def:sol_eps_0}
A weak solution to \eqref{stochastic_3_abstract} is a continuous $(\cF_t)$-adapted $L^2(\TTT_L)$-valued process $w$ such that its $\d t \otimes \d\PP$-equivalence class $\hat w$ meets
\[
\hat w \in L^2\left([0,\delta) \times \Omega, \d t \otimes \d\PP;H^1(\TTT_L)\right)
\]
and $\PP$-almost surely
\begin{equation}\label{solution_formula}
w(t,\cdot) = w_0 + \int_0^t A^0 \bar w(t',\cdot) \, \d t' + \int_0^t \left(B^0 \bar w(t',\cdot)\right) \d W_{H^2(\TTT_L)}(t',\cdot) \quad \mbox{for} \quad t \in [0,\delta),
\end{equation}
where $\bar w$ denotes any $H^1(\TTT_L)$-valued progressively measurable $\d t \otimes \d\PP$-version of $\hat w$.
\end{definition}

With help of Proposition~\ref{prop:var} we can show:
\begin{proposition}\label{prop:weak}
Suppose that $p \in [2,\infty)$ and let \eqref{cond_lambda} hold true. Then, for any
\[
w_0 \in L^p\left(\Omega,\cF_0,\PP;H^1(\TTT_L)\right)
\]
there exists a solution $w$ of \eqref{stochastic_3} with initial data $w_0$ satisfying the a-priori estimates
\begin{subequations}\label{apriori_w}
\begin{align}
\EE \esssup_{t \in [0,\delta)} \vertii{w(t,\cdot)}_{1,2}^p &\le C_1 \, \EE\vertii{w_0}_{1,2}^p, \label{apriori_w_1} \\
\limsup_{t \nearrow \delta} \EE \vertii{\partial_x w(t,\cdot)}_2^p &\le e^{C_2 \delta} \left(\EE \vertii{\partial_x w_0}_2^p + C_3 \, \delta \, \EE \verti{\int_0^L w_0 \, \d x}^p\right), \label{apriori_w_2}
\end{align}
\end{subequations}
where $C_1, C_2, C_3 < \infty$ are independent of $\delta$, $w$, and $w_0$. Furthermore, the mass is conserved, i.e., $\int_0^L w(t,\cdot) \, \d x = \int_0^L w_0 \, \d x$ holds true for $t \in [0,\delta)$, $\PP$-almost surely.
\end{proposition}
\begin{proof}[Proof of Proposition~\ref{prop:weak}]
Suppose that $w^\eps$ is the unique variational solution to the regularized equation \eqref{stochastic_regular} below, i.e.,
\[
\d w^\eps =  \left(\frac 1 2 \sum_{k \in \Z} \lambda_k^2 \partial_x\left(\psi_k\partial_x(\psi_k w^\eps)\right) + \eps \, \partial_x^2 w^\eps\right) \d t + \sum_{k \in \Z} \lambda_k \partial_x (\psi_k w^\eps) \, \d\beta^k \quad \mbox{on} \quad [0,\delta),
\]
given by Proposition~\ref{prop:var} below. Since the bound \eqref{apriori_w_eps_1} of Proposition~\ref{prop:var} is satisfied uniformly in $\eps$, by weak-$*$ sequential compactness of $L^p\left(\Omega,\cF,\PP;L^\infty\left([0,\delta);H^1(\TTT_L)\right)\right)$, we may take a subsequence, again denoted by $w^\eps$, that weak-$*$-converges to a limit function $w \in L^p\left(\Omega,\cF,\PP;L^\infty\left([0,\delta);H^1(\TTT_L)\right)\right)$. Testing \eqref{solution_formula_eps} of Definition~\ref{def:sol_eps} below with $\varphi \in C^\infty(\TTT_L)$ gives
\begin{align*}
\left(w^\eps(t,\cdot),\varphi\right)_2 &= \left(w_0,\varphi\right)_2 - \frac 1 2 \sum_{k \in \Z} \lambda_k^2 \int_0^t \left(\psi_k \partial_x \left(\psi_k \bar w^\eps(t',\cdot)\right), \partial_x \varphi\right)_2 \d t' \\
&\phantom{=} - \eps \int_0^t \left(\partial_x w^\eps(t',\cdot),\partial_x \varphi\right)_2 \d t' \\
&\phantom{=} - \sum_{k \in \Z} \lambda_k \int_0^t \left(\psi_k \bar w^\eps(t',\cdot),\partial_x \varphi\right)_2 \d \beta^k(t') \quad \mbox{for} \quad t \in [0,\delta), \quad \mbox{$\PP$-almost surely}.
\end{align*}
Now, we argue as in \cite[Proof of Theorem~4.2.4]{LiuRoeckner2015}, i.e., we test against $\eta \in L^\infty \left(\Omega ,\cF,\PP;L^\infty\left([0,\delta)\right)\right)$ and pass to the limit as $\eps \searrow 0$, so that
\begin{align}\label{test_weak_eps_0}
\EE \int_0^\delta \left(w(t,\cdot),\varphi\right)_2 \eta(t) \, \d t &= \EE \int_0^\delta \left(w_0,\varphi\right)_2 \eta(t) \, \d t \\
&\phantom{=} - \frac 1 2 \sum_{k \in \Z} \lambda_k^2 \EE \int_0^\delta \int_0^t \left(\psi_k \partial_x \left(\psi_k \bar w(t',\cdot)\right), \partial_x \varphi\right)_2 \d t' \, \eta(t) \, \d t \nonumber \\
&\phantom{=} - \sum_{k \in \Z} \lambda_k \EE \int_0^\delta \int_0^t \left(\psi_k \bar w(t',\cdot),\partial_x \varphi\right)_2 \d \beta^k(t') \, \eta(t) \, \d t. \nonumber
\end{align}
Since the limiting equation \eqref{test_weak_eps_0} holds true for all test functions $\eta \in L^\infty \left(\Omega ,\cF,\PP;L^\infty\left([0,\delta)\right)\right)$, it is true almost everywhere in $(\omega,t) \in \Omega \times [0,\delta)$. Next, as in \cite[Proof of Theorem~4.2.4]{LiuRoeckner2015}, we re-define $w$ by the right-hand side of the limiting equation \eqref{test_weak_eps_0}, so that
\begin{align}
\left(w(t,\cdot),\varphi\right)_2 &= \left(w_0,\varphi\right)_2 - \frac 1 2 \sum_{k \in \Z} \lambda_k^2 \int_0^t \left(\psi_k \partial_x \left(\psi_k \bar w(t',\cdot)\right), \partial_x \varphi\right)_2 \d t' \label{weak_eps_0} \\
&\phantom{=} - \sum_{k \in \Z} \lambda_k \int_0^t \left(\psi_k \bar w(t',\cdot),\partial_x\varphi\right)_2 \d \beta^k(t') \quad \mbox{for $t \in [0,\delta)$}, \quad \mbox{$\PP$-almost surely}. \nonumber
\end{align}
Hence, \eqref{solution_formula} is indeed satisfied and the initial value $w|_{t = 0} = w_0$ holds true in $H^{-1}(\TTT_L)$, $\PP$-almost surely. Taking $\varphi = 1$ in \eqref{weak_eps_0} implies conservation of mass, i.e., $\int_0^L w(t,\cdot) \, \d x = \int_0^L w_0 \, \d x$ holds true for $t \in [0,\delta)$, $\PP$-almost surely. Furthermore, uniformity of estimates \eqref{apriori_w_eps} in $\eps$ together with weak lower-semicontinuity of the norms and mass conservation imply that estimates \eqref{apriori_w} hold true. Finally, it is immediate to notice that from \eqref{solution_formula_eps} of Definition~\ref{def:sol_eps} below it follows
\begin{align*}
w_0 &\in L^2\left(\Omega,\cF,\PP;L^2(\TTT_L)\right), \\
\left(t \mapsto A^0 \bar w(t,\cdot)\right) &\in L^2\left([0,\delta) \times \Omega, \d t \otimes \d\PP;H^{-1}(\TTT_L)\right), \\
\left(t \mapsto B^0 \bar w(t,\cdot)\right) &\in L^2\left([0,\delta) \times \Omega, \d t \otimes \d\PP;L_2\left(H^2(\TTT_L);L^2(\TTT_L)\right)\right), \\
\hat w &\in L^2\left([0,\delta) \times \Omega, \d t \otimes \d\PP;H^1(\TTT_L)\right),
\end{align*}
so that $w$ is a continuous $(\cF_t)$-adapted $L^2(\TTT_L)$-valued process by \cite[Theorem~4.2.5]{LiuRoeckner2015}.
\end{proof}

We can furthermore show nonnegativity of weak solutions to \eqref{stochastic_3}:
\begin{proposition}[nonnegativity]\label{prop:stochastic_entropy}
In the situation of Proposition~\ref{prop:weak} assume $w_0 \ge 0$, $\PP$-almost surely. Then, we have $w \ge 0$, $\PP$-almost surely. 
\end{proposition}
\begin{proof}[Proof of Proposition~\ref{prop:stochastic_entropy}]
We first introduce suitable regular entropies. Therefore, we take $\eta_\eps(s) := \eta(s/\eps)$ for $s \in \R$ and $\eps > 0$, where $\eta \in C^\infty(\R)$ with $0 \le \eta \le 1$, $\left.\eta\right|_{(-\infty,-2]} = 1$, and $\left.\eta\right|_{\R \setminus [-1,\infty)} = 0$. We define
\[
\Gamma_\eps(s) := - s \eta_\eps(s), \quad \mbox{where} \quad \eps > 0 \quad \mbox{and} \quad s \in \R,
\]
and consider the entropy functional
\[
L^2(\TTT_L) \to \R, \quad \varphi \mapsto \int_0^L \Gamma_\eps(\varphi(x)) \, \d x.
\]
Applying It\^o's lemma in form of \cite[Theorem~3.1]{Krylov2013}, one may verify conditions \cite[\S3~(i)--(iv)]{Krylov2013}, which is done in \cite[\S4]{Krylov2013}. As a result, we obtain
\begin{align*}
\int_0^L \Gamma_\eps(w(t,\cdot)) \, \d x &= \int_0^L \Gamma_\eps(w_0) \, \d x + \sum_{k \in \Z} \lambda_k \int_0^t \int_0^L (\partial_s \Gamma_\eps)(w) \, (\partial_x (\psi_k w)) \, \d x \, \d \beta^k \\
&\phantom{=} - \frac 1 2 \sum_{k \in \Z} \lambda_k^2 \int_0^t \int_0^L (\partial_s^2 \Gamma_\eps)(w) \, (\partial_x w) \, \psi_k \, (\partial_x(\psi_k w)) \, \d x \, \d t' \\
&\phantom{=} + \frac 1 2 \sum_{k \in \Z} \lambda_k^2 \int_0^t \int_0^L (\partial_s^2 \Gamma_\eps)(w) \left(\partial_x (\psi_k w)\right)^2 \d x \, \d t', \quad \mbox{$\PP$-almost surely}. 
\end{align*}
We further simplify the second line and obtain
\begin{align*}
&\int_0^L (\partial_s^2 \Gamma_\eps)(w) \, (\partial_x w) \, \psi_k \, (\partial_x(\psi_k w)) \, \d x \\
&\quad =  \int_0^L (\partial_s^2 \Gamma_\eps)(w) \, (\partial_x(\psi_k w))^2 \, \d x - \int_0^L (\partial_s^2 \Gamma_\eps)(w) \, w \, (\partial_x \psi_k) \,(\partial_x(\psi_k w)) \, \d x \\
&\quad =  \int_0^L (\partial_s^2 \Gamma_\eps)(w) \, (\partial_x(\psi_k w))^2 \, \d x - \int_0^L (\partial_s^2 \Gamma_\eps)(w) w^2 \, (\partial_x \psi_k)^2 \, \d x \\
&\quad \phantom{=} - \frac 1 2 \int_0^L (\partial_s^2 \Gamma_\eps)(w) \, w \, (\partial_x w) \, (\partial_x \psi_k^2) \, \d x \\
&\quad =  \int_0^L (\partial_s^2 \Gamma_\eps)(w) \, (\partial_x(\psi_k w))^2 \, \d x - \int_0^L (s^2 \partial_s^2 \Gamma_\eps)(w) \, (\partial_x \psi_k)^2 \, \d x + \frac 1 2 \int_0^L \bar \Gamma_\eps(w) \, (\partial_x^2 \psi_k^2) \, \d x,
\end{align*}
$\PP$-almost surely, where we have defined
\[
\bar \Gamma_\eps(s) := \int_s^0 (s \partial_s^2 \Gamma_\eps)(s_1) \, \d s_1.
\]
This implies
\begin{align*}
\int_0^L \Gamma_\eps(w(t,\cdot)) \, \d x &= \int_0^L \Gamma_\eps(w_0) \, \d x + \sum_{k \in \Z} \lambda_k \int_0^t \int_0^L (\partial_s \Gamma_\eps)(w) \, \partial_x (\psi_k w) \, \d x \, \d \beta^k \\
&\phantom{=} + \frac 1 2 \sum_{k \in \Z} \lambda_k^2 \int_0^t \int_0^L (s^2 \partial_s^2 \Gamma_\eps)(w) \, (\partial_x \psi_k)^2 \, \d x \, \d t' \\
&\phantom{=} - \frac 1 4 \sum_{k \in \Z} \lambda_k^2 \int_0^t \int_0^L \bar \Gamma_\eps(w) \, (\partial_x^2 \psi_k^2) \, \d x \, \d t', \quad \mbox{$\PP$-almost surely}.
\end{align*}
Next, we recognize that
\begin{align*}
s^2 \partial_s^2 \Gamma_\eps(s) &= - 2 s (s \partial_s \eta_\eps)(s) - s (s^2 \partial_s^2 \eta_\eps)(s) \le \eps \left(2 \sup_{s \in \R} \verti{s \partial_s \eta(s)} + \sup_{s \in \R} \verti{s^2 \partial_s^2 \eta(s)}\right), \\
\verti{\bar \Gamma_\eps(s)} &= \verti{\int_s^0 \left(2 (s \partial_s \eta_\eps)(s_1) + (s^2 \partial_s^2 \eta_\eps)(s_1)\right) \d s_1} \le \eps \left(2 \sup_{s \in \R} \verti{s \partial_s \eta(s)} + \eps \sup_{s \in \R} \verti{s^2 \partial_s^2 \eta(s)}\right),
\end{align*}
where we have used $\supp (s^2 \partial_s^2 \eta_\eps) \subseteq \supp (s \partial_s \eta_\eps) \subseteq [-2\eps,-\eps]$. This implies together with \eqref{basis}, \eqref{der_psi_k}, \eqref{cond_lambda}, and after taking the expectation,
\begin{align*}
\esssup_{t \in [0,\delta)} \EE \int_0^L \Gamma_\eps(w(t,\cdot)) \, \d x &\le \EE \int_0^L \Gamma_\eps(w_0) \, \d x + C \eps \delta,
\end{align*}
where $C < \infty$ is independent of $\eps$, $\delta$, $w$, and $w_0$. Since $\Gamma_{\eps_1}(s) \le \Gamma_{\eps_2}(s)$ for $\eps_1 \ge \eps_2 \ge 0$ and $s \in \R$, we may take the limit as $\eps \searrow 0$ and get by monotone convergence,
\begin{align*}
\esssup_{t \in [0,\delta)} \left(- \EE \int_{\{w < 0\}} w(t,x) \, \d x\right) &= \esssup_{t \in [0,\delta)} \EE \int_0^L \Gamma_0(w(t,x)) \, \d x \\
& \le \EE \int_0^L \Gamma_0(w_0) \, \d x = - \EE \int_{\{w_0 < 0\}} w_0 \, \d x = 0.
\end{align*}
This implies $w \ge 0$, $\PP$-almost surely.
\end{proof}
%

\section{Regularity in time and uniform bounds of approximate solutions\label{sec:time_reg_approx}}
We use the notations and conventions of \S\ref{sec:weak_form}. For $u_0 \in H^1(\TTT_L)$ such that $u_0 \ge 0$, we define for every $N \in \N$ solutions $v_N \colon \Omega \times [0,T) \times \TTT_L \to [0,\infty)$ and $w_N \colon \Omega \times [0,T) \times \TTT_L \to [0,\infty)$ according to the splitting scheme \ref{item:deterministic}--\ref{item:stochastic}--\ref{item:deterministic_stochastic} through Theorem~\ref{th:tfe_weak}, Corollary~\ref{cor:energy_tfe}, and Proposition~\ref{prop:weak}. Note that indeed by Theorem~\ref{th:tfe_weak}~\eqref{item:bf1} and Definition~\ref{def:sol_eps_0} the limits $w_N((j-1)\delta,\cdot) = \lim_{t \nearrow j \delta} v_N(t,\cdot)$ and $v_N(j\delta,\cdot) = \lim_{t \nearrow j \delta} w_N(t,\cdot)$ are, $\PP$-almost surely, attained in $C^{\frac 1 2}(\TTT_L)$ and $L^2(\TTT_L)$, respectively, and because of \eqref{apriori_bf} of Corollary~\ref{cor:energy_tfe}, \eqref{apriori_w} of Proposition~\ref{prop:weak}, and weak lower-semicontunity of the appearing norms, we have $\EE \vertii{\partial_x w_N((j-1)\delta,\cdot)}_2^p < \infty$ and $ \vertii{\partial_x v_N(j\delta,\cdot)}_2^p < \infty$, where $j \in \{1,\ldots,N+1\}$. We further define the concatenated approximate solution $u_N \colon \Omega \times [0,T) \times \TTT_L \to [0,\infty)$ by
\begin{equation}\label{def_concat}
u_N(t,\cdot) := \begin{cases} v_N(2 t - (j-1) \delta, \cdot) & \mbox{for} \quad t \in \left[(j-1) \delta, (j - \frac 1 2) \delta\right), \\ w_N(2 t - j \delta, \cdot) & \mbox{for} \quad t \in \left[(j - \frac 1 2) \delta, j \delta\right), \end{cases} \quad \mbox{where} \quad j \in \left\{1,\ldots,N+1\right\},
\end{equation}
where we recall the notation $\delta = \frac{T}{N+1}$. By Theorem~\ref{th:tfe_weak} and Propositions~\ref{cor:energy_tfe} and \ref{prop:stochastic_entropy} we have $\int_0^L u_N \, \d x = \int_0^L u_0 \, \d x$ in $[0,T)$, $\PP$-almost surely, and $u_N \ge 0$, $\PP$-almost surely, for every $t \in [0,T)$. Furthermore, we can prove:
\begin{proposition}\label{prop:reg_vn_wn}
For any $p \in [2,\infty)$, there exists a constant $C < \infty$ such that for all $N \in \N$ we have
\[
v_N, w_N \in L^p\left(\Omega,\cF,\PP;L^\infty\left([0,T);H^1(\TTT_L)\right)\right)
\]
with
\begin{equation}\label{apriori}
\begin{aligned}
& \EE \esssup_{t \in [0,T)} \vertii{u_N(t,\cdot)}_{1,2}^p + \EE \esssup_{t \in [0,T)} \vertii{v_N(t,\cdot)}_{1,2}^p + \EE \esssup_{t \in [0,T)} \vertii{w_N(t,\cdot)}_{1,2}^p \\
& \quad + \EE \int_0^T \vertii{v_N(t,\cdot)}_{1,2}^{p-2} \int_{\{v_N(t,\cdot) > 0\}} (v_N \partial_x^3 v_N)^2 \, \d x \, \d t \le C \vertii{u_0}_{1,2}^p.
\end{aligned}
\end{equation}
\end{proposition}
\begin{proof}[Proof of Proposition~\ref{prop:reg_vn_wn}]
By Theorem~\ref{th:tfe_weak}~\eqref{item:mass_conserved}, Proposition~\ref{prop:weak}, and the fact that due to \eqref{connection} of property \ref{item:deterministic_stochastic} there are no jumps of $u_N$ at times $t \in \left\{\frac \delta 2,\ldots,(2N+1) \frac \delta 2\right\}$, we have
\begin{subequations}\label{glue}
\begin{equation}\label{glue_1}
\int_0^L u_0 \, \d x = \int_0^L u_N \, \d x = \int_0^L v_N \, \d x = \int_0^L w_N \, \d x \quad \mbox{for all $t \in [0,T)$}, \quad \mbox{$\PP$-almost surely},
\end{equation}
i.e., the mass is conserved. Since surface energy for the deterministic dynamics \ref{item:deterministic} is dissipated due to \eqref{apriori_bf} of Corollary~\ref{cor:energy_tfe} and the growth of the initial value for the stochastic dynamics \ref{item:stochastic} is controlled due to \eqref{apriori_w_2} of Proposition~\ref{prop:weak}, we obtain
\begin{align}
\EE \vertii{\partial_x v_N(j \delta,\cdot)}_2^p \le e^{C_2 j \delta}\left(\vertii{\partial_x u_0}_2^p + C_3 j \delta \left(\int_0^L u_0 \, \d x\right)^p\right), \label{glue_2} \\
\EE \vertii{\partial_x w_N(j \delta,\cdot)}_2^p \le e^{C_2 j \delta}\left(\vertii{\partial_x u_0}_2^p + C_3 j \delta \left(\int_0^L u_0 \, \d x\right)^p\right) \label{glue_3}
\end{align}
for $j \in \{0,\ldots,N\}$, where $C_2$ and $C_3$ are as in \eqref{apriori_w}. The combination of \eqref{glue_1} with \eqref{glue_2} and \eqref{glue_3} utilizing Poincar\'e's inequalities
\[
c \vertii{\varphi}_{1,2} \le \vertii{\partial_x \varphi}_2 + \verti{\int_0^L \varphi \, \d x} \le C \vertii{\varphi}_{1,2} \quad \mbox{for fixed $0 < c \le C < \infty$ and all $\varphi \in H^1(\TTT_L)$} 
\]
implies that there exists $C < \infty$ such that
\begin{equation}\label{glue_4}
\EE \vertii{v_N(j \delta,\cdot)}_{1,2}^p \le C \vertii{u_0}_{1,2}^p \quad \mbox{and} \quad \EE \vertii{w_N(j \delta,\cdot)}_{1,2}^p \le C \vertii{u_0}_{1,2}^p
\end{equation}
\end{subequations}
for $j \in \{0,\ldots,N\}$. Now combining \eqref{glue} with \eqref{apriori_bf} of Corollary~\ref{cor:energy_tfe} and \eqref{apriori_w_1} of Proposition~\ref{prop:weak} and making use of Poincar\'e's inequalities once more, we obtain \eqref{apriori} upon enlarging $C$.
\end{proof}
\begin{proposition}[regularity in time]\label{prop:reg_time_un}
For any $p \in [2,\infty)$, $\eps > 0$, $\kappa \in \left(2 \eps,2 p^{-1}\right) \cap \left(2 \eps,\frac 1 2\right]$, and $q \in \left(\frac{2}{\kappa-2\eps},\infty\right)$, there exists $C < \infty$ such that for all $N \in \N$ we have
\[
u_N \in L^p\left(\Omega,\cF,\PP;B^{\frac \kappa 2 - \eps,q}_q\left([0,T);B^{\frac 1 2 - 2 \kappa,q}_q(\TTT_L)\right)\right)
\]
with
\begin{equation}\label{bound_u_n_time}
\EE \vertii{u_N}_{B^{\frac \kappa 2 - \eps,q}_q\left([0,T);B^{\frac 1 2 - 2 \kappa,q}_q(\TTT_L)\right)}^p \le C \vertii{u_0}_{1,2}^p \left(1 + \vertii{u_0}_{1,2}^{\kappa p}\right).
\end{equation}
\end{proposition}
In order to prove Proposition~\ref{prop:reg_time_un}, we first prove regularity in time for $v_N$ and $w_N$ separately:
\begin{lemma}\label{lem:time_reg_vn}
For any $p \in [2,\infty)$, $\eps > 0$ and $q \in [p,\infty)$, there exists a constant $C < \infty$ such that for all $N \in \N$, $j \in \{1,\ldots,N+1\}$, and $\kappa \in \left(0, 2 p^{-1}\right)$ we have
\[
v_N \in L^{p}\left(\Omega,\cF,\PP;B^{\frac \kappa 2 -\eps,q}_q\left([(j-1)\delta,j\delta);B^{\frac 1 2 - 2 \kappa,q}_q(\TTT_L)\right)\right)
\]
with
\begin{equation}\label{uniform_bound_interp_1}
\EE \left(\sum_{j = 1}^{N+1} \vertii{v_N}_{B^{\frac \kappa 2-\eps,q}_q\left([(j-1)\delta,j\delta);B^{\frac 1 2 - 2 \kappa,q}_q(\TTT_L)\right)}^q\right)^{\frac p q} \le C \vertii{u_0}_{1,2}^p \left(1 + \vertii{u_0}_{1,2}^{\kappa p}\right).
\end{equation}
\end{lemma}
\begin{proof}[Proof of Lemma~\ref{lem:time_reg_vn}]
For $(j-1) \delta \le t_1 \le t_2 < j \delta$ we have from \eqref{test_tfe} of Theorem~\ref{th:tfe_weak} and a localization argument of the appearing test function in time
\[
\left(v_N(t_2,\cdot) - v_N(t_1,\cdot),\varphi\right)_2 = \int_{t_1}^{t_2} \int_{\{v_N(t,\cdot) > 0\}} v_N^2 \, (\partial_x^3 v_N) \, (\partial_x \varphi) \, \d x \, \d t, \quad \mbox{$\PP$-almost surely},
\]
where $\varphi \in C^\infty(\TTT_L)$. Applying the Cauchy-Schwarz inequality leads to
\begin{align*}
\vertii{v_N(t_2,\cdot) - v_N(t_1,\cdot)}_{H^{-1}(\TTT_L)} \le \int_{t_1}^{t_2} \left(\int_{\{v_N(t,\cdot) > 0\}} v_N^4 \, (\partial_x^3 v_N)^2 \, \d x\right)^{\frac 1 2} \d t, \quad \mbox{$\PP$-almost surely}.
\end{align*}
Hence, we obtain, after using the Sobolev embedding theorem and the Cauchy-Schwarz inequality once more,
\begin{align*}
\vertii{v_N(t_2,\cdot) - v_N(t_1,\cdot)}_{H^{-1}(\TTT_L)}^2 & \le C \left(\int_{t_1}^{t_2} \vertii{v_N(t,\cdot)}_{1,2} \left(\int_{\{v_N(t,\cdot) > 0\}} v_N^2 \, (\partial_x^3 v_N)^2 \, \d x\right)^{\frac 1 2} \d t\right)^2 \\
&\le C \, (t_2-t_1) \, \int_{t_1}^{t_2} \vertii{v_N(t,\cdot)}_{1,2}^2 \int_{\{v_N(t,\cdot) > 0\}} v_N^2 \, (\partial_x^3 v_N)^2 \, \d x \, \d t,
\end{align*}
$\PP$-almost surely, so that with help of \eqref{apriori} of Proposition~\ref{prop:reg_vn_wn}
\begin{equation}\label{uniform_bound_hoelder}
\vertii{v_N}_{L^2\left(\Omega,\cF,\PP;C^{\frac 1 2}\left([(j-1)\delta,j\delta);H^{-1}(\TTT_L)\right)\right)} \le C \vertii{u_0}_{1,2}^2,
\end{equation}
where $C < \infty$ only depends on $L$. Setting $\tilde p := \frac{2 (1-\kappa) p}{2-\kappa p}$ and
\begin{align*}
X_1 &:= L^\infty\left([(j-1)\delta,j\delta);H^1(\TTT_L)\right), \\
X_2 &:= W^{\frac 1 2,\infty}\left([(j-1)\delta,j\delta);H^{-1}(\TTT_L)\right) \hookleftarrow C^{\frac 1 2}\left([0,T);H^{-1}(\TTT_L)\right),
\end{align*}
we infer by interpolation
\begin{eqnarray}
\vertii{v_N}_{L^p\left(\Omega,\cF,\PP;\left(X_1,X_2\right)_{\kappa,p}\right)} &=& \vertii{v_N}_{\left(L^{\tilde p}\left(\Omega,\cF,\PP;X_1\right),L^2\left(\Omega,\cF,\PP;X_2\right)\right)_{\kappa,p}} \label{vn_bound_interp} \\
&\le& C \vertii{v_N}_{L^{\tilde p}\left(\Omega,\cF,\PP;X_1\right)}^{1-\kappa} \vertii{v_N}_{L^2\left(\Omega,\cF,\PP;X_2\right)}^\kappa \nonumber \\
&\stackrel{\eqref{apriori}, \eqref{uniform_bound_hoelder}}{\le}& C \vertii{u_0}_{1,2}^{1+\kappa}, \nonumber
\end{eqnarray}
where $C < \infty$ only depends on $\kappa$ and $p$, and we have applied \cite[Chapitre~VII, \S1, 1.1, Th\'eor\`eme~(1.1)]{LionsPeetre1964} or equivalently \cite[Theorem~1.18.4]{Triebel1978} in the first line and the standard interpolation inequality in the second line.

\medskip

Now, we may use \cite[Theorem~3.4.1~(b)]{BergLofstrom1976} or \cite[\S1.3.3, Theorem~(d)]{Triebel1978} and the Sobolev embedding theorem to deduce
\[
\left(X_1,X_2\right)_{\kappa,p} \hookrightarrow \left(Y_1,Y_2\right)_{\kappa,q},
\]
provided $p \le q$, where
\[
Y_1 := B^{-\eps,q}_q\left([(j-1)\delta,j\delta);B^{\frac 1 2,q}_q(\TTT_L)\right), \quad Y_2 := B^{\frac 1 2,q}_q\left([(j-1)\delta,j\delta);B^{-\frac 3 2,q}_q(\TTT_L)\right),
\]
and
\[
\vertii{v_N}_{\left(Y_1,Y_2\right)_{\kappa,q}} \le C \, \delta^{\frac{1}{q}} \vertii{v_N}_{\left(X_1,X_2\right)_{\kappa,p}}
\]
for $C < \infty$ independent of $\delta$ and $j \in \{1,\ldots,N+1\}$. From Lemma~\ref{lem:interpolation} below we infer
\[
\left(X_1,X_2\right)_{\kappa,p} \hookrightarrow B^{\frac \kappa 2 -\eps,q}_q\left([(j-1)\delta,j\delta);B^{\frac 1 2 - 2 \kappa,q}_q(\TTT_L)\right),
\]
with
\[
\vertii{v_N}_{B^{\frac \kappa 2 -\eps, q}_q\left([(j-1)\delta,j\delta);B^{\frac 1 2 - 2 \kappa,q}_q(\TTT_L)\right)} \le C \, \delta^{\frac{1}{q}} \vertii{v_N}_{\left(X_1,X_2\right)_{\kappa,p}},
\]
where $C < \infty$ is independent of $\delta$ and $j \in \{1,\ldots,N+1\}$. In conjunction with \eqref{vn_bound_interp} this implies \eqref{uniform_bound_interp_1} after raising to the power $q$ and summation over $j \in \{1,\ldots,N+1\}$.
\end{proof}
\begin{lemma}\label{lem:time_reg_wn}
For any $p \in [2,\infty)$, $\eps > 0$, and $q \in [p,\infty)$, there exists $C < \infty$ such that for all $N \in \N$, $j \in \{1,\ldots,N+1\}$, and $\gamma \in (0,1)$ we have
\[
w_N \in L^p\left(\Omega,\cF,\PP;B^{\frac \gamma 2 - \eps,q}_q\left([(j-1)\delta,j\delta);B^{\frac 1 2 - 2 \gamma,q}_q(\TTT_L)\right)\right)
\]
with
\begin{equation}\label{uniform_bound_interp_2}
\EE \left(\sum_{j = 1}^{N+1} \vertii{w_N}_{B^{\frac \gamma 2 - \eps,q}_q\left([(j-1)\delta,j\delta);B^{\frac 1 2 - 2 \gamma,q}_q(\TTT_L)\right)}^q\right)^{\frac p q} \le C \vertii{u_0}_{1,2}^p.
\end{equation}
\end{lemma}
\begin{proof}[Proof of Lemma~\ref{lem:time_reg_wn}]
We derive higher regularity in time $t$ for $w_N$. From \eqref{a_op_0}, \eqref{b_op_0}, \eqref{stochastic_3_abstract}, and \eqref{solution_formula} of Definition~\ref{def:sol_eps_0}, we infer
\begin{align*}
w_N(t,\cdot) &= w_N((j-1)\delta,\cdot) + \underbrace{\int_{(j-1)\delta}^t \frac 1 2 \sum_{k \in \Z} \lambda_k^2 \partial_x\left(\psi_k\partial_x(\psi_k w_N(t',\cdot))\right) \d t'}_{=: w_N^{(1)}(t,\cdot)} \\
&\phantom{=} + \underbrace{\int_{(j-1)\delta}^t \sum_{k \in \Z} \lambda_k \partial_x (\psi_k w_N(t',\cdot)) \, \d\beta^k(t')}_{=: w_N^{(2)}(t,\cdot)} \quad \mbox{for} \quad t \in [(j-1)\delta,j\delta),
\end{align*}
$\PP$-almost surely. We conclude that for $(j-1)\delta \le t_1 \le t_2 < j \delta$ and $\varphi \in C^\infty(\TTT_L)$ we have 
\begin{eqnarray*}
\vertii{w_N^{(1)}(t_2,\cdot) - w_N^{(1)}(t_1,\cdot)}_{H^{-1}(\TTT_L)}^p &\stackrel{\eqref{basis}, \eqref{der_psi_k}}{\le}& C \left(\sum_{k \in \Z} \lambda_k^2 \, \int_{t_1}^{t_2} \vertii{w_N(t,\cdot)}_{1,2} \, \d t\right)^p \\
&\stackrel{\eqref{cond_lambda}}{\le}& C \, \verti{t_1-t_2}^p \, \esssup_{t \in [0,T)} \vertii{w_N(t,\cdot)}_{1,2}^p,
\end{eqnarray*}
$\PP$-almost surely, so that with help of \eqref{apriori} of Proposition~\ref{prop:reg_vn_wn}
\[
\vertii{w_N^{(1)}}_{L^p\left(\Omega,\cF,\PP;C^{1-}\left([(j-1)\delta,j\delta];H^{-1}(\TTT_L)\right)\right)} \le C \vertii{u_0}_{1,2}.
\]
From \cite[Lemma~2.1]{FlandoliGatarek1995} we may further deduce for the stochastic integral
\begin{eqnarray*}
\lefteqn{\vertii{w_N^{(2)}}_{L^p\left(\Omega,\cF,\PP;W^{\frac 1 2 - \eps,q}\left([(j-1)\delta,j\delta);L^2(\TTT_L)\right)\right)}^q} \\
&\le& C \, \EE \int_{(j-1) \delta}^{j \delta} \left(\sum_{k \in \Z} \lambda_k^2 \vertii{\partial_x (\psi_k w_N(t,\cdot))}_2^2\right)^{\frac{q}{2}} \d t \\
&\stackrel{\eqref{basis},\eqref{der_psi_k},\eqref{cond_lambda}}{\le}& C \, \delta \, \EE \esssup_{t \in [(j-1)\delta,j\delta)} \vertii{w_N(t,\cdot)}_{1,2}^{q} \\
&\stackrel{\eqref{apriori}}{\le}& C \, \delta \vertii{u_0}_{1,2}^{q} \quad \mbox{for} \quad \alpha < \frac 1 2,
\end{eqnarray*}
where Proposition~\ref{prop:reg_vn_wn} has been used again and $C < \infty$ is independent of $\delta$ and $j \in \{1,\ldots,N+1\}$. This implies by interpolation with \eqref{apriori} using \cite[Chapitre~VII, \S1, 1.1, Th\'eor\`eme~(1.1)]{LionsPeetre1964} or \cite[Theorem~1.18.4]{Triebel1978}, and scaling in time,
\[
\vertii{w_N}_{L^p\left(\Omega,\cF,\PP;\left(X_1,X_2\right)_{\gamma,p}\right)} \le C \, \delta^{\frac{1}{q}} \vertii{u_0}_{1,2},
\]
where
\begin{align*}
X_1 &:= L^{q}\left([(j-1)\delta,j\delta);H^1(\TTT_L)\right) \hookrightarrow B^{-\eps,q}_q\left([(j-1)\delta,j\delta);B^{\frac 1 2,q}_q(\TTT_L)\right), \\
X_2 &:= W^{1,q}\left([(j-1)\delta,j\delta);H^{-1}(\TTT_L)\right) + W^{\frac 1 2 - \eps,q}\left([(j-1)\delta,j\delta);L^2(\TTT_L)\right) \\
&\hookrightarrow B^{\frac 1 2-\eps,q}_q\left([(j-1)\delta,j\delta);B^{- \frac 3 2,q}_q(\TTT_L)\right),
\end{align*}
with $\gamma \in [0,1]$ and where $C < \infty$ is independent of $\delta$ and $j \in \{1,\ldots,N+1\}$. Using \cite[Theorem~3.4.1~(b)]{BergLofstrom1976} or \cite[\S1.3.3, Theorem~(d)]{Triebel1978} and Lemma~\ref{lem:interpolation} below, we infer
\[
\left(X_1,X_2\right)_{\gamma,p} \hookrightarrow B^{\frac \gamma 2 -\eps,q}_q\left([(j-1)\delta,j\delta);B^{\frac 1 2 - 2 \gamma,q}_q(\TTT_L)\right),
\]
uniformly in $\delta$ and $j \in \{1,\ldots,N+1\}$, which leads to \eqref{uniform_bound_interp_2} as in the proof of Lemma~\ref{lem:time_reg_vn}.
\end{proof}

Proposition~\ref{prop:reg_time_un} follows by applying Lemmata~\ref{lem:time_reg_vn} and \ref{lem:time_reg_wn}:

\begin{proof}[Proof of Proposition~\ref{prop:reg_time_un}]
We may choose $\kappa = \gamma$ in \eqref{uniform_bound_interp_1} and \eqref{uniform_bound_interp_2} of Lemmata~\ref{lem:time_reg_vn} and \ref{lem:time_reg_wn} and note that by construction the function $u_N$ does not jump at times $t \in \left\{\frac \delta 2,\ldots,(2N+1) \frac \delta 2\right\}$. Since by assumption $\frac \kappa 2 - \eps - \frac 1 q > 0$, we have
\[
u_N \in L^p\left(\Omega,\cF,\PP;BC^0\left([0,T);B^{\frac 1 2 - 2 \kappa,q}_q(\TTT_L)\right)\right),
\]
so that Lemma~\ref{lem:connect} below is applicable, giving the bound \eqref{bound_u_n_time}.
\end{proof}
%

\section{Convergence of the splitting scheme\label{sec:sol_stfe}}
In this section, we pass to the limit as $N \to \infty$ (implying $\delta = \frac{T}{N+1} \to 0$) for the scheme \ref{item:deterministic}-\ref{item:stochastic}-\ref{item:deterministic_stochastic}. We use the notations and conventions introduced in \S\ref{sec:weak_form} and \S\ref{sec:time_reg_approx}. Note that the present reasoning is quite similar to the one in \cite[\S5]{FischerGruen2018}, except for those parts that are specific to the Trotter-Kato scheme \ref{item:deterministic}--\ref{item:stochastic}--\ref{item:deterministic_stochastic} and the lack of an interface potential (cf.~Proposition~\ref{prop:identify_limit_2}). We also refer to \cite[Proposition~5.4]{DareiotisGess2018} and to \cite[Theorem~3.1]{FlandoliGatarek1995} for other examples in which analogous arguments have been applied.

\subsection{Tightness and convergence of a subsequence\label{sec:tight}}
We make use of the following abstract result, which is a generalization of a theorem due to Skorokhod (cf.~\cite{Skorokhod1955}):
\begin{theorem}[Jakubowski~\cite{Jakubowski1997}]\label{th:jakubowski}
Suppose that $(\XX,\TT)$ is a topological space such that there exists a countable family $\left(f_N \colon \XX \to [-1,1]\right)_{N \in \N}$ of $\TT$-continuous functions separating points of $\XX$. Further assume that $\left(X_N\right)_{N \in \N}$ is a sequence of $\XX$-valued random variables and that for all $M \in \N$ there exists $\cK_M \Subset \XX$ such that for all $N \in \N$ we have $\PP\{ X_N \in \cK_M\} > 1 - \frac 1 M$ (\emph{tightness}). Then, there exists a subsequence of $\left(X_N\right)_{N \in \N}$, denoted by $\left(X_N\right)_{N \in \N}$ again, and random variables $\tilde X, \tilde X_N \colon [0,1] \to \XX$, where $N \in \N$ and $[0,1]$ is equipped with the Borel $\sigma$-algebra, such that $X_N \sim \tilde X_N$ and $\lim_{N \to \infty} \tilde X_N(\omega) = \tilde X(\omega)$ for all $\omega \in [0,1]$, where the limit is attained in the topology $\TT$.
\end{theorem}

We now apply Theorem~\ref{th:jakubowski} in order to derive point-wise convergence of in law identical subsequences:
\begin{proposition}[point-wise convergence]\label{prop:point_convergence}
We define the spaces
\begin{subequations}\label{x_eps}
\begin{align}
\XX_u &:= BC^0\left([0,T) \times \TTT_L\right), \label{space_x_u} \\
\XX_J &:= L^2\left([0,T) \times \TTT_L\right) \quad \mbox{endowed with the weak topology}, \\
\XX_W &:= BC^0\left([0,T);H^2(\TTT_L)\right). \label{space_x_w}
\end{align}
\end{subequations}
Then, there exist random variables $\tilde u, \tilde u_N \colon [0,1] \to \XX_u$, $\tilde J_N, \tilde J \colon [0,1] \to \XX_J$, and $\tilde W_N', \tilde W \colon [0,1] \to \XX_W$ with
\begin{equation}\label{equivalence_ujs}
\left(\tilde u_N, \tilde J_N, \tilde W_N'\right)  \sim  \left(u_N, J_N, W\right), \quad \mbox{where} \quad J_N := \ind_{\{v_N > 0\}} \, v_N^2 (\partial_x^3 v_N),
\end{equation}
as well as $\tilde u_N(\omega) \to \tilde u(\omega)$ in $\XX_u$, $\tilde J_N(\omega) \rightharpoonup \tilde J(\omega)$ in $\XX_J$, and $\tilde W_N'(\omega) \rightharpoonup \tilde W(\omega)$ in $\XX_W$ as $N \to \infty$, for every $\omega \in [0,1]$, up to taking a subsequence.
\end{proposition}
\begin{proof}[Proof of Proposition~\ref{prop:point_convergence}]
By Markov's inequality, we have for $R > 0$ and using Proposition~\ref{prop:reg_time_un} with $\eps \in \left(0,\frac \kappa 2\right)$, $p := 2$, $q \in [2,\infty)$, and $\kappa \in \left[0,\frac 1 2\right]$,
\begin{eqnarray*}
\PP\left\{\vertii{u_N}_{B^{\frac \kappa 2 -\eps, q}_q\left([0,T);B^{\frac 1 2 - 2 \kappa,q}_q(\TTT_L)\right)} > R\right\} &\le& \frac{1}{R^2} \EE \vertii{u_N}_{B^{\frac\kappa 2 - \eps, q}_q\left([0,T);B^{\frac 1 2 - 2 \kappa,q}_q(\TTT_L)\right)}^2 \\
&\stackrel{\eqref{bound_u_n_time}}{\le}& \frac{C^2}{R^2} \vertii{u_0}_{1,2}^2 \left(1 + \vertii{u_0}_{1,2}^{2\kappa}\right) \to 0 \quad \mbox{as} \quad R \to \infty,
\end{eqnarray*}
uniformly in $N \in \N$. Hence,
\[
\PP\left\{\vertii{u_N}_{B^{\frac \kappa 2 - \eps, q}_q\left([0,T);B^{\frac 1 2 - 2 \kappa,q}_q(\TTT_L)\right)} \le R\right\} \to 1 \quad \mbox{as} \quad R \to \infty,
\]
uniformly in $N \in \N$. Now, for $\kappa < \frac 1 4$ and $q > \max\left\{\frac{2}{\kappa-2\eps}, \frac{2}{1-4\kappa}\right\}$, by using the compactness result \cite[Theorem~4.4]{Amann2000} and the embedding \cite[\S3.5.5~Corollary~(i)]{SchmeisserTriebel1987}, we infer that
\[
B^{\frac 1 2 - 2 \kappa,q}_q(\TTT_L) \hookrightarrow BC^0\left(\TTT_L\right)
\]
is compact because $\frac 1 2 - 2 \kappa - \frac 1 q > 0$. Once more using \cite[Theorem~4.4]{Amann2000} and the embeddings \cite[(3.3) \& (3.8)]{Amann2000}, we conclude that
\[
B^{\frac \kappa 2-\eps, q}_q\left([0,T);B^{\frac 1 2 - 2 \kappa,q}_q(\TTT_L)\right) \hookrightarrow \XX_u = BC^0\left([0,T) \times \TTT_L\right)
\]
is compact because $\frac \kappa 2 - \eps - \frac 1 q > 0$. Therefore, the set
\[
\left\{\vertii{u}_{B^{\frac \kappa 2-\eps, q}_q\left([0,T);B^{\frac 1 2 - 2 \kappa,q}_q(\TTT_L)\right)} \le R\right\}
\]
is a compact subset of $\XX_u$ for all $R > 0$, so that we obtain tightness of $u_N$ in $\XX_u$.

\medskip

For tightness of $J_N$, observe that, again by Markov's inequality and Proposition~\ref{prop:reg_vn_wn},
\begin{eqnarray*}
\PP\left\{\vertii{J_N}_{L^2([0,T) \times \TTT_L)} > R\right\} &\le& \frac{1}{R^2} \int_0^T \EE \int_0^L v_N^4 (\partial_x^3 v_N)^2 \, \d x \, \d t \\
&\le& \frac{C}{R^2} \EE \int_0^T \vertii{v_N(t,\cdot)}_{1,2}^2  \int_0^L v_N^2 (\partial_x^3 v_N)^2 \, \d x \, \d t \\
&\stackrel{\eqref{apriori}}{\le}& \frac{C}{R^2} \vertii{u_0}_{1,2}^4 \to 0 \quad \mbox{as} \quad R \to \infty,
\end{eqnarray*}
uniformly in $N \in \N$, and that $\left\{\vertii{J}_{L^2([0,T) \times \TTT_L)} \le R\right\}$ is weakly compact in $L^2([0,T) \times \TTT_L)$.

\medskip

For tightness of $W$ in $\XX_W$ observe that the law of $W$, $\mu_{W}(\cA) := \PP\left\{W \in \cA\right\}$, where $\cA \in \cB\left(\XX_W\right)$, is a Radon measure by \cite[Theorem~3.16]{Klenke2008}, since $\XX_W$ is a Polish space. This implies regularity from the interior, i.e.,
\[
1 = \mu_{W}(\XX_W) = \sup\left\{\mu_{W}(\cK) \colon \cK \Subset \XX_W\right\},
\]
which is a reformulation of tightness.

\medskip

Now the claim follows by application of Theorem~\ref{th:jakubowski}.
\end{proof}

In what follows, we assume that the assumptions of Proposition~\ref{prop:point_convergence} are satisfied and we use the notation introduced there. It is convenient to introduce the rescaled and periodically stopped noise
\begin{subequations}\label{def_noise_n}
\begin{align}
W_N(t,\cdot) := \begin{cases} W((j-1)\delta, \cdot) & \mbox{for} \quad t \in \left[(j-1) \delta, (j - \frac 1 2) \delta)\right), \\ W(2 t - j \delta, \cdot) & \mbox{for} \quad t \in \left[(j - \frac 1 2) \delta, j \delta\right), \end{cases} \quad \mbox{where} \quad j \in \left\{1,\ldots,N+1\right\}, \\
\tilde W_N(t,\cdot) := \begin{cases} \tilde W_N'((j-1)\delta, \cdot) & \mbox{for} \quad t \in \left[(j-1) \delta, (j - \frac 1 2) \delta)\right), \\ \tilde W_N'(2 t - j \delta, \cdot) & \mbox{for} \quad t \in \left[(j - \frac 1 2) \delta, j \delta\right), \end{cases} \quad \mbox{where} \quad j \in \left\{1,\ldots,N+1\right\}. \label{def_noise_n_t}
\end{align}
\end{subequations}
We define the real-valued processes
\begin{subequations}\label{def_beta_k}
\begin{align}
\beta^k_N(t) &:= \lambda_k^{-1} \left(W_N(t,\cdot),\psi_k\right)_{2,2}, \\
\tilde \beta^k_N(t) &:= \lambda_k^{-1} \left(\tilde W_N(t,\cdot),\psi_k\right)_{2,2}, \label{def_t_beta_k_N}\\
\tilde \beta^k(t) &:= \lambda_k^{-1} \left(\tilde W(t,\cdot),\psi_k\right)_{2,2}, \label{def_t_beta_k}
\end{align}
\end{subequations}
so that
\[
W_N = \sum_{k \in \Z} \lambda_k \psi_k \beta^k_N, \quad \tilde W_N = \sum_{k \in \Z} \lambda_k \psi_k \tilde\beta^k_N, \quad \mbox{and} \quad \tilde W = \sum_{k \in \Z} \lambda_k \psi_k \tilde\beta^k.
\]
Furthermore, we define $\begin{pmatrix} \tilde \cF_t \end{pmatrix}_{t \in [0,T)}$ as the augmented filtration of
\[
\tilde \cF_t' := \sigma\left(\tilde u(t'), \tilde J(t'), \tilde W(t') \colon 0 \le t' \le t\right).
\]
\begin{proposition}\label{prop:brownian}
The processes $\begin{pmatrix} \tilde \beta^k \end{pmatrix}_{k \in \Z}$ are mutually independent standard real-valued $(\tilde \cF_t)$-Wiener processes.
\end{proposition}
\begin{proof}[Proof of Proposition~\ref{prop:brownian}]
We note that $W$ and $\tilde W$ as well as $W_N$ and $\tilde W_N$ have the same law and that $W$ and $W_N$ take values in $\XX_W$, $\PP$-almost surely. Hence, also $\tilde W$ and $\tilde W_N$ take values in $\XX_W$, $\tilde\PP$-almost surely. By definitions \eqref{def_noise_n_t}, \eqref{def_t_beta_k_N}, and \eqref{def_t_beta_k} this implies that
\begin{equation}\label{conv_noise}
\tilde W_N \to \tilde W \quad \mbox{in} \quad \XX_W \quad \mbox{and} \quad \tilde \beta_N^k \to \tilde \beta^k  \quad \mbox{in} \quad BC^0([0,T)) \quad \mbox{as} \quad N \to \infty, \quad \mbox{$\tilde\PP$-almost surely,}
\end{equation}
where $k \in \Z$. By definition \eqref{def_t_beta_k_N}, the $\tilde \beta^k$ are real-valued and $(\tilde\cF_t)$-adapted. Furthermore, since the joint laws of $\left(\tilde \beta^k\right)_{k \in \Z}$ and $\left(\beta^k\right)_{k \in \Z}$ or $\left(\tilde \beta^k_N\right)_{k \in \Z}$ and $\left(\beta^k_N\right)_{k \in \Z}$, respectively, coincide, the $\tilde \beta^k$ or $\tilde \beta^k_N$, respectively, are mutually independent. Then it suffices to show that the $\tilde\beta^k$ are in fact $(\tilde\cF_t)$-Wiener processes. This is analogous to \cite[Proposition~5.4]{DareiotisGess2018} or \cite[Lemma~5.7]{FischerGruen2018}, so we only sketch the arguments here.

\medskip

The first step is to show that
\[
\tilde \EE\left[\left(\tilde\beta^k(t) - \tilde \beta^k(t')\right) \tilde \Phi\right] = 0, \quad \mbox{where} \quad  \tilde \Phi := \Phi\left(\left.\tilde u\right|_{[0,t']}, \left.\tilde J\right|_{[0,t']}, \left.\tilde W\right|_{[0,t']}\right)
\]
and $\Phi \in C^0\left(\left.\XX_u\right|_{[0,t']} \times \left.\XX_{J}\right|_{[0,t']} \times \left.\XX_W\right|_{[0,t']}; [0,1]\right)$, so that $\tilde\beta^k$ is an $(\tilde\cF_t')$-martingale, where again $\tilde\cF_t' := \sigma\left(\tilde u(t''), \tilde J(t''), \tilde W(t'') \colon 0 \le t'' \le t\right)$. This follows from the convergence stated in Proposition~\ref{prop:point_convergence} and \eqref{conv_noise} as well as Vitali's convergence theorem. In the same way, we may conclude that also $\left(\tilde\beta^k(t)\right)^2 - t$ is an $(\tilde\cF_t')$-martingale.

\medskip

We denote by $\left(\tilde\cF_t''\right)_{[0,T)}$ the filtration for which all $\PP$-zero sets are added to $\begin{pmatrix} \tilde\cF_t' \end{pmatrix}_{t \in [0,T)}$. Since $\tilde\cF_{t'} = \bigcap_{t'' > t'} \tilde\cF_{t''}''$, continuity in time of $\tilde\beta^k$ implies with Vitali's convergence theorem
\[
\tilde\EE \left[ \left(\tilde\beta^k(t) - \tilde\beta^k(t')\right) \tilde \phi\right] = 0
\]
for all $\tilde\cF_{t'}$-measurable and bounded $\tilde \phi \colon [0,1] \to \R$, so that $\tilde\beta^k$ is an $(\tilde\cF_t)$-martingale. The same argument shows that also $\left(\tilde\beta^k(t)\right)^2-t$ is an $(\tilde\cF_t)$-martingale. By L\'evy's characterization theorem (cf.~\cite[Theorem~3.16]{KaratzasShreve1991}), we infer that the $\tilde\beta^k$ are $(\tilde\cF_t)$-Wiener processes. 
\end{proof}

It is in fact also possible to extract point-wise convergent subsequences of $\tilde v_N$ and $\tilde w_N$ (the latter are defined through \eqref{def_concat}, where $u_N$, $v_N$, and $w_N$ are replaced by $\tilde u_N$, $\tilde v_N$, and $\tilde w_N$, respectively) and to identify their limits.

\begin{corollary}\label{cor:linf_convergence}
Assume that $\tilde u_N$, $\tilde v_N$, $\tilde w_N$, and $\tilde u$ are given as in Proposition~\ref{prop:point_convergence}. Then
\begin{equation}\label{linf_convergence}
\vertii{\tilde u_N - \tilde u}_{BC^0\left([0,T) \times \TTT_L\right)} \to 0, \quad \vertii{\tilde v_N - \tilde u}_{L^\infty\left([0,T) \times \TTT_L\right)} \to 0, \quad \vertii{\tilde w_N - \tilde u}_{L^\infty\left([0,T) \times \TTT_L\right)} \to 0
\end{equation}
as $N \to \infty$, $\tilde\PP$-almost surely.
\end{corollary}
\begin{proof}[Proof of Corollary~\ref{cor:linf_convergence}]
Since $\XX_u = BC^0\left([0,T) \times \TTT_L\right)$, the first part of \eqref{linf_convergence} is a reformulation of Proposition~\ref{prop:point_convergence}. In view of \eqref{def_concat} this implies
\[
\vertii{\tilde v_N - \tilde u}_{L^\infty([0,T) \times \TTT_L)} \to 0 \quad \mbox{and} \quad \vertii{\tilde w_N - \tilde u}_{L^\infty([0,T) \times \TTT_L)} \to 0 \quad \mbox{as} \quad N \to \infty,
\]
$\tilde \PP$-almost surely. This proves the second and the third limit in \eqref{linf_convergence}.
\end{proof}
\begin{proposition}[weak convergence, identification of limits, a-priori estimate]\label{prop:identify_limit}
Let $\tilde u_N$ and $\tilde u$ be as in Proposition~\ref{prop:point_convergence}. Then, there exist subsequences of $\tilde u_N$, $\tilde v_N$ and $\tilde w_N$, again denoted by $\tilde u_N$, $\tilde v_N$, and $\tilde w_N$, such that for any $p \in [2,\infty)$,
\begin{equation}\label{weak_h1}
\tilde u_N \stackrel{*}{\rightharpoonup} \tilde u, \quad \tilde v_N \stackrel{*}{\rightharpoonup} \tilde u, \quad \mbox{and} \quad \tilde w_N \stackrel{*}{\rightharpoonup} \tilde u \quad \mbox{as} \quad N \to \infty \quad \mbox{in} \quad L^p\left([0,1];L^\infty\left([0,T);H^1(\TTT_L)\right)\right)
\end{equation}
as $N \to \infty$. Furthermore,
\begin{equation}\label{bound_u_space_time}
\tilde\EE \esssup_{t \in [0,T)} \vertii{\tilde u(t,\cdot)}_{1,2}^p \le C \vertii{u_0}_{1,2}^p
\end{equation}
for a constant $C < \infty$ independent of $\tilde u$ and $u_0$. Hence, $\tilde u$ is a bounded continuous $H^1_\mathrm{w}(\TTT_L)$-valued process.
\end{proposition}
\begin{proof}[Proof of Proposition~\ref{prop:identify_limit}]
The existence of subsequences meeting \eqref{weak_h1} follows by compactness, employing the bound \eqref{apriori} of Proposition~\ref{prop:reg_vn_wn}, uniqueness of the limit due to \eqref{linf_convergence} of Corollary~\ref{cor:linf_convergence}, and a diagonal-sequence argument to obtain convergence for all $p \in [2,\infty)$. Because of weak lower-semicontinuity of the norm, estimate~\eqref{apriori} of Proposition~\ref{prop:reg_vn_wn} translates into \eqref{bound_u_space_time}.

\medskip

Since $\tilde u \in L^\infty\left([0,T);H^1(\TTT_L)\right)$, $\tilde \PP$-almost surely, any sequence $(t_j) \in [0,T)$ with $t_j \to t \in [0,T)$ as $j \to \infty$ has a subsequence $(t_j')_j$, such that $\tilde u(t_j',\cdot)$ weak-$*$-converges in $H^1(\TTT_L)$, $\tilde\PP$-almost surely. Since $\tilde u \in BC^0\left([0,T) \times \TTT_L\right)$, $\tilde \PP$-almost surely, the limit is uniquely given by $\tilde u(t,\cdot)$ and thus also $\tilde u\left(t_j,\cdot\right) \stackrel{*}{\rightharpoonup} \tilde u(t,\cdot)$ in $H^1(\TTT_L)$, $\tilde\PP$-almost surely, proving the continuity statement.
\end{proof}

We can also identify the flux density:
\begin{proposition}\label{prop:identify_limit_2}
Let $\tilde u_N$, $\tilde u$, $\tilde J_N$, and $\tilde J$ be as in Proposition~\ref{prop:point_convergence}. Then the distributional derivative $\partial_x^3 \tilde u$ meets $\partial_x^3 \tilde u \in L^2(\{\tilde u > r\})$ for any $r > 0$ and further $\tilde J_N = \ind_{\{\tilde v_N > 0\}} \tilde v_N^2 (\partial_x^3 \tilde v_N)$ and $\tilde J = \ind_{\{\tilde u > 0\}} \tilde u^2 (\partial_x^3 \tilde u)$, $\tilde\PP$-almost surely.
\end{proposition}
\begin{proof}[Proof of Proposition~\ref{prop:identify_limit_2}]
Since by \eqref{equivalence_ujs} of Proposition~\ref{prop:point_convergence} the joint laws coincide, we have for any $\phi \in C^\infty([0,T] \times \TTT_L)$
\begin{align*}
0 &= \EE \verti{\int_0^T \int_0^L J_N \, \phi \, \d x \, \d t - \int_0^T \int_0^L \ind_{\{v_N(t,\cdot) > 0\}} v_N^2 \, (\partial_x^3 v_N) \, \phi \, \d x \, \d t} \\
&= \tilde\EE \verti{\int_0^T \int_0^L \tilde J_N \, \phi \, \d x \, \d t - \int_0^T \int_0^L \ind_{\{\tilde v_N(t,\cdot) > 0\}} \tilde v_N^2 \, (\partial_x^3 \tilde v_N) \, \phi \, \d x \, \d t},
\end{align*}
so that indeed $\tilde J_N = \ind_{\{\tilde v_N > 0\}} \tilde v_N^2 (\partial_x^3 \tilde v_N)$.

\medskip

Because of the a-priori estimate \eqref{apriori} of Proposition~\ref{prop:reg_vn_wn}, we have
\[
\tilde \EE \int_0^T \int_{\left\{\tilde v_N(t,\cdot) > 0\right\}} \tilde v_N^2 \left(\partial_x^3 \tilde v_N\right)^2 \d x \, \d t \le C \vertii{u_0}_{1,2}^2,
\]
where $C < \infty$ only depends on $T$. Hence, for fixed $r > 0$ we obtain
\begin{align*}
& \tilde \EE \int_0^T \int_0^L \left(\partial_x^3 \tilde v_N\right)^2 \ind_{\left\{\vertii{\tilde v_N - \tilde u}_{L^\infty([0,T) \times \TTT_L)} < \frac r 2\right\} \cap \left\{\tilde u > r\right\}} \d x \, \d t \\
& \quad \le \frac{4}{r^2} \, \tilde\EE \int_0^T \int_{\left\{\tilde v_N(t,\cdot) > \frac r 2\right\}} \tilde v_N^2 \, (\partial_x^3 \tilde v_N)^2 \, \d x \, \d t \le \frac{4 C}{r^2} \vertii{u_0}_{1,2}^2,
\end{align*}
so that upon taking a subsequence we obtain by compactness
\begin{equation}\label{partial_2_v_N}
\left(\partial_x^3 \tilde v_N\right) \ind_{\left\{\vertii{\tilde v_N - \tilde u}_{L^\infty([0,T) \times \TTT_L)} < \frac r 2\right\} \cap \left\{\tilde u > r\right\}} \rightharpoonup \tilde \eta \ind_{\left\{\tilde u > r\right\}} \quad \mbox{as} \quad N \to \infty
\end{equation}
in $L^2\left([0,1] \times [0,T) \times \TTT_L\right)$. Taking the limit as $r \searrow 0$, a diagonal-sequence argument implies that, up to taking another subsequence, \eqref{partial_2_v_N} holds true for any $r > 0$. Now, for $\tilde\zeta \in L^2\left([0,1];C^\infty([0,T] \times \TTT_L)\right)$ with $\supp_{(t,x) \in [0,T) \times \TTT_L} \tilde \zeta \Subset \left\{\tilde u > r\right\}$ for all $\omega \in [0,1]$, we have
\begin{align*}
\tilde \EE \int_0^T \int_{\left\{\tilde u(t,\cdot) > r\right\}} \tilde \eta \, \tilde \zeta \, \d x \, \d t &\leftarrow \tilde \EE \int_0^T \int_0^L \left(\partial_x^3 \tilde v_N\right) \ind_{\left\{\vertii{\tilde v_N - \tilde u}_{L^\infty([0,T) \times \TTT_L)} < \frac r 2\right\} \cap \left\{\tilde u > r\right\}} \, \tilde \zeta \, \d x \, \d t \\
&= - \tilde \EE \int_0^T \int_0^L \tilde v_N \, \ind_{\left\{\vertii{\tilde v_N - \tilde u}_{L^\infty([0,T) \times \TTT_L)} < \frac r 2\right\} \cap \left\{\tilde u > r\right\}} \, (\partial_x^3 \tilde \zeta) \, \d x \, \d t \\
&\rightarrow - \tilde \EE \int_0^T \int_{\left\{\tilde u(t,\cdot) > r\right\}} \tilde u \, (\partial_x^3 \tilde \zeta) \, \d x \, \d t
\end{align*}
as $N \to \infty$ for any $r > 0$ by using Vitali's convergence theorem in the last line. Application of the latter relies on \eqref{linf_convergence} of Corollary~\ref{cor:linf_convergence} and
\begin{eqnarray*}
\tilde\EE \int_0^T \int_0^L \verti{\tilde v_N}^{\frac 3 2} \verti{\partial_x^3 \tilde\zeta}^{\frac 3 2} \, \d x \, \d t &\le& \left(\tilde\EE \int_0^T \int_0^L \verti{\tilde v_N}^6 \d x \, \d t\right)^{\frac 1 4} \left(\tilde\EE \int_0^T \int_0^L (\partial_x^3 \tilde \zeta)^2 \, \d x \, \d t\right)^{\frac 3 4} \\
&\le& C \, T^{\frac 1 4} \, \left(\tilde\EE \esssup_{t \in [0,T)} \vertii{\tilde v_N(t,\cdot)}_{1,2}^6\right)^{\frac 1 4} \stackrel{\eqref{apriori}}{\le} C \vertii{u_0}_{1,2}^{\frac 3 2},
\end{eqnarray*}
where $C < \infty$ is independent of $N$ and H\"older's inequality and Proposition~\ref{prop:reg_vn_wn} have been used. Hence, we obtain $\tilde \eta = \partial_x^3 \tilde u$ distributionally on $\{\tilde u > 0\}$. For $\tilde\phi \in L^\infty\left([0,1] \times [0,T] \times \TTT_L\right)$ and $N$ sufficiently large, we may split up according to
\begin{eqnarray} \label{splitting}
\lefteqn{\tilde \EE \int_0^T \int_{\left\{\tilde v_N(t,\cdot) > 0\right\}} \tilde v_N^2 \, (\partial_x^3 \tilde v_N) \, \tilde \phi \, \d x \, \d t} \\
&=& \tilde \EE \int_0^T \int_0^L \tilde v_N^2 \, (\partial_x^3 \tilde v_N) \, \ind_{\left\{\vertii{\tilde v_N - \tilde u}_{L^\infty([0,T) \times \TTT_L)} < \frac r 2\right\} \cap \left\{\tilde u > r\right\}} \, \tilde \phi \, \d x \, \d t \nonumber \\
&& + \tilde \EE \int_0^T \int_{\left\{\tilde v_N(t,\cdot) > 0\right\}} \tilde v_N^2 (\partial_x^3 \tilde v_N) \, \ind_{\left\{\vertii{\tilde v_N - \tilde u}_{L^\infty([0,T) \times \TTT_L)} \ge \frac r 2\right\} \cup \left\{\tilde u \le r\right\}} \, \tilde \phi \, \d x \, \d t. \nonumber
\end{eqnarray}
Since by Proposition~\ref{prop:reg_vn_wn} and Sobolev embedding
\[
\tilde\EE \int_0^T \int_0^L \left(\tilde v_N - \tilde u\right)^6 \d x \, \d t \le C \left(\tilde\EE\esssup_{t \in [0,T)} \vertii{\tilde v_N(t,\cdot)}_{1,2}^6 + \tilde\EE\esssup_{t \in [0,T)} \vertii{\tilde u(t,\cdot)}_{1,2}^6\right) \stackrel{\eqref{apriori}}{\le} C \vertii{u_0}_{1,2}^6
\]
and $\vertii{\tilde v_N - \tilde u}_{L^\infty\left([0,T) \times \TTT_L\right)} \to 0$ as $N \to \infty$, $\tilde\PP$-almost surely, by \eqref{linf_convergence} of Corollary~\ref{cor:linf_convergence}, it follows by Vitali's convergence theorem that
\begin{equation}\label{l2_strong}
\tilde\EE \int_0^T \int_0^L \left(\tilde v_N - \tilde u\right)^4 \d x \, \d t \to 0 \quad \mbox{as} \quad N \to \infty.
\end{equation}
Hence, we obtain
\begin{eqnarray}
\lefteqn{\tilde \EE \int_0^T \int_0^L \tilde v_N^2 \, (\partial_x^3 \tilde v_N) \, \ind_{\left\{\vertii{\tilde v_N - \tilde u}_{L^\infty([0,T) \times \TTT_L)} < \frac r 2\right\} \cap \left\{\tilde u > r\right\}} \, \tilde \phi \, \d x \, \d t} \label{split_converge_1} \\
&\to& \tilde \EE \int_0^T \int_{\left\{\tilde u(t,\cdot) > r \right\}} \tilde u^2 \, (\partial_x^3 \tilde u) \, \tilde \phi \, \d x \, \d t \quad \mbox{as} \quad N \to \infty \nonumber
\end{eqnarray}
because of \eqref{partial_2_v_N} and \eqref{l2_strong}. Furthermore,
\begin{eqnarray}\label{split_converge_2}
\lefteqn{\verti{\tilde \EE \int_0^T \int_{\left\{\tilde v_N(t,\cdot) > 0\right\}} \tilde v_N^2 \, (\partial_x^3 \tilde v_N) \, \ind_{\left\{\vertii{\tilde v_N - \tilde u}_{L^\infty([0,T) \times \TTT_L)} \ge \frac r 2\right\} \cup \left\{\tilde u \le r\right\}} \, \tilde \phi \, \d x \, \d t}} \\
&\le& C \, \vertii{\tilde \phi}_{L^\infty([0,1] \times [0,T) \times \TTT_L)} \nonumber \\
&& \times \left(\tilde \EE \int_0^T \int_{\left\{\tilde u(t,\cdot) \le r\right\} \cap \left\{\tilde v_N(t,\cdot) > 0\right\}} \tilde v_N^2 \, (\partial_x^3 \tilde v_N)^2 \, \d x \, \d t\right)^{\frac 1 2} \nonumber\\
&& \times \left(\tilde\EE \int_0^T \int_{\left\{\tilde v_N(t,\cdot) > 0\right\}} (\tilde v_N)^2 \ind_{\left\{\vertii{\tilde v_N - \tilde u}_{L^\infty([0,T) \times \TTT_L)} \ge \frac r 2\right\} \cup \{\tilde u \le r\}} \, \d x \, \d t\right)^{\frac 1 2} \nonumber\\
&\stackrel{\eqref{apriori}}{\le}& C \, \vertii{\tilde \phi}_{L^\infty([0,1] \times [0,T) \times \TTT_L)} \vertii{u_0}_{1,2} \nonumber \\
&& \times \left(\tilde\EE \int_0^T \int_{\left\{\tilde v_N(t,\cdot) > 0\right\}} (\tilde v_N)^2 \ind_{\left\{\vertii{\tilde v_N - \tilde u}_{L^\infty([0,T) \times \TTT_L)} \ge \frac r 2\right\} \cup \{\tilde u \le r\}} \, \d x \, \d t\right)^{\frac 1 2}, \nonumber
\end{eqnarray}
where $C < \infty$ is independent of $N$ and Proposition~\ref{prop:reg_vn_wn} has been applied. Now, we note that by Sobolev embedding
\begin{eqnarray*}
\lefteqn{\tilde\EE \int_0^T \int_{\left\{\tilde v_N(t,\cdot) > 0\right\}} (\tilde v_N)^4 \ind_{\left\{\vertii{\tilde v_N - \tilde u}_{L^\infty([0,T) \times \TTT_L)} \ge \frac r 2\right\} \cup \{\tilde u \le r\}} \, \d x \, \d t} \\
&\le& C \, \tilde\EE\esssup_{t \in [0,T)} \vertii{\tilde v_N(t,\cdot)}_{1,2}^4 \stackrel{\eqref{apriori}}{\le} C \vertii{u_0}_{1,2}^4,
\end{eqnarray*}
where $C < \infty$ is independent of $N$ and, so that by \eqref{linf_convergence} of Corollary~\ref{cor:linf_convergence} we have by Vitali's convergence theorem
\begin{align*}
& \tilde\EE \int_0^T \int_{\left\{\tilde v_N(t,\cdot) > 0\right\}} (\tilde v_N)^2 \ind_{\left\{\vertii{\tilde v_N - \tilde u}_{L^\infty([0,T) \times \TTT_L)} \ge \frac r 2\right\} \cup \{\tilde u \le r\}} \, \d x \, \d t \\
& \quad \to \tilde \EE \int_0^T \int_0^L \tilde u^2 \, \ind_{\{\tilde u \le r\}} \, \d x \, \d t = O\left(r^2\right) \quad \mbox{as} \quad N \to \infty
\end{align*}
and \eqref{split_converge_2} implies
\begin{equation}\label{split_converge_3}
\tilde \EE \int_0^T \int_{\left\{\tilde v_N(t,\cdot) > 0\right\}} \tilde v_N^2 \, (\partial_x^3 \tilde v_N) \, \ind_{\left\{\vertii{\tilde v_N - \tilde u}_{L^\infty([0,T) \times \TTT_L)} \ge \frac r 2\right\} \cup \left\{\tilde u \le r\right\}} \, \tilde \phi \, \d x \, \d t = O(r) \quad \mbox{as} \quad N \to \infty.
\end{equation}
The limits \eqref{split_converge_1} and \eqref{split_converge_3} in \eqref{splitting} lead to
\begin{eqnarray}\label{limit_sequence}
\lefteqn{\tilde\EE \int_0^T \int_{\left\{\tilde v_N(t,\cdot) > 0\right\}} \tilde v_N^2 \, (\partial_x^3 \tilde v_N) \, \tilde \phi \, \d x \, \d t} \\
&=& \tilde\EE \int_0^T \int_{\left\{\tilde u(t,\cdot) > r\right\}} \tilde u^2 \, (\partial_x^3 \tilde u) \, \tilde \phi \, \d x \, \d t + O(r) \quad \mbox{as} \quad N \to \infty \nonumber \\
&\rightarrow& \tilde\EE \int_0^T \int_{\left\{\tilde u(t,\cdot) > 0\right\}} \tilde u^2 \, (\partial_x^3 \tilde u) \, \tilde \phi \, \d x \, \d t \quad \mbox{as} \quad r \searrow 0. \nonumber
\end{eqnarray}
The last step follows by dominated convergence, where we have employed that the integrand is absolutely integrable. The latter follows from
\[
\tilde u^2 \verti{\partial_x^3 \tilde u} \verti{\tilde\phi} \le \tilde u^2 \verti{\partial_x^3 \tilde u} \vertii{\tilde\phi}_{L^\infty([0,1] \times [0,T) \times \TTT_L)}
\]
and the fact that by monotone convergence, the first two lines of \eqref{limit_sequence}, and the Sobolev embedding theorem,
\begin{eqnarray*}
\lefteqn{\tilde\EE \int_0^T \int_{\left\{\tilde u(t,\cdot) > 0\right\}} \tilde u^2 \verti{\partial_x^3 \tilde u} \d x \, \d t} \\
&=& \lim_{r \searrow 0} \tilde\EE \int_0^T \int_{\left\{\tilde u(t,\cdot) > r\right\}} \tilde u^2 \, (\partial_x^3 \tilde u) \left(\ind_{\{\partial_x^3 \tilde u > 0\}} - \ind_{\{\partial_x^3 \tilde u < 0\}}\right) \d x \, \d t \\
&=& \lim_{N \to \infty} \tilde\EE \int_0^T \int_{\left\{\tilde v_N(t,\cdot) > 0\right\}} \tilde v_N^2 \, (\partial_x^3 \tilde v_N) \left(\ind_{\{\partial_x^3 \tilde u > 0\}} - \ind_{\{\partial_x^3 \tilde u < 0\}}\right) \d x \, \d t \\
&\le& C \lim_{N \to \infty} \left(\tilde\EE \int_0^T \vertii{\tilde v_N(t,\cdot)}_{1,2}^2 \int_{\left\{\tilde v_N(t,\cdot) > 0\right\}} \tilde v_N^2 \, (\partial_x^3 \tilde v_N)^2 \, \d x \, \d t\right)^{\frac 1 2} \stackrel{\eqref{apriori}}{\le} C \vertii{u_0}_{1,2}^2,
\end{eqnarray*}
where $C < \infty$ and Proposition~\ref{prop:reg_vn_wn} was used in the last step.

\medskip

From \eqref{limit_sequence} it follows that
\[
\tilde J_N = \ind_{\{\tilde v_N > 0\}} \, \tilde v_N^2 \, (\partial_x^3 \tilde v_N) \rightharpoonup \ind_{\{\tilde u > 0\}} \, \tilde u^2 \, (\partial_x^3 \tilde u) \quad \mbox{in $L^1\left([0,1] \times [0,T) \times \TTT_L\right)$ as $N \to \infty$},
\]
which together with $\tilde J_N \rightharpoonup \tilde J$ in $L^2\left([0,T) \times \TTT_L\right)$ as $N \to \infty$, $\tilde\PP$-almost surely, implies $\tilde J = \ind_{\{\tilde u > 0\}} \, \tilde u^2 \, (\partial_x^3 \tilde u)$.
\end{proof}
%

\subsection{Recovering the SPDE\label{sec:limit}}
From the scheme \ref{item:deterministic}--\ref{item:stochastic}--\ref{item:deterministic_stochastic} we deduce for $t \in [0,T)$ and recalling $\delta = \frac{T}{N+1}$
\begin{eqnarray*}
\lefteqn{\left(v_N(t,\cdot),\varphi\right)_2 - \left(u_0,\varphi\right)_2} \\
&\stackrel{\eqref{connection}}{=}& \left(v_N(t,\cdot),\varphi\right)_2 + \sum_{j = 1}^{\floor{\frac t \delta}} \left(- \left(v_N(j \delta,\cdot),\varphi\right)_2 + \lim_{t' \nearrow j \delta} \left(w_N(t',\cdot),\varphi\right)_2\right) \\
&& + \sum_{j = 1}^{\floor{\frac t \delta}} \left(\lim_{t' \nearrow j \delta} \left(v_N(t',\cdot),\varphi\right)_2 - \left(w_N((j-1)\delta,\cdot),\varphi\right)_2\right) - \left(v_N(0,\cdot),\varphi\right)_2 \\
&=& \left(v_N(t,\cdot),\varphi\right)_2 - \left(v_N\left(\floor{\tfrac t \delta} \delta,\cdot\right),\varphi\right)_2 \\
&& + \sum_{j = 1}^{\floor{\frac t \delta}} \left(\lim_{t' \nearrow j \delta}\left(v_N(t',\cdot),\varphi\right)_2 - \left(v_N((j-1)\delta,\cdot),\varphi\right)_2\right) \\
&& + \sum_{j = 1}^{\floor{\frac t \delta}} \left(\lim_{t' \nearrow j \delta}\left(w_N(t',\cdot),\varphi\right)_2 - \left(w_N((j-1)\delta,\cdot),\varphi\right)_2\right) \\
&\stackrel{\eqref{deterministic}, \eqref{stochastic}}{=}& \int_0^t \int_{\{v_N(t',\cdot) > 0\}} v_N^2 (\partial_x^3 v_N) \, (\partial_x \varphi) \, \d x \, \d t' \\
&& - \frac 1 2 \sum_{k \in \Z} \lambda_k^2 \int_0^{\floor{\frac t \delta} \delta} \left(\psi_k \partial_x \left(\psi_k w_N(t',\cdot)\right), \partial_x\varphi\right)_2 \d t' \\
&& - \sum_{k \in \Z} \lambda_k \int_0^{\floor{\frac t \delta} \delta} \left(\psi_k w_N(t',\cdot),\partial_x\varphi\right)_2 \d \beta^k(t'),
\end{eqnarray*}
$\tilde\PP$-almost surely, where $\varphi \in C^\infty(\TTT_L)$ is a test function. Note that equations~\eqref{deterministic} and \eqref{stochastic} follow rigorously from \eqref{test_tfe} of Theorem~\ref{th:tfe_weak} and \eqref{solution_formula} of Definition~\ref{def:sol_eps_0} tested against $\varphi$. Changing the stochastic basis to
\[
\left([0,1],\tilde\cF,\left(\tilde\cF_t\right)_{t \in [0,T)},\tilde\PP\right),
\]
we obtain for the in law equivalent convergent subsequences $\tilde u_N$, $\tilde v_N$, and $\tilde w_N$ for $t \in [0,T)$ by taking \eqref{def_noise}, \eqref{def_concat}, \eqref{def_noise_n}, and \eqref{def_beta_k} into account,
\begin{eqnarray}\label{eq_step_n}
\left(\tilde v_N(t,\cdot),\varphi\right)_2 - \left(u_0,\varphi\right)_2 &=& \int_0^t \int_{\left\{\tilde v_N(t',\cdot) > 0\right\}} \tilde v_N^2 \, (\partial_x^3 \tilde v_N) \, \partial_x\varphi \, \d x \, \d t' \\
&& - \frac 1 2 \sum_{k \in \Z} \lambda_k^2 \int_0^{\floor{\frac t \delta} \delta} \left(\psi_k \partial_x \left(\psi_k \tilde w_N(t',\cdot)\right),\partial_x \varphi\right)_2 \d t' \nonumber\\
&& - \sum_{k \in \Z} \lambda_k \int_0^{\floor{\frac t \delta} \delta} \left(\psi_k \tilde u_N(t',\cdot),\partial_x \varphi\right)_2 \d\tilde\beta^k_N(t'). \nonumber
\end{eqnarray}
Passing to the limit as $N \to \infty$, we obtain the main result, Theorem~\ref{th:main}, by applying Propositions~\ref{prop:point_convergence}, \ref{prop:brownian}, and \ref{prop:identify_limit}, and showing that the different terms appearing in \eqref{eq_step_n} converge in the sense stated in the next lemma:
\begin{lemma}\label{lem:conv_terms}
Assume that $\tilde u_N$, $\tilde v_N$, $\tilde w_N$, $\tilde u$, $\tilde v$, and $\tilde w$ are given as in Proposition~\ref{prop:point_convergence} and \ref{prop:identify_limit}. Then, for any $\varphi \in C^\infty(\TTT_L)$ and $t \in [0,T)$, and up to taking subsequences, we have
\begin{subequations}
\begin{align}
\left(\tilde v_N(t,\cdot),\varphi\right)_2 &\to \left(\tilde u(t,\cdot),\varphi\right)_2, \label{lim_eq_t1}\\
\int_0^t \int_{\left\{\tilde v_N(t',\cdot) > 0\right\}} \tilde v_N^2 \left(\partial_x^3 \tilde v_N\right) (\partial_x\varphi) \, \d x \, \d t' &\to \int_0^t \int_{\{\tilde u(t',\cdot) > 0\}} \tilde u^2 \left(\partial_x^3 \tilde u\right) (\partial_x\varphi) \, \d x \, \d t', \label{lim_eq_t2} \\
\sum_{k \in \Z} \lambda_k^2 \int_0^{\floor{\frac t \delta} \delta} \left(\psi_k \partial_x \left(\psi_k \tilde w_N(t',\cdot)\right),\partial_x \varphi\right)_2 \d t' &\to \sum_{k \in \Z} \lambda_k^2 \int_0^t \left(\psi_k \partial_x \left(\psi_k \tilde u(t',\cdot)\right),\partial_x \varphi\right)_2 \d t', \label{lim_eq_t3}\\
\sum_{k \in \Z} \lambda_k \int_0^{\floor{\frac t \delta} \delta} \left(\psi_k \tilde u_N(t',\cdot),\partial_x \varphi\right)_2 \d\tilde\beta^k_N(t') &\to \sum_{k \in \Z} \lambda_k \int_0^t \left(\psi_k \tilde u(t',\cdot),\partial_x \varphi\right)_2 \d\tilde\beta^k(t') \label{lim_eq_t4}
\end{align}
\end{subequations}
as $N \to \infty$, $\tilde\PP$-almost surely.
\end{lemma}
\begin{proof}[Proof of Lemma~\ref{lem:conv_terms}]
We prove each limit separately:
\proofstep{Proof of \eqref{lim_eq_t1}}
Since by \eqref{linf_convergence} of Corollary~\ref{cor:linf_convergence} we have $\vertii{\tilde v_N - \tilde u}_{L^\infty([0,T) \times \TTT_L)} \to 0$ as $N \to \infty$, $\tilde\PP$-almost surely, and $\tilde v_N$ is $\tilde\PP$-almost surely piece-wise continuous in time (cf.~Theorem~\ref{th:tfe_weak}~\eqref{item:bf1}), it holds
\[
\vertii{\tilde v_N - \tilde u}_{BC^0([0,T) \times \TTT_L)} = \vertii{\tilde v_N - \tilde u}_{L^\infty([0,T) \times \TTT_L)} \to 0 \quad \mbox{as} \quad N \to \infty.
\]
Hence, we obtain by bounded convergence that $\left(\tilde v_N(t,\cdot),\varphi\right)_2 \to \left(\tilde u(t,\cdot),\varphi\right)_2$ as $N \to \infty$ for $t \in [0,T)$, $\tilde \PP$-almost surely, proving \eqref{lim_eq_t1}.

\proofstep{Proof of \eqref{lim_eq_t2}}
The limit \eqref{lim_eq_t2} immediately follows from the weak convergence of the flux density $\tilde J_N$ stated in Proposition~\ref{prop:point_convergence}, i.e., $\tilde J_N = \ind_{\left\{\tilde v_N > 0\right\}}\tilde v_N^2 \left(\partial_x^3 v_N\right) \rightharpoonup \tilde J$ in $L^2\left([0,T) \times \TTT_L\right)$, $\PP$-almost surely, and the identification of the limit $\tilde J = \ind_{\left\{\tilde u > 0\right\}}\tilde u^2 \left(\partial_x^3 \tilde u\right)$ given in Proposition~\ref{prop:identify_limit_2}.

\proofstep{Proof of \eqref{lim_eq_t3}}
We have by \eqref{basis}, \eqref{der_psi_k}, \eqref{cond_lambda},
\eqref{linf_convergence} of Corollary~\ref{cor:linf_convergence}, and bounded convergence,
\begin{align*}
& \sum_{k \in \Z} \lambda_k^2 \int_0^{\floor{\frac t \delta} \delta} \left(\psi_k \partial_x \left(\psi_k \tilde w_N(t',\cdot)\right),\partial_x \varphi\right)_2 \d t' \\
&\quad = - \sum_{k \in \Z} \lambda_k^2 \int_0^{\floor{\frac t \delta} \delta} \left(\tilde w_N(t',\cdot),\psi_k \partial_x \left(\psi_k \partial_x \varphi\right)\right)_2 \d t' \\
& \quad \to - \sum_{k \in \Z} \lambda_k^2 \int_0^t \left(\tilde u(t',\cdot),\psi_k \partial_x \left(\psi_k \partial_x \varphi\right)\right)_2 \d t' \quad \mbox{as} \quad N \to \infty, \quad \mbox{$\tilde \PP$-almost surely} \\
& \quad = \sum_{k \in \Z} \lambda_k^2 \int_0^t \left(\psi_k \partial_x \left(\psi_k \tilde u(t',\cdot)\right),\partial_x \varphi\right)_2 \d t',
\end{align*}
proving \eqref{lim_eq_t3}.

\proofstep{Proof of \eqref{lim_eq_t4}}
For $\varphi \in C^\infty(\TTT_L)$ and $t \in [0,T)$, we define 
\begin{eqnarray}\label{def_mart_n}
M_{N,\varphi}(t) &:=& - \sum_{k \in \Z} \lambda_k \int_0^{\floor{\frac t \delta} \delta} \left(\psi_k \tilde u_N(t',\cdot),\partial_x \varphi\right)_2 \d\tilde\beta^k_N(t') \\
&\stackrel{\eqref{eq_step_n}}{=}& \left(\tilde v_N(t,\cdot),\varphi\right)_2 - \left(u_0,\varphi\right)_2 - \int_0^t \int_{\{\tilde v_N(t',\cdot) > 0\}} \tilde v_N^2 (\partial_x^3 \tilde v_N) \, (\partial_x\varphi) \, \d x \, \d t' \nonumber \\
&& + \frac 1 2 \sum_{k \in \Z} \lambda_k^2 \int_0^{\floor{\frac t \delta} \delta} \left(\psi_k \partial_x \left(\psi_k \tilde w_N(t',\cdot)\right),\partial_x \varphi\right)_2 \d t'. \nonumber
\end{eqnarray}
Note that $\tilde u_N$ and $\tilde\beta_N^k$ are adapted to $\tilde\cF_{N,t} := \sigma\left(\tilde u_N(t',\cdot), \tilde W_N(t',\cdot) \colon 0 \le t' \le t\right)$ (We do not need to include $\tilde J_N$ in view of Proposition~\ref{prop:identify_limit_2}.). In view of \eqref{def_concat}, \eqref{def_noise_n}, Proposition~\ref{prop:point_convergence}, and \eqref{def_t_beta_k_N}, we obtain for the quadratic variation process
\begin{eqnarray*}
\qvar{\tilde M_{N,\varphi}}{t} &=& \sum_{k \in \Z} \lambda_k^2 \int_0^{\floor{\frac t \delta} \delta} \left(\psi_k \tilde w_N(t',\cdot),\partial_x\varphi\right)_2^2 \d t' \\
&\le& \vertii{\partial_x \varphi}_2^2 \left(\sum_{k \in \Z} \lambda_k^2 \vertii{\psi_k}_{L^\infty(\TTT_L)}^2\right) \int_0^{\floor{\frac t \delta} \delta} \vertii{\tilde w_N(t',\cdot)}_2^2 \d t' \\
&\stackrel{\eqref{basis}, \eqref{cond_lambda}}{\le}& C \vertii{\partial_x \varphi}_2^2 \int_0^{\floor{\frac t \delta} \delta} \vertii{\tilde w_N(t',\cdot)}_2^2 \d t',
\end{eqnarray*}
so that
\begin{eqnarray}\label{q_var_n}
\tilde\EE \left(\qvar{\tilde M_{N,\varphi}}{t}\right)^q &\le& C \, t^q \vertii{\partial_x \varphi}_2^{2q} \esssup_{t' \in \left[0,T\right)} \tilde \EE \vertii{\tilde w_N(t',\cdot)}_2^{2q} \\
&\stackrel{\eqref{apriori}}{\le}& C \, t^q \vertii{\partial_x \varphi}_2^{2q} \vertii{u_0}_{1,2}^{2q} \quad \mbox{for} \quad q \ge 1, \nonumber
\end{eqnarray}
where $C < \infty$ is independent of $N$ and Proposition~\ref{prop:reg_vn_wn} has been applied. Hence, $\tilde M_{N,\varphi}$ is a square-integrable martingale with respect to $(\tilde\cF_{N,t})_{t \in [0,T)}$. We know from \eqref{lim_eq_t1}--\eqref{lim_eq_t3} that, for all $t \in [0,T)$,
\begin{equation}\label{lim_m_1}
\begin{aligned}
\tilde M_{N,\varphi}(t) \to \tilde M_\varphi(t) &:= \left(\tilde u(t,\cdot),\varphi\right)_2 - \left(u_0,\varphi\right)_2 - \int_0^t \int_{\left\{\tilde u(t',\cdot) > 0\right\}} \tilde u^2 \, (\partial_x^3 \tilde u) \, (\partial_x\varphi) \, \d x \, \d t' \\
&\phantom{:=} + \frac 1 2 \sum_{k \in \Z} \lambda_k^2 \int_0^t \left(\psi_k \partial_x \left(\psi_k \tilde u(t',\cdot)\right),\partial_x \varphi\right)_2 \d t' \quad \mbox{as} \quad N \to \infty,
\end{aligned}
\end{equation}
$\tilde \PP$-almost surely. Then, it suffices to show that, for all $t \in [0,T)$,
\begin{equation}\label{lim_m_2}
\tilde M_\varphi(t) = - \sum_{k \in \Z} \lambda_k \int_0^t \left(\psi_k \tilde u(t',\cdot),\partial_x\varphi\right)_2 \d\tilde\beta^k(t').
\end{equation}
Since $\tilde M_{N,\varphi}$ is a square-integrable $(\tilde\cF_{N,t})$-martingale, we have for $0 \le t' \le t < T$, and
\[
\Phi \in C^0\left(\left.\XX_u\right|_{[0,t']} \times \left.\XX_W\right|_{[0,t']};[0,1]\right)
\]
as in \eqref{space_x_u} and \eqref{space_x_w} of Proposition~\ref{prop:point_convergence} (Again, it is not necessary to include $\XX_J$ because of Proposition~\ref{prop:identify_limit_2}.) the identities 
\begin{subequations}\label{id_mart_n}
\begin{align}
\tilde \EE \left[\left(\tilde M_{N,\varphi}(t) - \tilde M_{N,\varphi}(t')\right) \tilde \Phi_N\right] &= 0, \label{id_mart_n_1}\\
\tilde \EE \left[\left(\left(\tilde M_{N,\varphi}(t)\right)^2 - \left(\tilde M_{N,\varphi}(t')\right)^2 - \sum_{k \in \Z} \lambda_k^2 \int_{\floor{\frac{t'}{\delta}} \delta}^{\floor{\frac t \delta} \delta} \left(\psi_k \tilde w_N(t'',\cdot),\partial_x \varphi\right)_2^2 \d t'' \right) \tilde \Phi_N\right] &= 0, \label{id_mart_n_2}\\
\tilde \EE \left[\left(\tilde\beta^k_N(t) \tilde M_{N,\varphi}(t) - \tilde\beta^k_N(t') \tilde M_{N,\varphi}(t') + \lambda_k \int_{\floor{\frac{t'}{\delta}} \delta}^{\floor{\frac t \delta} \delta} \left(\psi_k \tilde w_N(t'',\cdot),\partial_x\varphi\right)_2 \d t''\right) \tilde \Phi_N\right] &= 0, \label{id_mart_n_3}
\end{align}
where
\begin{equation}
\tilde \Phi_N := \Phi\left(\left. \tilde u_N \right|_{[0,t']}, \left. \tilde W_N\right|_{[0,t']}\right).
\end{equation}
\end{subequations}
We derive below that, in the limit as $N \to \infty$, we have for $0 \le t' \le t < T$
\begin{subequations}\label{id_mart}
\begin{align}
\tilde \EE \left[\left(\tilde M_\varphi(t) - \tilde M_\varphi(t')\right) \tilde \Phi\right] &= 0, \label{id_mart_1}\\
\tilde \EE \left[\left(\left(\tilde M_\varphi(t)\right)^2 - \left(\tilde M_\varphi(t')\right)^2 - \sum_{k \in \Z} \lambda_k^2 \int_{t'}^t \left(\psi_k \tilde u(t'',\cdot),\partial_x \varphi\right)_2^2 \d t'' \right) \tilde \Phi\right] &= 0, \label{id_mart_2}\\
\tilde \EE \left[\left(\tilde\beta^k(t) \tilde M_\varphi(t) - \tilde\beta^k(t') \tilde M_\varphi(t') + \lambda_k \int_{t'}^t \left(\psi_k \tilde u(t'',\cdot),\partial_x\varphi\right)_2 \d t''\right) \tilde \Phi\right] &= 0, \label{id_mart_3}
\end{align}
where
\begin{equation}
\tilde \Phi = \Phi\left(\left. \tilde u \right|_{\left[0,t'\right]}, \left. \tilde W\right|_{\left[0,t'\right]}\right).
\end{equation}
\end{subequations}
With the same argumentation as in the proof of Proposition~\ref{prop:brownian}, we may then infer that $\tilde M_\varphi$ is also an $(\tilde\cF_t)$-martingale. Hence, \eqref{lim_m_2} follows from \eqref{basis}, \eqref{cond_lambda}, and \cite[Proposition~A.1]{Hofmanova2013} or \cite{KrylovRozovskii1979}.

\medskip

In order to prove \eqref{id_mart}, we note that
\begin{equation}\label{conv_phi_n}
\verti{\tilde\Phi_N} \le 1 \quad \mbox{and} \quad \tilde \Phi_N \to \tilde \Phi \quad \mbox{as $N \to \infty$ point-wise, $\tilde\PP$-almost surely.}
\end{equation}

\medskip

\proofstep{Argument for \eqref{id_mart_1}}
From \eqref{def_mart_n} and \eqref{id_mart_n_1} we deduce
\begin{align*}
0 &= \tilde \EE \left[\left( \left(\tilde v_N(t,\cdot) - \tilde v_N(t',\cdot),\varphi\right)_2 - \int_{t'}^t \int_{\left\{\tilde v_N(t'',\cdot) > 0\right\}} \tilde v_N^2 (\partial_x^3 \tilde v_N) (\partial_x \varphi) \, \d x \, \d t''\right) \tilde \Phi_N\right] \\
&\phantom{=} + \frac 1 2 \tilde \EE \left[ \left(\sum_{k \in \Z} \lambda_k^2 \int_{\floor{\frac{t'}{\delta}} \delta}^{\floor{\frac t \delta} \delta} \left(\psi_k \partial_x \left(\psi_k \tilde w_N(t'',\cdot)\right),\partial_x \varphi\right)_2 \d t''\right) \tilde \Phi_N\right].
\end{align*}
Then, we note that
\begin{equation}\label{pr_mart_1_1}
\tilde \EE \left[ \left(\tilde v_N(t,\cdot) - \tilde v_N(t',\cdot),\varphi\right)_2 \tilde \Phi_N\right] \to \tilde \EE \left[ \left(\tilde u(t,\cdot) - \tilde u(t',\cdot),\varphi\right)_2 \tilde \Phi\right] \quad \mbox{as} \quad N \to \infty.
\end{equation}
Indeed, from \eqref{linf_convergence} of Corollary~\ref{cor:linf_convergence} and piece-wise continuity in time by \eqref{def_concat}, we infer
\[
\vertii{\tilde v_N(t,\cdot) - \tilde v_N(t',\cdot) - \tilde u(t,\cdot) + \tilde u(t',\cdot)}_{L^\infty(\TTT_L)} \to 0 \quad \mbox{as} \quad N \to \infty,
\]
$\tilde \PP$-almost surely,
\[
\tilde \EE \left[\verti{\left(\tilde v_N(t,\cdot) - \tilde v_N(t',\cdot),\varphi\right)_2}^2 (\tilde \Phi_N)^2\right] \stackrel{\eqref{conv_phi_n}}{\le} 4 \esssup_{t'' \in [0,T)} \tilde \EE \vertii{\tilde v_N(t'',\cdot)}_2^2 \vertii{\varphi}_2^2 \stackrel{\eqref{apriori}}{\le} C \vertii{u_0}_{1,2}^2 \vertii{\varphi}_2^2,
\]
where $C < \infty$ is independent of $N$ and Proposition~\ref{prop:reg_vn_wn} has been applied, so that with \eqref{conv_phi_n} the claim \eqref{pr_mart_1_1} follows by Vitali's convergence theorem.

\medskip

We argue again by Vitali's convergence theorem to infer that
\begin{equation}\label{pr_mart_1_2}
\tilde \EE \left[\int_{t'}^t \int_{\left\{\tilde v_N(t'',\cdot) > 0\right\}} \tilde v_N^2 (\partial_x^3 \tilde v_N) (\partial_x \varphi) \, \d x \, \d t'' \, \tilde \Phi_N\right] \to \tilde \EE \left[\int_{t'}^t \int_{\left\{\tilde u(t'',\cdot) > 0\right\}} \tilde u^2 (\partial_x^3 \tilde u) (\partial_x \varphi) \, \d x \, \d t'' \, \tilde \Phi\right]
\end{equation}
as $N \to \infty$. Indeed, this follows from \eqref{lim_eq_t2}, \eqref{conv_phi_n}, and
\begin{eqnarray*}
\lefteqn{\tilde\EE \left[\verti{\int_{t'}^t \int_{\{\tilde v_N(t'',\cdot) > 0\}} \tilde v_N^2 \, (\partial_x^3 \tilde v_N) \, (\partial_x\varphi) \, \d x \, \d t''}^2 (\tilde\Phi_N)^2\right]} \\
&\le& C \vertii{\partial_x\varphi}_2^2 \, \tilde\EE \int_0^T \vertii{\tilde v_N(t'',\cdot)}_{1,2}^2 \int_{\{\tilde v_N(t'',\cdot) > 0\}} \tilde v_N^2 \, (\partial_x^3 \tilde v_N)^2 \, \d x \, \d t'' \\
&\stackrel{\eqref{apriori}}{\le}& C \vertii{\partial_x \varphi}_2^2 \vertii{u_0}_{1,2}^4,
\end{eqnarray*}
where $C < \infty$ is independent of $N$ and Proposition~\ref{prop:reg_vn_wn} has been applied.

\medskip

Finally, using \eqref{basis}, \eqref{der_psi_k_1}, \eqref{cond_lambda}, Proposition~\ref{prop:reg_vn_wn}, \eqref{linf_convergence} of Corollary~\ref{cor:linf_convergence}, \eqref{conv_phi_n}, and Vitali's convergence theorem, we have
\begin{align*}
& \tilde \EE \left[ \left(\sum_{k \in \Z} \lambda_k^2 \int_{\floor{\frac{t'}{\delta}} \delta}^{\floor{\frac t \delta} \delta} \left(\psi_k \partial_x \left(\psi_k \tilde w_N(t',\cdot)\right),\partial_x \varphi\right)_2 \d t''\right) \tilde \Phi_N\right] \\
& \quad = - \tilde \EE \left[ \left(\sum_{k \in \Z} \lambda_k^2 \int_{\floor{\frac{t'}{\delta}} \delta}^{\floor{\frac t \delta} \delta} \left(\tilde w_N(t',\cdot),\psi_k \partial_x \left(\psi_k \partial_x \varphi\right)\right)_2 \d t''\right) \tilde \Phi_N\right] \\
& \quad \to - \tilde \EE \left[ \left(\sum_{k \in \Z} \lambda_k^2 \int_{t'}^t \left(\tilde u(t',\cdot),\psi_k \partial_x \left(\psi_k \partial_x \varphi\right)\right)_2 \d t''\right) \tilde \Phi\right] \quad \mbox{as} \quad N \to \infty \\
& \quad \phantom{\to} = \EE \left[ \left(\sum_{k \in \Z} \lambda_k^2 \int_{t'}^t \left(\psi_k \partial_x \left(\psi_k \tilde u(t',\cdot)\right), \partial_x \varphi\right)_2 \d t''\right) \tilde \Phi\right].
\end{align*}

Altogether, we infer that taking the limit as $N \to \infty$ in \eqref{id_mart_n_1}, we may conclude that \eqref{id_mart_1} holds true.

\proofstep{Argument for \eqref{id_mart_2}}
First, we note that
\begin{eqnarray*}
\lefteqn{\tilde\EE \left(\sum_{k \in \Z} \lambda_k^2 \int_{\floor{\frac{t'}{\delta}} \delta}^{\floor{\frac t \delta} \delta} \left(\psi_k \tilde w_N(t'',\cdot),\partial_x\varphi\right)_2^2 \d t''\right)^2} \\
&\stackrel{\eqref{basis}, \eqref{cond_lambda}}{\le} C \vertii{\partial_x \varphi}_2^4 \, \esssup_{t'' \in [0,T)} \tilde\EE \vertii{\tilde w_N(t'',\cdot)}_2^4 \stackrel{\eqref{apriori}}{\le} C \vertii{\partial_x\varphi}_2^4 \vertii{u_0}_{1,2}^4,
\end{eqnarray*}
where $C < \infty$ is independent of $N$ and Proposition~\ref{prop:reg_vn_wn} has been utilized. Additionally, by \eqref{linf_convergence} of Corollary~\ref{cor:linf_convergence} and bounded convergence
\[
\sum_{k \in \Z} \lambda_k^2 \int_{\floor{\frac{t'}{\delta}} \delta}^{\floor{\frac t \delta} \delta} \left(\psi_k \tilde w_N(t'',\cdot),\partial_x\varphi\right)_2^2 \d t'' \to \sum_{k \in \Z} \lambda_k^2 \int_{t'}^t \left(\psi_k \tilde u(t'',\cdot),\partial_x\varphi\right)_2^2 \d t'' \quad \mbox{as $N \to \infty$,}
\]
$\tilde\PP$-almost surely. Together with \eqref{conv_phi_n} this implies
\[
\tilde\EE\left[\sum_{k \in \Z} \lambda_k^2 \int_{\floor{\frac{t'}{\delta}} \delta}^{\floor{\frac t \delta} \delta} \left(\psi_k \tilde w_N(t',\cdot),\partial_x\varphi\right)_2^2 \d t'' \, \tilde\Phi_N\right] \to \tilde\EE\left[\sum_{k \in \Z} \lambda_k^2 \int_{t'}^t \left(\psi_k \tilde u(t'',\cdot),\partial_x\varphi\right)_2^2 \d t'' \, \tilde\Phi\right]
\]
as $N \to \infty$ by Vitali's convergence theorem. Now, by \eqref{lim_m_1},
\[
\tilde M_{N,\varphi}(t) \to \tilde M_\varphi(t) \quad \mbox{and} \quad \tilde M_{N,\varphi}(t') \to \tilde M_\varphi(t') \quad \mbox{as $N \to \infty$, $\tilde\PP$-almost surely},
\]
and further applying the Burkholder-Davis-Gundy inequality (cf.~\cite[Theorem~3.28]{KaratzasShreve1991}) gives
\[
\tilde \EE \left[ \left(\tilde M_{N,\varphi}(t)\right)^4 (\tilde\Phi_N)^2\right] \stackrel{\eqref{conv_phi_n}}{\le} C \, \tilde \EE \qvar{\tilde M_{N,\varphi}}{t}^2 \stackrel{\eqref{q_var_n}}\le C \, t^2 \vertii{\partial_x \varphi}_2^4\vertii{u_0}_{1,2}^4,
\]
where $C < \infty$ is independent of $N$, so that by Vitali's convergence theorem
\[
\tilde\EE \left[\left(\tilde M_{N,\varphi}(t)\right)^2 \tilde\Phi_N\right] \to \tilde\EE \left[\left(\tilde M_\varphi(t)\right)^2 \tilde\Phi\right] \quad \mbox{and} \quad \tilde\EE \left[\left(\tilde M_{N,\varphi}(t')\right)^2 \tilde\Phi_N\right] \to \tilde\EE \left[\left(\tilde M_\varphi(t')\right)^2 \tilde\Phi\right]
\]
as $N \to \infty$, where \eqref{conv_phi_n} has been used once more. Therefore, \eqref{id_mart_2} follows by taking the limit as $N \to \infty$ in \eqref{id_mart_n_2}.

\proofstep{Argument for \eqref{id_mart_3}}
With the same reasoning as before, we have
\[
\tilde\EE\left[\int_{\floor{\frac{t'}{\delta}} \delta}^{\floor{\frac t \delta} \delta} \left(\psi_k \tilde w_N(t'',\cdot),\partial_x\varphi\right)_2 \d t'' \, \tilde\Phi_N\right] \to \tilde\EE\left[\int_{t'}^t \left(\psi_k \tilde u(t'',\cdot),\partial_x\varphi\right)_2 \d t'' \, \tilde\Phi\right] \quad \mbox{as} \quad N \to \infty.
\]
Furthermore, with help of the Cauchy-Schwarz and the Burkholder-Davis-Gundy inequality (cf.~\cite[Theorem~3.28]{KaratzasShreve1991})
\begin{eqnarray*}
\tilde\EE \left[(\tilde\beta^k_N(t))^2 \, (\tilde M_{N,\varphi}(t))^2 \, (\tilde\Phi_N)^2\right] &\stackrel{\eqref{conv_phi_n}}{\le}& \sqrt{\tilde \EE \left(\tilde\beta^k_N(t)\right)^4} \, \sqrt{\tilde \EE \left(\tilde M_{N,\varphi}(t)\right)^4} \\
&\le& C \, \sqrt{\tilde \EE \qvar{\tilde\beta^k_N}{t}^2} \, \sqrt{\tilde \EE \qvar{\tilde M_{N,\varphi}}{t}^2} \\
&\stackrel{\eqref{q_var_n}}{\le}& C \, t^2 \vertii{\partial_x \varphi}_2^2 \vertii{u_0}_{1,2}^2,
\end{eqnarray*}
where $C < \infty$ is independent of $N$, which implies with $\tilde\beta_N^k \to \tilde\beta_N$ as $N \to \infty$ uniformly in $[0,T)$, $\tilde\PP$-almost surely, by Proposition~\ref{prop:point_convergence} and \eqref{def_t_beta_k_N}, \eqref{conv_phi_n}, and $\tilde M_{N,\varphi}(t) \to \tilde M_\varphi(t)$ as $N \to \infty$, $\tilde\PP$-almost surely, by \eqref{lim_m_1}, the limits
\begin{align*}
\tilde\EE \left[\tilde\beta^k(t) \, \tilde M_{N,\varphi}(t) \, \tilde\Phi_N\right] &\to \tilde\EE \left[\tilde\beta^k(t) \, \tilde M_\varphi(t) \, \tilde\Phi\right], \\
\tilde\EE \left[\tilde\beta^k_N(t') \, \tilde M_{N,\varphi}(t') \, \tilde\Phi_N\right] &\to \tilde\EE \left[\tilde\beta^k(t') \, \tilde M_\varphi(t') \, \tilde\Phi\right]
\end{align*}
as $N \to \infty$ by Vitali's convergence theorem. Hence, \eqref{id_mart_3} follows from \eqref{id_mart_n_3}.
\end{proof}
%

\section{Concluding remarks\label{sec:conclude}}
The Trotter-Kato splitting scheme \ref{item:deterministic}--\ref{item:stochastic}--\ref{item:deterministic_stochastic}, utilized in the present work for the construction of solutions to \eqref{stfe}, can also be used for the design of a suitable numerical scheme. Hence, an interesting direction for future research may be to further develop the present analysis to prove the convergence of this or a similar numerical algorithm. A numerical treatment of the stochastic thin-film equation with It\^o noise and an additional interface potential has been introduced by Gr\"un, Mecke, and Rauscher in \cite[\S3.1]{GruenMeckeRauscher2006}. Furthermore, it may be of interest to test whether employing Stratonovich noise leads to different findings in the droplet formation simulations carried out in \cite{GruenMeckeRauscher2006}.

\medskip

It appears to be challenging to investigate the stochastic thin-film equation
\begin{equation}\label{stfe_general}
\d u = - \partial_x \left(u^n \partial_x^3 u\right) \d t + \partial_x \left(u^{\frac n 2} \circ \d W\right),
\end{equation}
where $n \in [1,3]$ and where the cubic mobility $n = 3$ (corresponding to no slip at the substrate) is of particular interest. In this case, however, the noise is nonlinear and singular for $n < 2$, so that for instance shocks in the stochastic dynamics may form. Hence, we expect the analysis in this situation to be significantly more involved. For relevant analysis in the case of the second-order SPDE
\[
\d u = \Delta u^m \, \d t + \nabla \cdot \left(u^p \circ \d W\right)
\]
we refer to the works \cite{DareiotisGess2018,GessFehrmann2019,GessSouganidis2015}.

\medskip

It should also be noted that, besides the weak solution approach, an extensive theory of classical solutions to the thin-film equation, based on maximal-regularity estimates of the linearized evolution, has been developed, starting with the works of Bringmann, Giacomelli, Kn\"upfer, and Otto \cite{GiacomelliKnuepferOtto2008,GiacomelliKnuepfer2010,BringmannGiacomelliKnuepferOtto2016} for linear mobility in one space dimension and with zero contact angle and later on further developed to include nonlinear mobilities, nonzero contact angles, and higher dimensions in \cite{Knuepfer2011,GiacomelliGnannKnuefperOtto2014,Knuepfer2015,Degtyarev2017,Knuepfer2019_err,John2015,Gnann2016,GnannPetrache2018}. On the other hand, there have been recent developments in the theory of mild solutions and maximal regularity for stochastic partial differential equations due to van Neerven, Veraar, and Weis \cite{VanNeervenVeraarWeis2012,VanNeervenVeraarWeis2012_2} and Hornung \cite{Hornung2019}. It would be a viable goal to combine these techniques in order to obtain a stronger control of the solution.

\medskip

Finally, it would be an illuminating task to study the self-similar behavior of the stochastic thin-film equation \eqref{stfe_general} analytically and thus to lift the numerical findings and dimensional analysis of Davidovitch, Moro, and Stone in \cite{DavidovitchMoroStone2005} to full mathematical rigor. Note that   again analytic results in the deterministic case have been obtained for the thin-film equation with linear mobility, starting with the works of Bernoff and Witelski in \cite{BernoffWitelski2002} and Carrillo and Toscani in \cite{CarrilloToscani2002} and later on upgraded in \cite{CarlenUlusoy2007,MatthesMcCannSavare2009,CarlenUlusoy2014,Gnann2015,McCannSeis2015,Seis2018}.

\medskip

We believe that all questions detailed above are interesting future directions, but appear to be analytically quite challenging to address.

\appendix

\section{Viscous regularization of stochastic dynamics\label{sec:viscous}}
Let $\left(\Omega,\cF,\left(\cF_t\right)_{t \in [0,\delta]},\PP\right)$ be a complete filtered probability space with a complete and right-continuous filtration $\left(\cF_t\right)_{t \in [0,\delta]}$. Further write $\left(\beta^k\right)_{k \in \Z}$ for mutually independent standard real-valued $(\cF_t)$-Wiener processes. Consider the viscous regularization
\begin{equation}\label{stochastic_regular}
\d w^\eps =  \left(\frac 1 2 \sum_{k \in \Z} \lambda_k^2 \partial_x\left(\psi_k\partial_x(\psi_k w^\eps)\right) + \eps \, \partial_x^2 w^\eps\right) \d t + \sum_{k \in \Z} \lambda_k \partial_x (\psi_k w^\eps) \, \d\beta^k \quad \mbox{on} \quad [0,\delta)
\end{equation}
of equation~\eqref{stochastic_3}, where $\eps \in (0,1]$. Our aim is to construct a variational solution to \eqref{stochastic_regular}. Therefore, we introduce the operators
\begin{equation}\label{a_op}
A^\eps \colon H^2(\TTT_L) \to L^2(\TTT_L), \quad w \mapsto \frac 1 2 \sum_{k \in \Z} \lambda_k^2 \partial_x\left(\psi_k\partial_x(\psi_k w)\right) + \eps \, \partial_x^2 w
\end{equation}
and the diagonal Hilbert-Schmidt-valued operator
\begin{equation}\label{b_op}
B \colon H^2(\TTT_L) \to L_2\left(H^2(\TTT_L);H^1(\TTT_L)\right), \quad w \mapsto \left(v \mapsto \sum_{k \in \Z} \lambda_k \left(v,\psi_k\right)_{2,2} \left(\partial_x (\psi_k w)\right)\right),
\end{equation}
Equation~\eqref{stochastic_regular} then attains the abstract form
\begin{equation}\label{stochastic_regular_abstract}
\d w^\eps = A^\eps w^\eps \, \d t + (B w^\eps) \, \d W_{H^2(\TTT_L)},
\end{equation}
where
\begin{equation}\label{def_wiener}
W_{H^2(\TTT_L)} := \sum_{k \in \Z} \beta^k \psi_k
\end{equation}
is a cylindrical $(\cF_t)$-Wiener process in $H^2(\TTT_L)$. The underlying Gelfand triple is
\[
\left(L^2(\TTT_L),H^1(\TTT_L),H^2(\TTT_L)\right).
\]
We use the following notion of solutions (see \cite[Definition~5.1.2]{LiuRoeckner2015}):
\begin{definition}\label{def:sol_eps}
A variational solution to \eqref{stochastic_regular_abstract} is a continuous $(\cF_t)$-adapted $H^1(\TTT_L)$-valued process $w^\eps$ such that
\[
\hat w^\eps \in L^2\left([0,\delta) \times \Omega, \d t \otimes \d\PP; H^2(\TTT_L)\right),
\]
where $\hat w^\eps$ denotes the $\d t \otimes \d\PP$-equivalence class of $w^\eps$, and
\begin{equation}\label{solution_formula_eps}
w^\eps(t,\cdot) = w_0 + \int_0^t A^\eps \bar w^\eps(t',\cdot) \, \d t' + \int_0^t (B \bar w^\eps(t',\cdot)) \, \d W_{H^2(\TTT_L)}(t',\cdot) \quad \mbox{for} \quad t \in [0,\delta),
\end{equation}
$\PP$-almost surely. Here, $\bar w^\eps$ denotes any $H^2(\TTT_L)$-valued progressively measurable (i.e., for any $t \in [0,\delta)$ the process $\bar w^\eps|_{[0,t] \times \Omega \times \TTT_L}$ is $\cB([0,t]) \otimes \cF_t \otimes \cB(\TTT_L)$-measurable) $\d t \otimes \d\PP$-version of $\hat w^\eps$.
\end{definition}
\begin{proposition}\label{prop:var}
Assume that \eqref{cond_lambda} holds true and that $p \in [2,\infty)$. Then, for any $w_0 \in L^p\left(\Omega,\cF_0,\PP;H^1(\TTT_L)\right)$, equation~\eqref{stochastic_regular} has a unique variational solution $w^\eps$ with initial value $w_0$ satisfying
\begin{equation}\label{regularity_variational}
\EE\left(\sup_{t \in [0,\delta)} \vertii{w^\eps(t,\cdot)}_{1,2}^p + \int_0^\delta \vertii{w^\eps(t,\cdot)}_{2,2}^2 \, \d t\right) < \infty.
\end{equation}
Furthermore, we have the a-priori estimates
\begin{subequations}\label{apriori_w_eps}
\begin{align}
\EE \sup_{t \in [0,\delta)} \vertii{w^\eps(t,\cdot)}_{1,2}^p &\le C_1 \, \EE\vertii{w_0}_{1,2}^p, \label{apriori_w_eps_1} \\ 
\lim_{t \nearrow \delta} \EE \vertii{\partial_x w^\eps(t,\cdot)}_2^p &\le e^{C_2 \delta} \left(\EE \vertii{\partial_x w_0}_2^p + C_3 \, \delta \, \EE \verti{\int_0^L w_0 \, \d x}^p\right), \label{apriori_w_eps_2}
\end{align}
\end{subequations}
where $C_1, C_2, C_3 < \infty$ are independent of $\eps$, $\delta$, $w^\eps$, and $w_0$.
\end{proposition}

A main ingredient for proving Proposition~\ref{prop:var} is the following lemma, for which the use of Stratonovich calculus (see the discussion in \S\ref{sec:ito_strat}) is essential:
\begin{lemma}[monotonicity and coercivity]\label{lem:mon_coerc}
Suppose \eqref{cond_lambda} holds true. Then, for $w \in H^2(\TTT_L)$ we have
\begin{subequations}\label{mon_coerc}
\begin{equation}\label{mon_coerc_1}
2 \dualpair{A^\eps w}{w} + \vertii{B w}_{L_2\left(H^2(\TTT_L);L^2(\TTT_L)\right)}^2 \le C \vertii{w}_2^2 - 2 \eps \vertii{w}_{1,2}^2
\end{equation}
and
\begin{equation}\label{mon_coerc_2}
2 \dualpair{\partial_x A^\eps w}{\partial_x w} + \vertii{B w}_{L_2\left(H^2(\TTT_L);\dot H^1(\TTT_L)\right)}^2 \le C \vertii{w}_{1,2}^2 - 2 \eps \vertii{\partial_x w}_{1,2}^2
\end{equation}
for some $C < \infty$ independent of $w$ and $\eps$, so that in particular
\begin{equation}\label{mon_coerc_3}
2 \Dualpair{A^\eps w}{w} + \vertii{B w}_{L_2\left(H^2(\TTT_L);H^1(\TTT_L)\right)}^2 \le C \vertii{w}_{1,2}^2 - 2 \eps \vertii{w}_{2,2}^2.
\end{equation}
\end{subequations}
\end{lemma}
\begin{proof}[Proof of Lemma~\ref{lem:mon_coerc}]
By definition, estimate~\eqref{mon_coerc_3} follows by adding \eqref{mon_coerc_1} and \eqref{mon_coerc_2}. We prove \eqref{mon_coerc_1} and \eqref{mon_coerc_2} separately:
\proofstep{Proof of \eqref{mon_coerc_1}}
Observe that for $w \in H^2(\TTT_L)$ we obtain through integration by parts
\begin{align*}
\dualpair{A^\eps w}{w} &= - \frac 1 2 \sum_{k \in \Z} \lambda_k^2 \int_0^L \psi_k^2 (\partial_x w)^2 \, \d x - \eps \int_0^L (\partial_x w)^2 \d x - \frac 1 8 \sum_{k \in \Z} \lambda_k^2 \int_0^L \left(\partial_x \psi_k^2\right) \left(\partial_x w^2\right) \d x \\
&= - \frac 1 2 \sum_{k \in \Z} \lambda_k^2 \int_0^L \psi_k^2 (\partial_x w)^2 \, \d x + \frac 1 8 \sum_{k \in \Z} \lambda_k^2 \int_0^L \left(\partial_x^2 \psi_k^2\right) \, w^2 \, \d x - \eps \int_0^L (\partial_x w)^2 \d x
\end{align*}
and further
\begin{align*}
\vertii{B w}_{L_2\left(H^2(\TTT_L);L^2(\TTT_L)\right)}^2 &= \sum_{k \in \Z} \lambda_k^2 \int_0^L \left((\partial_x \psi_k) w + \psi_k (\partial_x w)\right)^2 \d x \\
&= \sum_{k \in \Z} \lambda_k^2 \int_0^L \psi_k^2 \, (\partial_x w)^2 \, \d x - \sum_{k \in \Z} \lambda_k^2 \int_0^L \psi_k \, (\partial_x^2 \psi_k) \, w^2 \, \d x,
\end{align*}
so that the term $\sum_{k \in \Z} \lambda_k^2 \int_0^L \psi_k^2 \, (\partial_x w)^2 \, \d x$ cancels and we get
\begin{eqnarray*}
\lefteqn{2 \dualpair{A^\eps w}{w} + \vertii{B w}_{L_2\left(H^2(\TTT_L);L^2(\TTT_L)\right)}^2} \\
&=& \frac 1 4 \sum_{k \in \Z} \lambda_k^2 \int_0^L \left((\partial_x^2 \psi_k^2) - 4 \psi_k (\partial_x^2 \psi_k)\right) w^2 \, \d x - 2 \eps \int_0^L (\partial_x w)^2 \, \d x \\
&\stackrel{\eqref{basis},\eqref{der_psi_k}}{\le}& C \sum_{k \in \Z} \lambda_k^2 \vertii{w}_2^2 - 2 \eps \vertii{\partial_x w}_2^2 \stackrel{\eqref{cond_lambda}}{\le} C \vertii{w}_2^2 - 2 \eps \vertii{w}_{1,2}^2
\end{eqnarray*}
for some $C < \infty$ independent of $\eps$, where we have used $\eps \le 1$.

\proofstep{Proof of \eqref{mon_coerc_2}}
Again, for $w \in H^2(\TTT_L)$ we integrate by parts several times and arrive at
\begin{eqnarray*}
\lefteqn{\dualpair{\partial_x A^\eps w}{\partial_x w}} \\
&=& - \frac 1 2 \sum_{k \in \Z} \lambda_k^2 \int_0^L \left(\partial_x \left(\psi_k \partial_x (\psi_k w)\right)\right) (\partial_x^2 w) \, \d x - \eps \int_0^L (\partial_x^2 w)^2 \, \d x \\
&=& - \frac 1 2 \sum_{k \in \Z} \lambda_k^2 \int_0^L \left(\psi_k^2 (\partial_x^2 w) + \frac 3 2 (\partial_x \psi_k^2) (\partial_x w) + \frac 1 2 (\partial_x^2 \psi_k^2) w\right) (\partial_x^2 w) \, \d x - \eps \int_0^L (\partial_x^2 w)^2 \, \d x \\
&=& - \frac 1 2 \sum_{k \in \Z} \lambda_k^2 \int_0^L \psi_k^2 (\partial_x^2 w)^2 \, \d x - \frac 3 8 \sum_{k \in \Z} \lambda_k^2 (\partial_x \psi_k^2) \left(\partial_x (\partial_x w)^2\right) \d x \\
&& + \frac 1 4 \sum_{k \in \Z} \lambda_k^2 \int_0^L (\partial_x^2 \psi_k^2) (\partial_x w)^2 \, \d x + \frac 1 8 \sum_{k \in \Z} \lambda_k^2 \int_0^L (\partial_x^3 \psi_k^2) (\partial_x w^2) \, \d x - \eps \int_0^L (\partial_x^2 w)^2 \, \d x \\
&=& - \frac 1 2 \sum_{k \in \Z} \lambda_k^2 \int_0^L \psi_k^2 (\partial_x^2 w)^2 \, \d x + \frac 5 8 \sum_{k \in \Z} \lambda_k^2 \int_0^L \left(\partial_x^2 \psi_k^2\right) (\partial_x w)^2 \, \d x \\
&& - \frac 1 8 \sum_{k \in \Z} \lambda_k^2 \int_0^L \left(\partial_x^4 \psi_k^2\right) w^2 \, \d x - \eps \int_0^L (\partial_x^2 w)^2 \, \d x
\end{eqnarray*}
and
\begin{eqnarray*}
\lefteqn{\vertii{B w}_{L_2\left(H^2(\TTT_L);\dot H^1(\TTT_L)\right)}^2} \\
&=& \sum_{k \in \Z} \lambda_k^2 \int_0^L \left((\partial_x^2 \psi_k) w + 2 (\partial_x \psi_k) (\partial_x w) + \psi_k (\partial_x^2 w)\right)^2 \d x \\
&=& \sum_{k \in \Z} \lambda_k^2 \int_0^L \psi_k^2 (\partial_x^2 w)^2 \, \d x + 4 \sum_{k \in \Z} \lambda_k^2 \int_0^L (\partial_x \psi_k)^2 (\partial_x w)^2 \, \d x + \sum_{k \in \Z} \lambda_k^2 \int_0^L (\partial_x^2 \psi_k)^2 w^2 \, \d x \\
&& + \sum_{k \in \Z} \lambda_k^2 \int_0^L (\partial_x \psi_k^2) \left(\partial_x (\partial_x w)^2\right) \d x + \sum_{k \in \Z} \lambda_k^2 \int_0^L \left(\partial_x  (\partial_x\psi_k)^2\right) (\partial_x w^2) \, \d x \\
&& + 2 \sum_{k \in \Z} \lambda_k^2 \int_0^L \psi_k (\partial_x^2 \psi_k) w (\partial_x^2 w) \, \d x \\
&=& \sum_{k \in \Z} \lambda_k^2 \int_0^L \psi_k^2 (\partial_x^2 w)^2 \, \d x + \sum_{k \in \Z} \lambda_k^2 \int_0^L \left(4 (\partial_x \psi_k)^2 - (\partial_x^2 \psi_k^2) - 2 \psi_k (\partial_x^2 \psi_k)\right) (\partial_x w)^2 \, \d x \\
&& + \sum_{k \in \Z} \lambda_k^2 \int_0^L \left((\partial_x^2 \psi_k)^2 - \partial_x^2 (\partial_x \psi_k)^2 + \partial_x^2 \left(\psi_k (\partial_x^2 \psi_k)\right)\right) w^2 \, \d x \\
&=& \sum_{k \in \Z} \lambda_k^2 \int_0^L \psi_k^2 (\partial_x^2 w)^2 \, \d x + \sum_{k \in \Z} \lambda_k^2 \int_0^L \left(2 (\partial_x \psi_k)^2 - 4 \psi_k (\partial_x^2 \psi_k)\right) (\partial_x w)^2 \, \d x \\
&& + \sum_{k \in \Z} \lambda_k^2 \int_0^L \psi_k \, (\partial_x^4 \psi_k) \, w^2 \, \d x,
\end{eqnarray*}
and hence $\sum_{k \in \Z} \lambda_k^2 \int_0^L \psi_k^2 (\partial_x^2 w)^2 \, \d x$ cancels and we arrive at
\begin{eqnarray*}
\lefteqn{2 \dualpair{\partial_x A^\eps w}{\partial_x w} + \vertii{B w}_{L_2\left(H^2(\TTT_L);\dot H^1(\TTT_L)\right)}^2} \\
&=& \frac 3 2 \sum_{k \in \Z} \lambda_k^2 \int_0^L \left(3 (\partial_x \psi_k)^2 - \psi_k (\partial_x^2 \psi_k)\right) (\partial_x w)^2 \, \d x \\
&& +  \frac 1 4 \sum_{k \in \Z} \lambda_k^2 \int_0^L \left(- (\partial_x^4 \psi_k^2) + 4 \psi_k (\partial_x^4 \psi_k)\right) w^2 \, \d x - 2 \eps \int_0^L (\partial_x^2 w)^2 \, \d x \\
&\stackrel{\eqref{basis},\eqref{der_psi_k}}{\le}& C \, \sum_{k \in \Z} \lambda_k^2 \vertii{w}_{1,2}^2 - 2 \eps \vertii{\partial_x^2 w}_2^2 \stackrel{\eqref{cond_lambda}}{\le} C \vertii{w}_{1,2}^2 - 2 \eps \vertii{\partial_x w}_{1,2}^2
\end{eqnarray*}
for some $C < \infty$ independent of $\eps$, where we have used $\eps \le 1$.
\end{proof}
\begin{proof}[Proof of Proposition~\ref{prop:var}]
We verify sufficient conditions for variational solutions to \eqref{stochastic_regular} as can be found for instance in \cite[Theorem~4.2.4]{LiuRoeckner2015}.

\proofstep{Hemicontinuity}
For $u, v, w \in H^2(\TTT_L)$ and $s \in \R$ we have
\[
\Dualpair{A^\eps (u+sv)}{w} = \Dualpair{A^\eps u}{w} + s \Dualpair{A^\eps v}{w},
\]
which is for fixed $u$, $v$, and $w$ a linear function in $s$ and in particular hemicontinuous.

\proofstep{Weak monotonicity and coercivity}
This follows from \eqref{mon_coerc_3} of Lemma~\ref{lem:mon_coerc}.

\proofstep{Boundedness}
For $w \in H^1(\TTT_L)$ and $\varphi \in C^\infty(\TTT_L)$ we have
\begin{eqnarray*}
\verti{\Dualpair{A^\eps w}{\varphi}} &\le& \frac 1 2 \sum_{k \in \Z} \lambda_k^2 \verti{\int_0^L \psi_k \left(\partial_x (\psi_k w)\right) (\partial_x \varphi) \, \d x} + \eps \verti{\int_0^L (\partial_x w) (\partial_x \varphi) \, \d x} \\
&& + \frac 1 2 \sum_{k \in \Z} \lambda_k^2 \verti{\int_0^L \left(\partial_x \left(\psi_k \left(\partial_x (\psi_k w)\right)\right)\right) (\partial_x^2 \varphi) \, \d x} + \eps \verti{\int_0^L (\partial_x^2 w) (\partial_x^2 \varphi) \, \d x} \\
&\stackrel{\eqref{basis}, \eqref{der_psi_k}}{\le}& C \left(\sum_{k \in \Z} \lambda_k^2 + \eps\right) \vertii{w}_{2,2} \vertii{\varphi}_{2,2},
\end{eqnarray*}
so that $\vertii{A^\eps w}_{L^2(\TTT_L)} \stackrel{\eqref{cond_lambda}}{\le} C \vertii{w}_{2,2}$ since $\eps \le 1$.

\proofstep{A-priori estimate \eqref{apriori_w_eps_1}}
From \cite[Theorem~4.2.4]{LiuRoeckner2015} we infer that a unique variational solution to \eqref{stochastic_regular_abstract} as in Definition~\ref{def:sol_eps} exists and \eqref{regularity_variational} is satisfied. While general $p \in [2,\infty)$ are treated in \cite[Theorem~5.1.3]{LiuRoeckner2015} or \cite[Theorem~1.1]{LiuRoeckner2010}, the noise there does not allow for a gradient structure as in the present case. Nonetheless, the reasoning mainly follows the proof of \cite[Lemma~5.1.5]{LiuRoeckner2015}.

\medskip

Using It\^o's lemma (cf.~\cite[Theorem~3.1]{Krylov2013} or \cite[Theorem~4.2.5]{LiuRoeckner2015}) and equation~\eqref{solution_formula_eps} of Definition~\ref{def:sol_eps}, we obtain for $t \in [0,\delta)$
\begin{eqnarray*}
\lefteqn{\vertii{w^\eps(t,\cdot)}_{1,2}^2 - \vertii{w_0}_{1,2}^2} \\
&=& 2 \int_0^t \left(\left(B w^\eps(t',\cdot)\right) \d W_{H^2(\TTT_L)}(t'), w^\eps(t',\cdot)\right)_{1,2} \\
&& + \int_0^t \left(2 \Dualpair{A^\eps w^\eps(t',\cdot)}{w^\eps(t',\cdot)} + \vertii{B w^\eps(t',\cdot)}_{L_2(H^2(\TTT_L);H^1(\TTT_L))}^2\right) \d t' \\
&\stackrel{\eqref{b_op},\eqref{def_wiener}}{=}& 2 \sum_{k \in \Z} \lambda_k \int_0^t \left(\partial_x \left(\psi_k w^\eps(t',\cdot)\right), w^\eps(t',\cdot)\right)_{1,2} \d \beta^k(t') \\
&& + \int_0^t \left(2 \Dualpair{A^\eps w^\eps(t',\cdot)}{w^\eps(t',\cdot)} + \vertii{B w^\eps(t',\cdot)}_{L_2(H^2(\TTT_L);H^1(\TTT_L))}^2\right) \d t',
\end{eqnarray*}
$\PP$-almost surely. For $p \ge 4$ this implies again using It\^o's lemma for $\R \owns y \mapsto \verti{y}^{\frac p 2}$
\begin{align}
&\vertii{w^\eps(t,\cdot)}_{1,2}^p - \vertii{w_0}_{1,2}^p \label{apriori_w_proof} \\
&\quad = p \sum_{k \in \Z} \lambda_k \int_0^t \vertii{w^\eps(t',\cdot)}_{1,2}^{p-2} \left(\partial_x \left(\psi_k w^\eps(t',\cdot)\right),w^\eps(t',\cdot)\right)_{1,2} \d \beta^k(t') \nonumber \\
&\quad \phantom{=} + \frac p 2 \int_0^t \vertii{w^\eps(t',\cdot)}_{1,2}^{p-2} \left(2 \Dualpair{A^\eps w^\eps(t',\cdot)}{w^\eps(t',\cdot)} + \vertii{B w^\eps(t',\cdot)}_{L_2(H^2(\TTT_L);H^1(\TTT_L))}^2\right) \d t' \nonumber \\
&\quad \phantom{=} + \frac{p (p-2)}{2} \sum_{k \in \Z} \lambda_k^2 \int_0^t \vertii{w^\eps(t',\cdot)}_{1,2}^{p-4} \left(\partial_x \left(\psi_k w^\eps(t',\cdot)\right),w^\eps(t',\cdot)\right)_{1,2}^2 \, \d t', \nonumber
\end{align}
$\PP$-almost surely. Next, we introduce for any $R > 0$ the stopping times
\[
\tau_R := \inf\left\{t \in [0,\delta) \colon \vertii{w^\eps(t,\cdot)}_{1,2} > R\right\} \wedge \delta.
\]
By Markov's inequality and using \eqref{regularity_variational} for $p = 2$
\[
\PP \left\{\sup_{t \in [0,\delta)} \vertii{w^\eps(t,\cdot)}_{1,2} > R\right\} \le \frac{1}{R^2} \EE \sup_{t \in [0,\delta)} \vertii{w^\eps(t,\cdot)}_{1,2}^2 \to 0 \quad \mbox{as} \quad R \to \infty,
\]
so that $\lim_{R \to \infty} \tau_R = \delta$, $\PP$-almost surely. The Burkholder-Davis-Gundy inequality (cf.~\cite[Theorem~3.28]{KaratzasShreve1991}) implies
\begin{eqnarray*}
\lefteqn{\EE \sup_{t' \in [0,\tau_R \wedge t)} \verti{\int_0^{t'} \vertii{w^\eps(t'',\cdot)}_{1,2}^{p-2} \left(\partial_x (\psi_k w^\eps(t'',\cdot)), w^\eps(t'',\cdot)\right)_{1,2} \d \beta^k(t'')}} \\
&\le& 3 \, \EE \sqrt{\int_0^{\tau_R \wedge t} \vertii{w^\eps(t',\cdot)}_{1,2}^{2 p - 4} \left(\partial_x (\psi_k w^\eps(t',\cdot)), w^\eps(t',\cdot)\right)_{1,2}^2 \d t'}.
\end{eqnarray*}
Now, we note that integration by parts gives
\begin{align*}
\left(\partial_x (\psi_k w^\eps(t',\cdot)), w^\eps(t',\cdot)\right)_{1,2} &= \int_0^L (\partial_x \psi_k) \, (w^\eps)^2 \, \d x + \int_0^L \psi_k \, w^\eps \, (\partial_x w^\eps) \, \d x \\
&\phantom{=} + \int_0^L (\partial_x^2 \psi_k) \, w^\eps \, (\partial_x w^\eps) \, \d x + \frac 3 2 \int_0^L (\partial_x \psi_k) \, (\partial_x w^\eps)^2 \, \d x,
\end{align*}
so that with \eqref{basis} and \eqref{der_psi_k} we have
\[
\verti{\left(\partial_x (\psi_k w^\eps(t',\cdot)), w^\eps(t',\cdot)\right)_{1,2}} \le C \vertii{w^\eps(t',\cdot)}_{1,2}^2, \quad \mbox{$\PP$-almost surely},
\]
where $C < \infty$ only depends on $L$, and hence by Young's inequality
\begin{eqnarray*}
\lefteqn{\EE \sup_{t' \in [0,\tau_R \wedge t)} \verti{\int_0^{t'} \vertii{w^\eps(t'',\cdot)}_{1,2}^{p-2} \left(\partial_x (\psi_k w^\eps(t'',\cdot)), w^\eps(t'',\cdot)\right)_{1,2} \d \beta^k(t'')}} \\
&\le& C \, \EE \sqrt{\int_0^{\tau_R \wedge t} \vertii{w^\eps(t',\cdot)}_{1,2}^{2 p} \d t'} \le C \, \EE \sqrt{\sup_{t' \in [0,\tau_R \wedge t)} \vertii{w^\eps(t',\cdot)}_{1,2}^p \int_0^{\tau_R \wedge t} \vertii{w^\eps(t',\cdot)}_{1,2}^p \d t'} \\
&\le& \nu \, \EE \sup_{t' \in [0,\tau_R)} \vertii{w^\eps(t',\cdot)}_{1,2}^p + \frac C \nu \, \EE \int_0^{\tau_R} \vertii{w^\eps(t',\cdot)}_{1,2}^p \d t',
\end{eqnarray*}
where $\nu > 0$ can be chosen arbitrarily small and $C < \infty$ is independent of $R$. Furthermore, with the same computation also
\[
\int_0^t \vertii{w^\eps(t',\cdot)}_{1,2}^{p-4} \left(\partial_x \left(\psi_k w^\eps(t',\cdot)\right),w^\eps(t',\cdot)\right)_{1,2}^2 \, \d t' \le C \int_0^t \vertii{w^\eps(t',\cdot)}_{1,2}^p \, \d t', \quad \mbox{$\PP$-almost surely},
\]
where $C < \infty$ only depends on $L$. Now, the combination with \eqref{cond_lambda}, \eqref{mon_coerc_3} of Lemma~\ref{lem:mon_coerc}, and \eqref{apriori_w_proof} gives for sufficiently small $\nu$
\[
\EE \sup_{t' \in [0,\tau_R \wedge t)} \vertii{w^\eps(t',\cdot)}_{1,2}^p \le C \left( \EE \vertii{w_0}_{1,2}^p + \int_0^{\tau_R \wedge t} \EE \sup_{t'' \in [0,\tau_R \wedge t')} \vertii{w^\eps(t'',\cdot)}_{1,2}^p \d t''\right),
\]
where $C < \infty$ only depends on $L$ and $p \in \{2\} \cup [4,\infty)$. Gr\"onwall's inequality implies
\begin{align*}
\EE \sup_{t \in [0,\tau_R)} \vertii{w^\eps(t,\cdot)}_{1,2}^p &\le C_1 \EE \vertii{w_0}_{1,2}^p,
\end{align*}
with $C_1 < \infty$ only depending on $L$ and $T$, so that \eqref{apriori_w_eps_1} for $p \in \{2\} \cup [4,\infty)$ follows by monotone convergence in the limit as $R \to \infty$. The case $p \in (2,4)$ is obtained by complex interpolation using the Banach-valued Riesz-Thorin theorem (cf.~\cite[Theorem~2.2.1]{HytonenVanNeervenVeraarWeis2016} or more generally \cite[Theorem~5.1.2]{BergLofstrom1976}).

\proofstep{A-priori estimate \eqref{apriori_w_eps_2}}
We precisely keep track on the constants appearing in order to derive estimate~\eqref{apriori_w_eps_2}:

\medskip

With help of It\^o's lemma (cf.~\cite[Theorem~3.1]{Krylov2013} or \cite[Theorem~4.2.5]{LiuRoeckner2015}) we obtain for $t \in [0,\delta)$ and utilizing equation~\eqref{solution_formula_eps} of Definition~\ref{def:sol_eps}
\begin{eqnarray*}
\lefteqn{\vertii{\partial_x w^\eps(t,\cdot)}_2^2 - \vertii{\partial_x w_0}_2^2} \\
&\stackrel{\eqref{b_op}}{=}& 2 \sum_{k \in \Z} \lambda_k \int_0^t \left(\partial_x^2 \left(\psi_k w^\eps(t',\cdot)\right),\partial_x w^\eps(t',\cdot)\right)_2 \d \beta^k(t') \\
&& + \int_0^t \left(2 \dualpair{\partial_x A^\eps w^\eps(t',\cdot)}{\partial_x w^\eps(t',\cdot)} + \vertii{B w^\eps(t',\cdot)}_{L_2(H^2(\TTT_L);\dot H^1(\TTT_L))}^2\right) \d t',
\end{eqnarray*}
$\PP$-almost surely. For $p \ge 4$, It\^o's formula applied to $\R \owns y \mapsto \verti{y}^{\frac p 2}$ gives
\begin{align*}
&\vertii{\partial_x w^\eps(t,\cdot)}_2^p - \vertii{\partial_x w_0}_2^p \\
&\quad = p \sum_{k \in \Z} \lambda_k \int_0^t \vertii{\partial_x w^\eps(t',\cdot)}_2^{p-2} \left(\partial_x^2 \left(\psi_k w^\eps(t',\cdot)\right),\partial_x w^\eps(t',\cdot)\right)_2 \d \beta^k(t') \\
&\quad \phantom{=} + \frac p 2 \int_0^t \vertii{\partial_x w^\eps(t',\cdot)}_2^{p-2} \left(2 \dualpair{\partial_x A^\eps w^\eps(t',\cdot)}{\partial_x w^\eps(t',\cdot)} + \vertii{B w^\eps(t',\cdot)}_{L_2(H^2(\TTT_L);\dot H^1(\TTT_L))}^2\right) \d t' \\
&\quad \phantom{=} + \frac{p (p-2)}{2} \sum_{k \in \Z} \lambda_k^2 \int_0^t \vertii{\partial_x w^\eps(t',\cdot)}_2^{p-4} \left(\partial_x^2 \left(\psi_k w^\eps(t',\cdot)\right),\partial_x w^\eps(t',\cdot)\right)_2^2 \, \d t',
\end{align*}
$\PP$-almost surely. Taking the expectation gives
\begin{align*}
& \EE\vertii{\partial_x w^\eps(t,\cdot)}_2^p - \EE\vertii{\partial_x w^\eps(0,\cdot)}_2^p \\
&\quad = \frac p 2 \EE \int_0^t \vertii{\partial_x w^\eps(t',\cdot)}_2^{p-2} \left(2 \dualpair{\partial_x A^\eps w^\eps(t',\cdot)}{\partial_x w^\eps(t',\cdot)} + \vertii{B w^\eps(t',\cdot)}_{L_2(H^2(\TTT_L);\dot H^1(\TTT_L))}^2\right) \d t' \\
&\quad \phantom{=} + \frac{p (p-2)}{2} \sum_{k \in \Z} \lambda_k^2 \, \EE \int_0^t \vertii{\partial_x w^\eps(t',\cdot)}_2^{p-4} \left(\partial_x^2 \left(\psi_k w^\eps(t',\cdot)\right),\partial_x w^\eps(t',\cdot)\right)_2^2 \, \d t'.
\end{align*}
For the last line observe that through integration by parts as before
\[
\left(\partial_x^2 \left(\psi_k w^\eps(t',\cdot)\right),\partial_x w^\eps(t',\cdot)\right)_2 = \int_0^L (\partial_x^2 \psi_k) w^\eps \partial_x w^\eps \, \d x + \frac 3 2 \int_0^L (\partial_x \psi_k) (\partial_x w^\eps)^2 \, \d x,
\]
$\PP$-almost surely, that is,
\[
\left(\partial_x^2 \left(\psi_k w^\eps(t',\cdot)\right),\partial_x w^\eps(t',\cdot)\right)_2^2 \stackrel{\eqref{basis},\eqref{der_psi_k}}{\le} C \, \vertii{\partial_x w^\eps(t',\cdot)}_2^2 \vertii{w^\eps(t',\cdot)}_{1,2}^2, \quad \mbox{$\PP$-almost surely.}
\]
Further applying \eqref{mon_coerc_2} of Lemma~\ref{lem:mon_coerc} gives
\begin{eqnarray*}
\lefteqn{\EE\vertii{\partial_x w^\eps(t,\cdot)}_2^p - \EE\vertii{\partial_x w^\eps(0,\cdot)}_2^p} \\
&\le& C \, p (p+1) \sum_{k \in \Z} \lambda_k^2 \, \EE \int_0^t \vertii{\partial_x w^\eps(t',\cdot)}_2^{p-2} \left(\vertii{\partial_x w^\eps(t',\cdot)}_2^2 + \left(\int_0^L w^\eps(t',\cdot) \, \d x\right)^2\right) \d t' \\
&\stackrel{\eqref{cond_lambda}}{\le}& C_2 \int_0^t \EE \vertii{\partial_x w^\eps(t',\cdot)}_2^p \, \d t' + C_3 \, \EE \int_0^t \verti{\int_0^L w^\eps(t',\cdot) \, \d x}^p \d t',
\end{eqnarray*}
where we have applied Poincar\'e's inequality 
\[
\vertii{w^\eps(t',\cdot)}_{1,2} \le C \left(\vertii{\partial_x w^\eps(t',\cdot)}_2 + \verti{\int_0^L w^\eps(t',\cdot) \, \d x}\right), \quad \mbox{$\PP$-almost surely},
\]
and Young's inequality. Testing of \eqref{solution_formula_eps} of Definition~\ref{def:sol_eps} against a non-trivial constant gives $\int_0^L w^\eps(t,\cdot) \, \d x = \int_0^L w_0 \, \d x$ for $t \in [0,T)$, $\PP$-almost surely. Now the claim \eqref{apriori_w_eps} for $p \in \{2\} \cup [4,\infty)$ follows from Gr\"onwall's inequality and the general case $p \in [2,\infty)$ by complex interpolation using the Banach-valued Riesz-Thorin theorem (cf.~\cite[Theorem~2.2.1]{HytonenVanNeervenVeraarWeis2016} or more generally \cite[Theorem~5.1.2]{BergLofstrom1976}).
\end{proof}
%

\section{Real interpolation of Besov spaces with mixed smoothness\label{sec:interpolation}}
The following result from interpolation theory is essential in proving regularity in time (cf.~Proposition~\ref{prop:reg_time_un}).
\begin{lemma}\label{lem:interpolation}
Suppose $\delta \in (0,\infty)$, $r_1, r_2, s_1, s_2 \in \R$ with $r_1 \ne r_2$ and $s_1 \ne s_2$, $q \in [1,\infty)$, and $\kappa \in [0,1]$. Then
\begin{equation}\label{sob_interp}
\left(B^{r_1,q}_q\left([0,\delta);B^{s_1,q}_q(\TTT_L)\right), B^{r_2,q}_q\left([0,\delta);B^{s_2,q}_q(\TTT_L)\right)\right)_{\kappa,q} = B^{r,q}_q\left([0,\delta);B^{s,q}_q(\TTT_L)\right),
\end{equation}
where $r = (1-\kappa) r_1 + \kappa r_2$ and $s = (1-\kappa) s_1 + \kappa s_2$. The norms in \eqref{sob_interp} are equivalent, with bounds that are independent of $\delta$.
\end{lemma}
\begin{proof}
By \cite[Proposition~4.2]{Amann2000} and scaling in the time variable, there exists a $\delta$-uniformly bounded linear extension operator
\[
E \colon Y_j := B^{r_j,q}_q\left([0,\delta);B^{s_j,q}_q(\TTT_L)\right) \to Z_j := B^{r_j,q}_q\left(\R;B^{s_j,q}_q(\TTT_L)\right),
\]
that is, setting $R (\cdot) := (\cdot)|_{[0,\delta)}$, we have $R E v = v$ for $v \in Y_j$ and the operator norm of $E$ is independent of $\delta$. Now, we may apply \cite[Theorem~3.1]{Amann2000}, \cite[Theorem~2.7.2~(i)]{Amann2019}, or \cite[(6.9)]{Grisvard1966} to deduce
\[
\left(Z_1,Z_2\right)_{\kappa,q} = B^{r,q}_q\left(\R;\left(B^{s_1,q}_q(\TTT_L),B^{s_2,q}_q(\TTT_L)\right)_{\kappa,q}\right).
\]
The interpolation of periodic Besov spaces is known (cf.~\cite[\S3.6.1, Theorem~1~(i)]{SchmeisserTriebel1987}), that is,
\[
\left(B^{s_1,q}_q(\TTT_L),B^{s_2,q}_q(\TTT_L)\right)_{\kappa,q} = B^{s,q}_q(\TTT_L).
\]
Altogether, 
\[
Y_j = R (E Y_j) = R Z_j = R B^{r,q}_q\left(\R;B^{s,q}_q(\TTT_L)\right) = B^{r,q}_q\left([0,\delta);B^{s,q}_q(\TTT_L)\right)
\]
with $\delta$-uniformly equivalent norms, which yields \eqref{sob_interp}.
\end{proof}
\begin{lemma}\label{lem:connect}
Suppose that $X$ is a Banach space, $T \in (0,\infty)$, $N \in \N$, $\delta := \frac T N$, $s \in (0,1)$, and $q \in (1,\infty]$. If $\phi \in BC^0\left([0,T);X\right)$ and $\phi \in B^{s,q}_q\left(\left[(j-1) \frac \delta 2, j \frac \delta 2\right);X\right)$ for $j \in \left\{1,\ldots,2N\right\}$, then $\phi \in B^{s,q}_q\left([0, T);X\right)$ with
\begin{equation}\label{sum_norms}
\vertii{\phi}_{B^{s,q}_q\left([0, T);X\right)} \le C \left(\sum_{j = 1}^{2N} \vertii{\phi}_{B^{s,q}_q\left(\left[(j-1) \frac \delta 2, j \frac \delta 2\right);X\right)}^q\right)^{\frac 1 q},
\end{equation}
where $C < \infty$ only depends on $q$.
\end{lemma}
\begin{proof}
By mollification with a standard mollifier and using the interpolation property, we see that we can approximate $\phi$ on any interval $\left[(j-1) \frac \delta 2, j \frac \delta 2\right)$ by a function $\phi_\eps \in C^\infty\left(\left[(j-1) \frac \delta 2, j \frac \delta 2\right);X\right)$ with
\[
\vertii{\phi - \phi_\eps}_{BC^0(\left[(j-1) \frac \delta 2, j \frac \delta 2\right);X)} \to 0 \quad \mbox{and} \quad \vertii{\phi - \phi_\eps}_{B^{s,q}_q\left(\left[(j-1) \frac \delta 2, j \frac \delta 2\right);X\right)} \quad \mbox{as} \quad \eps \searrow 0,
\]
where $j \in \left\{1,\ldots,2N\right\}$.
Adding to $\phi_\eps$ the polygonal chain through the points
\[
\left(j \tfrac \delta 2, \phi\left(j \tfrac \delta 2,\cdot\right) - \phi_\eps\left(j \tfrac \delta 2,\cdot\right)\right), \quad \mbox{where} \quad j \in \left\{0,\ldots,2N\right\},
\]
we may without loss of generality additionally assume $\phi_\eps \in BC^0\left([0,T);X\right)$, so that in particular $\phi_\eps \in W^{1,\infty}([0,T);X)$, and
\[
\vertii{\phi - \phi_\eps}_{BC^0\left([0,T);X\right)} \to 0 \quad \mbox{as} \quad \eps \searrow 0.
\]
Since up to a $q$-dependent constant (denoted by $\sim_q$) we have for $\psi \in W^{1,\infty}\left([0,T);X\right)$,
\[
\vertii{\psi}_{B^{s,q}_q\left([0,T);X\right)}^q \sim_q \int_0^\infty \inf_{\psi = \psi_1 + \psi_2} \left(\tau^{- s q} \vertii{\psi_1}_{L^q\left([0,T);X\right)}^q + \tau^{(1-s) q} \vertii{\psi_2}_{W^{1,q}\left([0,T);X\right)}^q\right) \frac{\d\tau}{\tau},
\]
we recognize that by Sobolev embedding $\psi_2 \in BC^0\left([0,T);X\right)$, so that any decomposition $\psi = \psi_1 + \psi_2$ with $\psi_1 \in L^q\left([0,T);X\right)$ and $\psi_2 \in W^{1,q}\left([0,T);X\right)$ induces a decomposition $\psi = \psi_1 + \psi_2$ with $\psi_1 \in L^q\left(\left[(j-1) \frac \delta 2, j \frac \delta 2\right);X\right)$ and $\psi_2 \in W^{1,q}\left(\left[(j-1) \frac \delta 2, j \frac \delta 2\right);X\right)$ for all $j \in \left\{1,\ldots,2N\right\}$. Hence, we can conclude that
\begin{align*}
&\vertii{\psi}_{B^{s,q}_q\left([0,T);X\right)}^q \\
& \quad \le C \sum_{j = 1}^{2N} \int_0^\infty \inf_{\psi = \psi_1 + \psi_2} \left(\tau^{- s q} \vertii{\psi_1}_{L^q\left(\left[(j-1) \frac \delta 2, j \frac \delta 2\right);X\right)}^q + \tau^{(1-s) q} \vertii{\psi_2}_{W^{1,q}\left(\left[(j-1) \frac \delta 2, j \frac \delta 2\right);X\right)}^q\right) \frac{\d\tau}{\tau} \\
& \quad \sim_q \sum_{j = 1}^{2N} \vertii{\psi}_{B^{s,q}_q\left(\left[(j-1) \frac \delta 2, j \frac \delta 2\right);X\right)}^q.
\end{align*}
This implies $\vertii{\phi_\eps - \phi_{\eps'}}_{B^{s,q}_q\left([0,T);X\right)} \to 0$ as $\eps, \eps' \searrow 0$, so that $\phi \in B^{s,q}_q\left([0,T);X\right)$ and \eqref{sum_norms} holds true.
\end{proof}
\bibliography{gess_gnann_stfe_final} 
\bibliographystyle{plain} 

\end{document}